\documentclass{smfart}
\usepackage[english,frenchb]{babel}
\usepackage[T1]{fontenc}
\usepackage[utf8]{inputenc}
\usepackage{amsmath}
\usepackage{array}
\usepackage{amsthm}
\usepackage{amssymb}
\input xy
\xyoption{all}
\usepackage{hyperref}

\newcommand{\A}{{\mathcal{A}}}
\newcommand{\B}{{\mathcal{B}}}
\newcommand{\C}{{\mathcal{C}}}
\newcommand{\D}{{\mathcal{D}}}
\newcommand{\F}{{\mathcal{F}}}
\newcommand{\fct}{{\mathbf{Fct}}}
\newcommand{\pol}{{\mathcal{P}ol}}
\newcommand{\E}{{\mathcal{E}}}

\newcommand{\s}{{\mathcal{S}}}
\newcommand{\I}{{\mathcal{I}}}

\newcommand{\M}{{\mathcal{M}}}

\newcommand{\G}{{\mathcal{G}}}

\newcommand{\col}{{\rm colim}\,}
\newcommand{\U}{{\mathcal{U}}}
\newcommand{\V}{{\mathcal{V}}}
\newcommand{\Md}{\text{-}\mathbf{Mod}}
\newcommand{\Mpol}{\text{-}\mathbf{Pol}}
\newcommand{\md}{\text{-}\mathbf{mod}}

\newcommand{\Si}{\mathfrak{S}}
\newcommand{\FF}{\mathbb{F}}
\newcommand{\Irr}{\mathrm{Irr}}
\newcommand{\lnc}{\textbf{\textsc{lnc}}}
\newcommand{\pfa}{\textbf{\textsc{pfa}}}
\newcommand{\tfc}{\textbf{\textsc{tfa}}}
\newcommand{\ph}{$\mathrm{Hom}$-polynomial}
\newcommand{\phs}{$\mathrm{Hom}$-polynomiaux}
\newcommand{\df}{\mathrm{df}}
\newcommand{\op}{\mathrm{op}}
\newcommand{\GL}{\operatorname{GL}}
\newcommand{\SL}{\operatorname{SL}}

\title{Décompositions à la Steinberg sur une catégorie additive}

\author{Aur\'elien Djament\thanks{CNRS, Univ. Lille, UMR 8524 - Laboratoire Paul Painlevé, F-59000 Lille, France ; aurelien.djament@univ-lille.fr.}, Antoine Touz\'e\thanks{Univ. Lille, CNRS, UMR 8524 - Laboratoire Paul Painlevé, F-59000 Lille, France ; antoine.touze@univ-lille.fr.} et Christine Vespa\thanks{Université de Strasbourg, CNRS, IRMA UMR 7501, F-67000 Strasbourg, France ; vespa@math.unistra.fr.}}

\newtheorem{thi}{Th\'eor\`eme}

\newtheorem{pri}[thi]{Proposition}
\newtheorem{pbi}[thi]{Problème}
\newtheorem{cji}[thi]{Conjecture}

\newtheorem{thm}{Th\'eor\`eme}[section]
\newtheorem{pr}[thm]{Proposition}
\newtheorem{cor}[thm]{Corollaire}
\newtheorem{lm}[thm]{Lemme}

\newtheorem{conj}[thm]{Conjecture}

\theoremstyle{definition}
\newtheorem{defi}[thm]{D\'efinition}
\newtheorem{nota}[thm]{Notation}
\newtheorem{sit}[thm]{Situation}

\theoremstyle{remark}

\newtheorem{rem}[thm]{Remarque}
\newtheorem{ex}[thm]{Exemple}

\begin{document}

\maketitle

\begin{abstract}
Nous donnons une description des foncteurs simples à valeurs de dimensions finies allant d'une petite catégorie additive dans la catégorie des espaces vectoriels sur un corps assez gros (par exemple algébriquement clos). Cette description est analogue à des théorèmes de décomposition tensorielle de Steinberg en théorie des représentations des groupes. Nos résultats reposent sur la notion de foncteur polynomial introduite par Eilenberg et MacLane. Nous en donnons des applications aux représentations des groupes linéaires et aux propriétés de finitude des catégories de foncteurs.
\end{abstract}

{\selectlanguage{english}
{\begin{abstract}
We give a description of simple functors taking finite dimensional values, from a small additive category to the category of vector spaces over a field which is big enough (for example algebraically closed). This description is analogous to Steinberg's tensor product theorems in group representation theory. Our results depend upon the notion of polynomial functor introduced by Eilenberg and MacLane. We give applications to representations of general linear groups and to finiteness properties of functor categories.
\end{abstract}
}}

\noindent
{\em Mots-clefs : } représentations et foncteurs polynomiaux, représentations et foncteurs simples, catégories additives, décomposition tensorielle, présentation finie.

\medskip

\noindent
{\em Classification MSC 2010 : } 18A25, 18E05, 18E15, 20C99, 20G05, 20J99.

\section*{Introduction}

\subsubsection*{Représentations des petites catégories}
Pour $\C$ une catégorie essentiellement petite et $K$ un corps commutatif, on considère la catégorie $\F(\C;K)$ des foncteurs de $\C$ vers les $K$-espaces vectoriels. Cette catégorie abélienne est analogue à celle des représentations $K$-linéaires d'un groupe ou, plus généralement, d'un monoïde (correspondant au cas d'une catégorie $\C$ à un seul objet), si bien qu'elle est également  appelée  catégorie des représentations $K$-linéaires de $\C$.

En topologie algébrique, les foncteurs entre catégories de modules sont notamment reliés aux modules sur l'algèbre de Steenrod \cite{HLS,Ku1,Ku2,Ku3}, à l'homologie de Hochschild topologique (THH) \cite{PiraWald} ou à l'homologie stable des groupes \cite{FFSS,Betley,DV,DjaR}.
En théorie des représentations, les foncteurs \textit{additifs} d'une catégorie additive vers une catégorie abélienne sont depuis longtemps utilisés de façon systématique. Ils apparaissent par exemple dans l'étude des représentations indécomposables des algèbres de dimension finie, comme dans les travaux d'Auslander \cite{Aus-rec,Aus82} ou de Gabriel \cite{Gab75}, ou dans l'étude des propriétés de pureté pour les modules sur des anneaux généraux \cite{Kra}. Plus récemment, l'étude de la catégorie $\F(\mathrm{FI};K)$, où $\mathrm{FI}$ est la catégorie des ensembles finis avec injections, liée aux représentations des groupes symétriques, s'est avérée fructueuse -- voir par exemple \cite{CEF, FarbICM}.

Ainsi, il existe de profondes interactions entre la théorie des représentations des groupes ou des monoïdes et celle des petites catégories. Une question centrale dans l'étude de la catégorie $\F(\C;K)$ est la suivante :

\begin{pbi}\label{pb-descrisimples} Comment décrire les objets simples de la catégorie abélienne $\F(\C;K)$ ?
\end{pbi}

L'objectif premier de cet article est de répondre à cette question pour une catégorie source \emph{additive}, par exemple, la catégorie $\mathbf{P}(A)$ des modules à gauche projectifs de type fini sur un anneau $A$.

La proposition~\ref{intro-prf} ci-après (page~\pageref{intro-prf}) rappelle les liens étroits existant entre foncteurs simples de $\F(\C;K)$ et représentations $K$-linéaires simples des groupes d'automorphismes des objets de $\C$, lorsque les ensembles de morphismes entre deux objets donnés de $\C$ sont \emph{finis}.
En l'absence d'hypothèse de finitude, de grandes difficultés structurelles, liées à des phénomènes sauvages, apparaissent pour les représentations des petites catégories, comme pour celles des groupes ou des monoïdes. Ainsi, la grande majorité de cet article ne concerne que des foncteurs de $\F(\C;K)$ dont les valeurs sont des $K$-espaces vectoriels \emph{de dimensions finies}. Nous noterons $\F^{\df}(\C;K)$ la sous-catégorie pleine de ces foncteurs. Lorsque les ensembles de morphismes sont finis dans $\C$, tous les foncteurs simples (ou, plus généralement, de type fini) de $\F(\C;K)$ appartiennent à $\F^{\df}(\C;K)$, mais c'est loin d'être le cas en général.

Dans toute la suite de l'introduction, \textbf{$\A$ désigne une catégorie additive essentiellement petite et $K$ un corps commutatif, algébriquement clos}. Cette dernière hypothèse sur $K$ est souvent affaiblie pour les résultats donnés dans le corps de l'article, mais elle permet de simplifier plusieurs énoncés de cette introduction.

 Nos deux résultats principaux sur la structure des foncteurs simples de $\F(\A;K)$ sont des théorèmes de décomposition tensorielle. Ils constituent des analogues de décompositions tensorielles établies par R. Steinberg, dont nous rappelons maintenant les énoncés.

\subsubsection*{Deux théorèmes classiques de décomposition tensorielle}
  
 Dans \cite{RSt}, Steinberg démontre qu'une représentation complexe $M$ de dimension finie de $\SL_n(\mathbb{Z})$ ($n\ge 3$) est simple si et seulement si elle se décompose sous la forme:
\begin{align}M\simeq S\otimes T\label{eqn-2}\end{align}
où $S$ est une représentation simple factorisant par un quotient fini $\SL_n(\mathbb{Z}/m\mathbb{Z})$ de $\SL_n(\mathbb{Z})$ et $T$ est une représentation polynomiale simple. Une telle décomposition tensorielle est alors unique.  Des cas particuliers de cette décomposition tensorielle avaient été obtenus préalablement par Bass-Milnor-Serre \cite{BMS} et Serre \cite{Serre} ; les théorèmes de super-rigidité de Margulis \cite[chap.~VII, §\,5]{Marg} donnent des résultats nettement plus généraux en ce sens, à l'aide de méthodes beaucoup plus complexes.

Un autre théorème dû à Steinberg \cite{St-TPT} concerne les groupes de Chevalley $G(\FF_q)$ comme $\GL_n(\FF_q)$ ou $\SL_n(\FF_q)$, où $q=p^r$ avec $p$ premier. Plus précisément, si $K$ contient $\FF_q$, une représentation $K$-linéaire $M$ de $G(\FF_q)$ est simple si et seulement si elle se décompose comme un produit tensoriel 
\begin{align}M\simeq M_0^{\{0\}}\otimes M_1^{\{1\}}\otimes \dots\otimes M_{r-1}^{\{r-1\}}\label{eqn-1}\end{align}
où les $M_i$ sont des représentations simples $p$-restreintes du groupe algébrique $G(K)$, et l'exposant $^{\{i\}}$ indique que l'on a restreint la représentation $M_i$ le long 
du morphisme de groupes $G(\FF_q)\to G(K)$ induit par le plongement $\FF_q\to K$, $x\mapsto x^{p^i}$. Dans le cas de $\SL_n(\FF_q)$ les représentations $M_i$ sont uniquement déterminées, dans celui de $\GL_n(\FF_q)$ elles sont uniquement déterminées à tensorisation par des puissances du déterminant près.

\subsubsection*{Décomposition tensorielle fonctorielle globale}
 
Notre premier résultat est un analogue fonctoriel de la décomposition \eqref{eqn-2}. Il ramène la compréhension des simples de $\F^{\df}(\A;K)$ à celle de deux classes fondamentales de foncteurs. 

La première, celle des  
\emph{foncteurs polynomiaux}, fut  introduite dès les années 1950 par Eilenberg et MacLane \cite{EML} à des fins de topologie algébrique ; elle est reliée à la notion de représentation polynomiale des groupes linéaires. Les foncteurs polynomiaux constituent une généralisation naturelle des foncteurs additifs (les foncteurs polynomiaux de degré $1$ sont, à un terme constant près, les foncteurs additifs). Nous renvoyons le lecteur à la section~\ref{paragraphe-pol} pour plus de détails sur cette notion classique.
 
La deuxième classe fondamentale de foncteurs, celle des foncteurs \emph{antipolynomiaux}, ne semble pas avoir reçu d'attention particulière jusqu'à présent. Afin de la définir, nous introduisons à la définition \ref{def-CategKtriv} la notion de catégorie \emph{$K$-triviale} : c'est une petite catégorie additive $\B$ telle que, pour tous objets $x$ et $y$ de $\B$, le groupe abélien $\B(x,y)$ est fini et d'ordre inversible dans $K$. Nous dirons qu'un foncteur de $\F(\A;K)$ est \emph{antipolynomial} s'il se factorise à travers un foncteur additif $\A\to\B$ dont le but est une catégorie $K$-triviale (voir la définition \ref{def-antipol}). 

Foncteurs polynomiaux et foncteurs antipolynomiaux forment deux classes de foncteurs essentiellement disjointes : par la proposition~\ref{prel-antipol}, seuls les foncteurs constants sont simultanément polynomiaux et antipolynomiaux. 
Le terme <<~antipolynomial~>> se justifie par le fait que cette classe de foncteurs possède des propriétés très différentes des foncteurs polynomiaux. Par exemple, les simples antipolynomiaux sont automatiquement à valeurs de dimensions finies (voir la proposition \ref{prel-antipol}) et leurs fonctions de dimensions sont -- au moins conjecturalement -- des polynômes d'exponentielles (voir la proposition~\ref{pr-kud} et la conjecture~\ref{conj-dim}), alors que les fonctions de dimensions des foncteurs polynomiaux simples sont des polynômes.   

Le résultat suivant montre qu'une large classe de foncteurs de $\F^{\df}(\A;K)$ est construite à partir des foncteurs polynomiaux et antipolynomaiux:

\begin{thi}[Corollaire~\ref{cor-tftcf}]\label{intro-thm-gl} Soit $F$ un foncteur de longueur finie de $\F^{\df}(\A;K)$. Alors il existe un unique bifoncteur $B$ de $\F^{\df}(\A\times\A;K)$ qui est polynomial par rapport à la première variable et antipolynomial par rapport à la seconde tel que $F$ soit isomorphe à la composée de $B$ et de la diagonale $\A\to\A\times\A$.
\end{thi}

On en déduit notre premier théorème fondamental sur la structure des foncteurs simples. Dans cet énoncé, le produit tensoriel de foncteurs est pris sur $K$ et se calcule au but : $(S\otimes T)(a)=S(a)\otimes_K T(a)$.

\begin{thi}[Corollaire~\ref{cor-tens-St1}] \label{intro-thm1}
Un foncteur $F$ de $\F^{\df}(\A;K)$ est simple si et seulement s'il est isomorphe à un produit tensoriel 
$$F\simeq S\otimes T$$ 
où $S$ est un foncteur simple 
polynomial de $\F^{\df}(\A;K)$ et $T$ un foncteur simple antipolynomial de $\F(\A;K)$.

De plus, $S$ et $T$ sont uniquement déterminés par $F$, à isomorphisme près.
\end{thi}

Bien que la décomposition tensorielle \eqref{eqn-2} et le théorème \ref{intro-thm1} soient formellement similaires, les relations entre ces deux énoncés ne sont pas complètement comprises (voir les remarques~\ref{rq-disc-Stein} et~\ref{rem-comparaison-SL-GL} pour plus de détails). 

Nos résultats montrent en particulier que si l'une des deux classes fondamentales de foncteurs de $\F^{\df}(\A;K)$ (foncteurs polynomiaux ou antipolynomiaux) est réduite aux foncteurs constants, alors tous les foncteurs vérifiant des hypothèses raisonnables de finitude appartiennent à l'autre classe fondamentale. Il existe également des situations d'intérêt où les deux classes fondamentales sont simultanément réduites aux foncteurs constants -- ceci se produit par exemple lorsque $\A$ est une catégorie linéaire sur un corps infini de caractéristique différente de celle de $K$. Dans ce cas, on peut déduire de nos résultats que tous les foncteurs de $\F^{\df}(\A;K)$ sont constants, voir le corollaire~\ref{cor-dfcst}.

\subsubsection*{Décomposition tensorielle des foncteurs polynomiaux simples}  Rappelons que les représentations polynomiales simples des groupes $\GL_n(K)$ s'obtiennent par des constructions fonctorielles, c'est-à-dire qu'elles sont de la forme $F(K^n)$ où $F$ est un endofoncteur simple des $K$-espaces vectoriels. Nous appelons \emph{foncteurs élémentaires} les foncteurs $F$ qui sont des quotients simples des puissances tensorielles. Nous renvoyons à la section~\ref{sect-stein-equi} et l'appendice~\ref{app-elt} pour des descriptions plus explicites, et mentionnons seulement que si $K$ est de caractéristique nulle, les foncteurs élémentaires sont les foncteurs de Schur classiques \cite[Chap.~6.1]{FH}.

Notre deuxième théorème principal de structure des foncteurs simples est un analogue fonctoriel de la décomposition tensorielle \eqref{eqn-1}. Dans ce qui suit, on note $\pi^* F$ la composée $F\circ\pi$ de deux foncteurs, et on prend pour convention que tous les facteurs d'un produit tensoriel infini sont égaux à $K$, sauf un nombre fini.

\begin{thi}[Théorème~\ref{steinberg-pol-gal}] \label{intro-thm2}
Un foncteur polynomial $F$ de $\F^{\df}(\A;K)$ est simple si et seulement s'il est isomorphe à un produit tensoriel
\[F\simeq\bigotimes_\pi\pi^*\mathrm{E}_\pi\]
indexé par un ensemble complet de représentants des classes d'isomorphisme de foncteurs additifs simples $\pi$ de $\F^{\df}(\A;K)$, où les $\mathrm{E}_\pi$ sont des endofoncteurs élémentaires des $K$-espaces vectoriels. De plus, les $\mathrm{E}_\pi$ sont uniques à isomorphisme près.
\end{thi} 

Pour $\A=\mathbf{P}(A)$, la catégorie des $A$-modules à gauche projectifs de type fini,
les foncteurs additifs de $\F(\A;K)$ s'identifient aux $(K,A)$-bimodules : le théorème précédent fournit une classification des foncteurs polynomiaux simples de $\F^{\df}(\mathbf{P}(A);K)$ à partir de celle des $(K,A)$-bimodules simples de dimension finie sur $K$.

Les théorèmes~\ref{intro-thm1} et~\ref{intro-thm2} sont valables pour toutes les petites catégories additives $\A$, sans restriction. Nous illustrons ces théorèmes dans trois cas particuliers.
\begin{enumerate}
\item Si $\A=\mathbf{P}(\mathbb{Z})$, le foncteur $\pi$ d'extension des scalaires est le seul foncteur additif simple de $\A$ vers les $K$-espaces vectoriels. Les foncteurs simples à valeurs de dimensions finies sont donc les produits tensoriels $S\otimes \pi^*E$, où $S$ est un foncteur simple factorisant par une catégorie $\mathbf{P}(\mathbb{Z}/m\mathbb{Z})$ avec $m$ inversible dans $K$, et $E$ est un foncteur élémentaire.
\item En prenant $\A=\mathbf{P}(\FF_q)$ et $K=\FF_q$, on retrouve un théorème de Kuhn (voir l'exemple~\ref{ex-kuhn}).   
\item Pour $\A=A\md$, la catégorie des modules de type fini sur une $K$-algèbre $A$ de dimension finie, nos théorèmes montrent que les foncteurs simples à valeurs de dimensions finies sont les produits tensoriels de la forme $\pi_1^*E_1\otimes\dots\otimes \pi^*_nE_n$ où les $E_i$ sont des foncteurs élémentaires et les $\pi_i:A\md\to K\md$ sont des foncteurs additifs simples. Ces foncteurs additifs simples ont été décrits par Auslander \cite{Aus82}. Nos théorèmes permettent donc d'étendre la classification d'Auslander aux foncteurs  simples non additifs.
\end{enumerate}

\subsubsection*{Structure des foncteurs simples antipolynomiaux}

 Au vu de leur définition, l'étude des foncteurs simples antipolynomiaux se ramène à la question suivante :
 
\begin{pbi}\label{intro-pb-antipol} Que dire de la structure des foncteurs simples de $\F(\A;K)$ lorsque $\A$ est une catégorie $K$-triviale ?
\end{pbi}

Une réduction élémentaire réside dans la propriété suivante, où l'on note $_{(p)}V$ la composante $p$-primaire d'un groupe abélien fini $V$, pour un nombre premier $p$. On désigne par $_{(p)}\A$ la sous-catégorie de $\A$ avec les mêmes objets que $\A$ et dont les groupes abéliens de morphismes sont donnés par $(_{(p)}\A)(x,y)={_{(p)}(\A(x,y))}$. On a un foncteur canonique (qui est l'identité sur les objets et qui envoie un morphisme sur sa composante $p$-primaire) $\pi_p:\A\to {_{(p)}\A}$.

\begin{pri}[Proposition~\ref{dec-prim}]
Soit $S$ un foncteur de $\F(\A;K)$. Si les groupes abéliens de morphismes $\A(x,y)$ sont finis pour tous objets $x$ et $y$ de $\A$, alors $S$ est simple si et seulement s'il existe, pour tout nombre premier $p$, un simple $S_p$ de $\F({_{(p)}}\A;K)$, de sorte que
\[S\simeq \bigotimes_p \pi_p^*S_p\]
(le produit tensoriel étant indexé par l'ensemble des nombres premiers $p$).

De plus, dans une telle décomposition tensorielle, les $S_p$ sont uniquement déterminés à isomorphisme près.
\end{pri}

On peut donc se ramener, dans le problème~\ref{intro-pb-antipol}, au cas où $\A$ est $p$-primaire pour un nombre premier $p$ distinct de la caractéristique de $K$. L'archétype de cette situation est le cas où $\A=\mathbf{P}(A)$ pour un $p$-anneau fini $A$, c'est-à-dire un anneau fini de caractéristique une puissance de $p$.

Hélas, même dans un cas aussi simple que $A=\mathbb{Z}/p^2\mathbb{Z}$ et $K=\mathbb{C}$, la classification des simples semble hors d'atteinte. En effet, de manière générale, sans hypothèse d'additivité, la finitude des ensembles de morphismes de la source d'une catégorie de foncteurs permet de relier précisément le problème~\ref{pb-descrisimples} à un problème de représentations de groupes finis, grâce à la proposition suivante, dans laquelle $\Irr(\E)$ désigne l'ensemble des classes d'isomorphisme d'objets simples d'une catégorie abélienne avec générateur $\E$ et $\mathrm{Iso}(\C)$ l'ensemble des classes d'isomorphisme d'objets d'une petite catégorie $\C$.

\begin{pri}[Proposition~\ref{simples-sfin}]\label{intro-prf} Soit $\C$ une petite catégorie dont les idempotents se scindent, et telle que $\C(x,x)$ soit un ensemble fini pour tout objet $x$. On a une bijection
\[\bigsqcup_{[x]\in\mathrm{Iso}(\C)}\Irr(K[\mathrm{Aut}_\C(x)]\Md)\xrightarrow[]{\simeq}\Irr(\F(\C;K))\;.\]
\end{pri}

On obtient ainsi une correspondance entre les classes d'isomorphisme de foncteurs simples de $\F(\mathbf{P}(\mathbb{Z}/p^r\mathbb{Z});K)$ (où $p$ est un nombre premier et $r>0$ un entier) et les classes d'isomorphisme de représentations $K$-linéaires simples des différents groupes $\GL_n(\mathbb{Z}/p^r\mathbb{Z})$. Mais lorsque $K$ est de caractéristique $0$, la description de ces représentations simples est un problème sauvage pour $r>1$ (cf. \cite[§\,4]{VS-survol}).

La classification des simples étant hors d'atteinte, on peut se tourner vers la recherche de propriétés plus globales et qualitatives des foncteurs simples antipolynomiaux. À cette fin, nous formulons dans l'article la conjecture suivante, qui propose des éléments de réponse au problème~\ref{intro-pb-antipol}. Nous démontrons cette conjecture dans le cas où l'anneau $A$ est semi-simple (proposition~\ref{pr-kud}) grâce à un résultat de Kuhn \cite{Ku-adv}.
\begin{cji}[Conjecture~\ref{conj-dim}]\label{intro-conj-dim}
Soit $A$ un $p$-anneau fini. On suppose $\mathrm{car}(K)\neq p$. Alors pour tout foncteur simple $F$ de $\F(\mathbf{P}(A);K)$, il existe une fonction polynomiale $f : \mathbb{Z}\to\mathbb{Z}$ telle que $\forall n\in\mathbb{N}\quad\dim_K F(A^n)=f(p^n)$.
\end{cji}

Dans la suite de cette introduction, nous décrivons quelques applications de nos résultats sur la structure des foncteurs simples.
On peut classer ces applications en deux familles.

D'une part, l'évaluation sur $A^n$ d'un foncteur de  $\F(\mathbf{P}(A);K)$ est naturellement une représentation du monoïde multiplicatif de matrices $\M_n(A)$. Le foncteur
d'évaluation $\F(\mathbf{P}(A);K)\to  K[\M_n(A)]\Md$ admet une section, le prolongement intermédiaire, qui se comporte bien vis-à-vis des propriétés de simplicité. On peut donc (sous des hypothèses techniques adéquates) transférer nos résultats et obtenir des informations sur la structure des représentations simples de $\M_n(A)$. Par restriction, nous obtenons aussi des résultats sur les représentations simples de $\GL_n(A)$ et $\SL_n(A)$.

D'autre part, nos résultats de décomposition permettent de ramener l'étude de propriétés de finitude des catégories de foncteurs $\F(\A;K)$ et $\F^{\df}(\A;K)$ à des cas particuliers plus accessibles. Nous pouvons ensuite étudier ces cas particuliers par un examen direct ou à l'aide de la littérature existante, et ainsi obtenir des résultats de finitude généraux sur les catégories de foncteurs.

\subsubsection*{Représentations simples de $\GL_n(A)$ ($A$ anneau fini)}

Dans le cas d'un  $p$-anneau fini $A$, où $p$ est la caractéristique de $K$, nous pouvons obtenir à partir des théorèmes~\ref{intro-thm1} et~\ref{intro-thm2} un théorème de décomposition tensorielle pour les groupes finis $\GL_n(A)$.
(Le théorème \ref{classif-St} est légèrement plus précis sur la propriété d'unicité.)  
\begin{thi}[Théorème \ref{classif-St}]\label{thm-intro-3}
Si $K$ est de caractéristique $p\ne 0$, soient $A$ un $p$-anneau fini, et $\{B_1,\dots,B_m\}$ un système complet de $(K,A)$-bimodules simples.

Soit $M$ une représentation $K$-linéaire de $\GL_n(A)$. Alors $M$ est simple si et seulement si elle peut s'écrire comme un produit tensoriel 
$$M\simeq M_1^{[B_1]}\otimes\dots\otimes M_m^{[B_m]}$$
où les $M_i$ sont des représentations polynomiales simples $p$-restreintes des groupes $\GL_{n d_i}(K)$, où $d_i=\dim_K B_i$ et l'exposant $^{[B_i]}$ signifie que l'on restreint la représentation $M_i$ le long du morphisme de groupes $\GL_n(A)\to \GL_{nd_i}(K)$ induit par le morphisme $A\to \mathcal{M}_{d_i}(K)$ correspondant au bimodule $B_i$. De plus, les $M_i$ sont uniquement déterminées à tensorisation par une puissance du déterminant près.
\end{thi}

Le théorème~\ref{thm-intro-3} constitue une généralisation aux anneaux finis arbitraires du théorème du produit tensoriel de Steinberg (décomposition tensorielle \eqref{eqn-1}), auquel il se réduit lorsque $A$ est un corps. Les méthodes de notre article donnent donc une nouvelle démonstration de ce résultat fondamental.

\subsubsection*{Structure des représentations simples EML-polynomiales}

Nos théorèmes de décomposition tensorielle des foncteurs simples ne s’appliquent qu’aux foncteurs à valeurs de dimensions finies. Pour pouvoir transférer ces théorèmes à une représentation simple $M$ de $G_n(A) =\M_n(A)$, $\GL_n(A)$ ou $\SL_n(A)$, il faut donc s’assurer que le prolongement intermédiaire de $M$, c’est-à-dire le foncteur simple $F$ de $\F(\mathbf{P}(A);K)$ tel que $F(A^n)\simeq M$, soit à valeurs de dimensions finies. Si $A$ est un anneau fini, cette condition de finitude est automatiquement vérifiée, mais si $A$ est infini, il peut exister des foncteurs simples ne possédant qu’une seule valeur non nulle de dimension finie -- nous en donnons des exemples à l’appendice~\ref{apa-1}. Nous démontrons au lemme~\ref{lm-vdf} que pour un anneau \emph{commutatif} $A$, il est possible de contrôler ce phénomène de finitude sur le prolongement intermédiaire en imposant une condition de régularité sur le morphisme d’action $\rho : G_n(A)\to M$.

Ceci nous conduit à introduire la notion de \textit{représentation polynomiale à la Eilenberg-MacLane}, ou plus brièvement \emph{représentation EML-polynomiale}, dans la définition~\ref{def-rep-pol}. C'est une version généralisée de la notion classique de représentation polynomiale \cite{Green}, dans laquelle on demande que $\rho$ soit une fonction polynomiale à la Eilenberg-MacLane (par opposition à un polynôme). Les deux propriétés suivantes donnent un grand nombre d'exemples.
\begin{enumerate}
\item Toute représentation polynomiale de $\GL_n(K)$ au sens classique \cite{Green} est EML-polynomiale.
\item La notion de représentation EML-polynomiale est stable par produit tensoriel et par restriction le long des morphismes d'anneaux.
\end{enumerate}

Si $M$ est une représentation de $G_n(K)$ et $\phi : A\to K$ un morphisme d'anneaux, on note $M^{[\phi]}$ la représentation de $G_n(A)$ obtenue par restriction de $M$ le long du morphisme de groupes $G_n(\phi) : G_n(A)\to G_n(K)$. Cela concorde avec la notation du théorème~\ref{thm-intro-3} si l'on assimile $\phi$ au $(K,A)$-bimodule $K$ sur lequel $A$ agit via $\phi$.

Nous montrons le théorème de décomposition suivant. 

\begin{thi}[Théorème~\ref{th-SL}]\label{intro-thm4} Soient $A$ un anneau commutatif et $M$ une représentation $K$-linéaire de dimension finie de $\SL_n(A)$. On suppose que $M$ est EML-polynomiale.
Alors $M$ est simple si et seulement si elle peut s'écrire comme un produit tensoriel
\[M\simeq M_1^{[\phi_1]}\otimes\dots \otimes M_m^{[\phi_m]}\]
où les $M_i$ sont des représentations polynomiales (au sens classique) simples de $\SL_n(K)$, qui sont $p$-restreintes si $K$ est de caractéristique $p>0$, et les $\phi_i:A\to K$ sont des morphismes d'anneaux deux à deux distincts. 

De plus, les morphismes $\phi_i$ et les représentations $M_i$ apparaissant dans une telle décomposition sont uniquement déterminés.
\end{thi}

Ce théorème est de même nature que les résultats de Borel et Tits \cite[§\,10]{Borel-Tits}, mais les hypothèses (et la méthode de démonstration) sont différentes : ces auteurs n'ont besoin d'aucune hypothèse apparentée à la polynomialité à la Eilenberg-MacLane mais ils ne traitent de groupes linéaires que sur des \emph{corps} (commutatifs), tandis que, dans notre théorème~\ref{intro-thm4} l'anneau commutatif $A$ est arbitraire. Notre théorème peut tomber en défaut si l'on omet l'hypothèse de polynomialité sur les représentations. Ce phénomène est lié à la $K$-théorie de $A$, voir la remarque \ref{rem-Kth}.

Nous établissons également des analogues du théorème 10 pour les représentations de $\M_n(A)$ et $\GL_n(A)$ -- avec la nuance importante que dans le cas des groupes linéaires l'unicité peut tomber en défaut, voir l'exemple \ref{ex-pas-unique}. 

Même pour $n=1$, la variante du théorème~\ref{intro-thm4} que nous établissons (proposition~\ref{pia1} ci-après) pour les monoïdes $\M_n(A)$ n'est pas totalement évidente : nous n'en connaissons pas de démonstration directe. Ce résultat, comme plus généralement la notion de représentation EML-polynomiale et le théorème~\ref{intro-thm4}, est de nature \emph{arithmétique} : les hypothèses combinent une propriété liée purement à la structure additive des anneaux (la polynomialité) à une propriété, a priori indépendante, liée purement à leur structure multiplicative (la préservation du produit). La conclusion peut se lire comme un résultat de \emph{rigidité} : les objets s'obtiennent en fait de manière très contrainte à partir de morphismes d'anneaux, c'est-à-dire respectant simultanément les structures additive et multiplicative.

\begin{pri}[Proposition~\ref{cor-arith1}]\label{pia1}
Soient $A$ un anneau commutatif et $\zeta : A\to K$ une fonction EML-polynomiale de degré $d>0$, et multiplicative au sens où $\zeta(xy)=\zeta(x)\zeta(y)$ pour tous $x$, $y\in A$. Alors $\zeta$ peut s'écrire comme le produit (point par point) de $d$ morphismes d'anneaux de $A$ dans $K$, qui sont uniques à l'ordre des facteurs près.
\end{pri}

\subsubsection*{Application aux propriétés de finitude des catégories de foncteurs}

Nous appliquons également nos théorèmes principaux à l'étude de propriétés de finitude des catégories $\F(\A;K)$ (dans cette partie, le corps $K$ peut être quelconque). Ces propriétés sont banales lorsque cette catégorie est localement noethérienne, mais elles valent sous des hypothèses bien plus faibles et faciles à vérifier. On rappelle qu'une catégorie abélienne est \emph{localement noethérienne} si elle possède un ensemble de générateurs noethériens, et qu'un objet est dit \emph{fini} s'il est à la fois noethérien et artinien. Le résultat suivant utilise la notion classique \cite{Mi72,Street} d'idéal d'une catégorie additive, qui généralise la notion d'idéal bilatère d'un anneau. Nous disons  qu'un idéal $\I$ de $\A$ est \emph{$K$-cotrivial} si la catégorie additive quotient $\A/\I$ est $K$-triviale (voir la définition \ref{df-ideaux}). 

\begin{thi}[Théorème \ref{th-ptens-fini}]
 Si, pour tout idéal $K$-cotrivial $\I$ de $\A$, les catégories $\F(\A/\I;K)$ et $\F((\A/\I)^{op};K)$ sont localement noethériennes, alors  la classe des foncteurs finis de $\F^{\df}(\A;K)$ est stable par produit tensoriel.
\end{thi}

 L'hypothèse de ce théorème est en général difficile à vérifier, mais la finitude des groupes abéliens de morphismes entre deux objets de $\A/\I$ donne dans certains cas accès à des méthodes combinatoires fructueuses. En particulier, cette hypothèse est satisfaite pour $\A=\mathbf{P}(A)$, où $A$ est un anneau \emph{quelconque} (y compris non noethérien). Cela provient de profonds résultats de noethérianité établis par Putman-Sam-Snowden \cite{PSam,SamSn}.
 
 Nous nous intéressons aussi à la propriété de {\em présentation finie} pour les foncteurs finis. Cette question apparaît par exemple de façon centrale (pour les foncteurs additifs) dans les travaux d'Auslander \cite[§\,3]{Aus82}. Une conséquence de notre critère général de présentation finie pour les simples (aux hypothèses plus techniques -- voir le théorème~\ref{thapf}) est le résultat suivant.
 
 \begin{thi}[Corollaire~\ref{cor-pfdf}]\label{thipfi} Supposons que les anneaux $A$ et $A\otimes_\mathbb{Z}K$ sont noethériens à droite. Alors tout foncteur fini de $\F^{\df}(\mathbf{P}(A);K)$ est de présentation finie dans $\F(\mathbf{P}(A);K)$.
 \end{thi}
 
 Ce résultat ne découle pas de propriétés formelles générales car la catégorie $\F(\mathbf{P}(A);K)$ n'est localement noethérienne que si $A$ est \emph{fini} (proposition~\ref{rq-anninf}).
 
\subsubsection*{Foncteurs simples ne prenant pas des valeurs de dimensions finies}

Les théorèmes~\ref{intro-thm1} et~\ref{intro-thm2} concernent les foncteurs simples prenant des valeurs de dimensions finies. Nous discutons dans l'appendice~\ref{apa} quelques questions liées aux foncteurs simples de $\F(\A;K)$ n'appartenant pas à $\F^{\df}(\A;K)$. Notons que l'hypothèse de valeurs de dimensions finies pourrait être remplacée, quitte à agrandir le corps commutatif $K$, par celle d'\textit{endofinitude centrale}. Cette notion, plus intrinsèque, et variante de la notion d'endofinitude, classique pour les modules (cf. par exemple \cite[§\,6.3 et 6.4]{Kra}), a été étudiée dans le cadre des représentations des groupes par différents auteurs, dont Wehrfritz \cite{W83}. Pour un foncteur simple $S$ de $\F(\C;K)$, l'endofinitude centrale signifie que les valeurs de $S$ sont de dimensions finies sur le centre de son corps d'endomorphismes.

Dans la section~\ref{apa-1}, nous donnons des exemples explicites de foncteurs simples de $\F(\mathbf{P}(L);K)$ (pour $L$ un corps commutatif infini) prenant une seule valeur non nulle de dimension finie. 

Nous montrons également, dans la section~\ref{apa-2}, qu'il existe souvent des foncteurs simples dont toutes les valeurs non nulles sont de dimensions infinies, par exemple dans  $\F(\mathbf{P}(A);K)$, où $A$ est un anneau commutatif de caractéristique nulle. Une description complète de ces foncteurs paraît toutefois hors de portée.

\paragraph*{Remerciements.} 
Le premier auteur remercie Ofer Gabber pour l'entretien qu'il lui a accordé sur des questions arithmétiques reliées à ce travail.

Les auteurs remercient vivement le rapporteur anonyme dont la relecture attentive a permis de substantielles améliorations de la présentation de la première version de cet article. Ils ont bénéficié du soutien partiel de l'Agence Nationale de la Recherche, via le projet ANR ChroK [ANR-16-CE40-0003] et le labex CEMPI [ANR-11-LABX-0007-01]. Ils ne soutiennent pas pour autant le principe de l’ANR, dont ils revendiquent la restitution des moyens aux laboratoires sous forme de crédits récurrents.

\tableofcontents

\paragraph*{Conventions.}
Dans tout l'article, on désigne par :
\begin{itemize}
\item[$\bullet$]$\A$ une catégorie additive essentiellement petite ;
\item[$\bullet$] $\E$ une catégorie de Grothendieck ;
\item[$\bullet$] $K$ un corps commutatif ;
\item[$\bullet$] $K[\M]$ la $K$-algèbre d'un monoïde (ou d'un groupe) $\M$, de sorte que les $K[\M]$-modules s'identifient aux représentations $K$-linéaires de $\M$ ;
\item[$\bullet$] $\mathbf{Ab}$ la catégorie des groupes abéliens et, pour un anneau $A$, $A\Md$ la catégorie des $A$-modules à gauche ;
\item[$\bullet$] $\Si_d$, pour un entier $d\ge 0$, le groupe symétrique des permutations de l'ensemble à $d$ éléments $\{1,\dots,d\}$.
\end{itemize}
Sauf mention explicite du contraire, les produits tensoriels de base non spécifiée sont pris sur $K$. Tous les anneaux sont associatifs et unitaires.

\part{Préliminaires}\label{part-prelim}

\section{Rappels sur les catégories de foncteurs}\label{paragraphe-fonct}

\subsection{Catégories de Grothendieck}\label{subsec-Groth}
La plupart des catégories abéliennes apparaissant dans cet article seront des \emph{catégories de Grothendieck} à l'instar des catégories considérées dans \cite{Gabriel}, c'est-à-dire des catégories abéliennes avec sommes arbitraires, colimites filtrantes exactes et un générateur. Le générateur assure en particulier que les sous-objets d'un objet donné forment un ensemble. La catégorie $A\Md$ des modules à gauche sur un anneau $A$ -- en particulier la catégorie $\mathbf{Ab}$ des groupes abéliens -- est un exemple de catégorie de Grothendieck.

Un objet $s$ d'une catégorie de Grothendieck $\E$ est \emph{simple} s'il est non nul et n'admet aucun sous-objet non trivial. L'anneau des endomorphismes $\mathrm{End}_\E(s)$ est alors un corps (généralement non commutatif). Les classes d'isomorphisme d'objets simples de $\E$ forment un ensemble, qu'on notera $\Irr(\E)$, comme abréviation d'\emph{irréductible}, synonyme courant de \emph{simple} pour les représentations linéaires des groupes. Un objet de $\E$ est dit \emph{semi-simple} s'il est somme directe d'une famille d'objets simples de $\E$. Il est dit \emph{semi-simple isotypique de type $s$} s'il est somme directe de copies d'un objet simple $s$. Les propriétés des objets semi-simples de $\E$ sont analogues à celles des objets semi-simples dans les catégories de modules. Pour plus de détails nous renvoyons à \cite[§\,4]{Bki},
qui traite de catégories de modules, mais dont les démonstrations valent en fait dans toute catégorie de Grothendieck.

Un objet $x$ de $\E$ est \emph{de type fini} si toute famille de sous-objets de $x$, filtrante croissante pour l'inclusion et de réunion $x$, contient $x$. De manière équivalente, le foncteur $\E(x,-)$ commute aux colimites filtrantes de monomorphismes. Un objet de type fini non nul possède toujours un quotient simple. Dualement, un objet $x$ est \emph{de type cofini} si toute famille de sous-objets, filtrante décroissante pour l'inclusion et d'intersection nulle, contient l'objet nul. Un objet de type cofini non nul possède toujours un sous-objet simple.

\subsection{Catégories de foncteurs}

Soient $\C$ et $\D$ deux catégories, avec $\C$ petite. Ici, nous nommons \emph{petite} toute catégorie dont les classes d'isomorphisme d'objets forment un ensemble (on supposera \emph{toujours} que les morphismes entre deux objets donnés d'une catégorie forment un ensemble), c'est-à-dire une catégorie nommée usuellement \emph{essentiellement petite}. On note $\fct(\C,\D)$ la catégorie des foncteurs\, de $\C$ vers $\D$, les morphismes étant les transformations naturelles. On note $\Phi^*:\fct(\C',\D)\to\fct(\C,\D)$ le foncteur de précomposition par un foncteur $\Phi:\C\to\C'$ entre petites catégories. Si $\Phi$ est une équivalence, alors $\Phi^*$ est aussi une équivalence. En prenant pour $\C'$ un squelette de la catégorie $\C$, cela montre que l'abus de terminologie entre \emph{petite} et \emph{essentiellement petite} pour la source est anodin.

Si $\D$ possède des limites (resp. des colimites), il en est de même pour $\fct(\C,\D)$, où elles se calculent au {\em but}. 
En particulier, si $\D$ est abélienne (resp. avec colimites filtrantes exactes), il en est de même pour $\fct(\C,\D)$. 
Pour une catégorie $\D$ avec sommes arbitraires, $M$ un objet de $\D$ et $E$ un ensemble, on note $M[E]$ la somme de copies de $M$ indexées par $E$. Une variation classique autour du lemme de Yoneda donne une bijection:
\[\fct(\C,\D)(M[\C(t,-)],F)\simeq\D(M,F(t))\]
naturelle en les objets $M$, $t$ et $F$ de $\D$, $\C$ et $\fct(\C,\D)$ respectivement. 
Il en résulte que si $M$ est un objet générateur (resp. projectif) de $\D$, alors les foncteurs $M[\C(t,-)]$ engendrent (resp. sont projectifs dans) $\fct(\C,\D)$ lorsque $t$ parcourt un squelette de $\C$. Dualement, si $\D$ possède des produits, on note $M^{E}$ le produit de copies de $M$ indexées par $E$, et on a une bijection naturelle en $M$, $t$ et $F$:
\begin{equation}\label{eq-injy}\fct(\C,\D)(F,M^{\C(-,t)})\simeq\D(F(t),M)\;.
\end{equation}
En particulier, si $M$ est un cogénérateur injectif de $\D$, alors les foncteurs $M^{\C(-,t)}$ forment un ensemble de cogénérateurs injectifs de $\fct(\C,\D)$.
Ces considérations impliquent la proposition classique suivante.
\begin{pr}
Si $\D$ est une catégorie de Grothendieck (resp. une catégorie abélienne avec assez de projectifs, resp. avec assez d'injectifs) alors $\fct(\C,\D)$ est une catégorie de Grothendieck (resp. une catégorie abélienne avec assez de projectifs, resp. avec assez d'injectifs).
\end{pr}

La notion suivante, dont on pourra trouver une présentation plus conceptuelle en termes d'extensions de Kan et quelques propriétés classiques dans \cite[§\,2.3]{Dja-FM}, par exemple, interviendra de façon importante dans plusieurs parties de cet article.

\begin{defi}\label{def-supp} Soient $\C$ et $\D$ des catégories ; on suppose $\C$ petite. Soient $S$ un ensemble d'objets de $\C$ et $F$ un foncteur de $\fct(\C,\D)$.
\begin{enumerate}
\item Si $\D$ est une catégorie cocomplète, on dit que $S$ est un \emph{support} de $F$ si ce foncteur est isomorphe à un quotient d'une somme de foncteurs du type $M[\C(s,-)]$, où $M$ est un objet de $\D$ et $s$ un élément de $S$.
\item  Dualement, si $\D$ est une catégorie complète, on dit que $S$ est un \emph{co-support} de $F$ si $F$ est isomorphe à un sous-foncteur d'un produit de foncteurs du type $M^{\C(-,s)}$, où $M$ est un objet de $\D$ et $s$ un élément de $S$.
\end{enumerate}
\end{defi}

Si $k$ est un anneau, on note $\F(\C;k)$ la catégorie $\fct(\C,k\Md)$. C'est une catégorie de Grothendieck dont un générateur projectif est donné par la somme des foncteurs $P^t_\C(-)=k[\C(t,-)]$ ($t$ parcourant un squelette de $\C$), qui représentent l'évaluation en $t$, et qu'on appelle \emph{projectifs standard}. Un foncteur est de type fini si et seulement s'il est quotient d'une somme finie de projectifs standard ; un tel foncteur est donc en particulier à support fini. 

Pour une catégorie source $\A$ additive et petite, et un anneau commutatif $k$, on dispose dans $\F(\A;k)$ d'isomorphismes naturels
\begin{equation}\label{eq-ptproj}
P^x_\A\otimes_k P^y_\A\simeq P^{x\oplus y}_\A
\end{equation}
dont on déduit :
\begin{pr}\label{pr-pttf} Si $\A$ est une petite catégorie additive et $k$ un anneau commutatif, le produit tensoriel sur $k$ de deux foncteurs de type fini de $\F(\A;k)$ est de type fini.
\end{pr}

Si $A$ est un anneau, on note $\mathbf{P}(A)$ la petite sous-catégorie pleine de $A\Md$ des $A$-modules projectifs de type fini. La catégorie $\F(\mathbf{P}(A);k)$ sera notée $\F(A,k)$.

Si $k=K$ est un corps, on note $\F^{\df}(\C;K)$ (resp. $\F^{\df}(A,K)$) la sous-catégorie pleine de $\F(\C;K)$ (resp. $\F(A,K)$) des foncteurs prenant des valeurs de dimensions finies sur $K$. En général, la structure de la catégorie $\F(\C;K)$ peut être très différente de la structure de sa sous-catégorie $\F^{\df}(\C;K)$. La situation est toutefois considérablement simplifiée lorsque la catégorie source satisfait de fortes propriétés de finitude, comme le montre la proposition suivante.

\begin{pr}\label{pr-dim-finie}
Supposons que pour tous les objets $x$ et $y$ de $\C$, l'ensemble $\C(x,y)$ est de cardinal fini. Alors les foncteurs de type fini (en particulier les foncteurs simples) de $\F(\C;K)$ sont à valeurs de dimensions finies.
\end{pr}

\begin{proof}
Si $F$ est de type fini, il est quotient d'une somme directe finie de projectifs standard $P^x_\C$. La propriété de finitude de $\C$ implique que ces projectifs standard sont à valeurs de dimensions finies, donc $F$ l'est également.
\end{proof}

\subsection{Prolongements intermédiaires}\label{par-prol-interm}

Le point de départ des rappels de ce paragraphe est la propriété classique suivante, démontrée dans \cite[prop.~3.4.1 et~3.4.2]{Borc}.

\begin{pr}\label{prreflex}
 Soient $\C$ et $\C'$ deux catégories et $F : \C\to\C'$ un foncteur. On suppose que $F$ possède un adjoint à gauche $G$ et un adjoint à droite $H$.
 
 Les assertions suivantes sont alors équivalentes :
 \begin{enumerate}
  \item l'unité ${\rm Id}_{\C'}\to FG$ est un isomorphisme ;
  \item la coünité $FH\to {\rm Id}_{\C'}$ est un isomorphisme.
 \end{enumerate}
 \end{pr}
 
La terminologie ci-dessous s'inspire de \cite[déf.~3.5.2]{Borc}.
 
 \begin{defi}\label{lm-2adj} On dit qu'un foncteur est \emph{réflexif} s'il possède des adjoints à droite et à gauche et vérifie les conditions équivalentes de la proposition~\ref{prreflex}.
 \end{defi}

 La définition suivante provient de la théorie des recollements, qui est  sous-jacente à ce paragraphe et à d'autres considérations de l'article, mais nous n'aurons explicitement besoin que de la notion de {\em prolongement intermédiaire}. Les recollements et prolongements intermédiaires ont d'abord été introduits dans le cadre des catégories triangulées dans \cite{BBD}. Pour le contexte abélien, notamment de catégories de foncteurs, nous nous référons à \cite{Ku2}.

\begin{defi}\label{df-prolint}
 Soient $\D$, $\D'$ des catégories abéliennes et $F : \D\to\D'$ un foncteur réflexif. On note $G$ et $H$ les adjoints à gauche et à droite respectivement de $F$. On appelle {\em prolongement intermédiaire} associé à $F$  le foncteur, noté $T_F : \D'\to\D$, image de la transformation naturelle $G\to H$ dont l'adjointe ${\rm Id}_{\D'}\to FH$ est l'inverse de la coünité.
\end{defi}

Pour les propriétés élémentaires suivantes, on renvoie par exemple à \cite[§\,4]{Ku2}.

\begin{pr}\label{pr-prolongement-interm}
 Soit $F : \D\to\D'$ un foncteur réflexif entre catégories abéliennes.
 \begin{enumerate}
 \item L'image par le foncteur $F$ d'un objet simple $s$ de $\D$ est soit nulle, soit simple dans $\D'$. Dans ce dernier cas, $s$ est isomorphe à $T_F F(s)$.
  \item L'endofoncteur $FT_F$ de $\D'$ est isomorphe à l'identité.
  \item Le foncteur $T_F$ est pleinement fidèle. Il est additif et préserve les monomorphismes et les épimorphismes.
  \item Le foncteur $T_F$ envoie un objet simple de $\D'$ sur un objet simple de $\D$. De plus, deux objets simples non isomorphes ont des images par $T$ non isomorphes.
 \end{enumerate}
\end{pr}

L'une des classes de prolongements intermédiaires les plus importantes dans la théorie des catégories de foncteurs est celle provenant de la proposition suivante, qui constitue un cas particulier de \cite[proposition~3.7.3]{Borc}.

\begin{pr}\label{restr-fct-refl}
 Soient $\C$ une petite catégorie et $\D$ une sous-catégorie pleine de $\C$. Le foncteur de restriction $\fct(\C,\E)\to\fct(\D,\E)$ est réflexif.
\end{pr}

En considérant le cas où $\D$ est une catégorie à un objet, la proposition~\ref{pr-prolongement-interm} permet de ramener l'étude des objets simples de $\F(\C;k)$ à celle des représentations des monoïdes d'endomorphismes de $\C$ :

\begin{cor}\label{simples-eval}
 Soient $\C$ une petite catégorie, $t$ un objet de $\C$ et $k$ un anneau. Alors le foncteur $\F(\C;k)\to k[{\rm End}_\C(t)]\Md$ d'évaluation en $t$ envoie un objet simple $F$ sur un objet nul ou simple, qui a alors le même corps d'endomorphismes que $F$.
 
 De plus, le prolongement intermédiaire associé prend ses valeurs dans les foncteurs de $\F(\C;k)$ dont $\{t\}$ est un support et un co-support.
\end{cor}

\begin{proof} La première assertion découle des propositions~\ref{pr-prolongement-interm} et~\ref{restr-fct-refl}.

Le prolongement intermédiaire est un quotient de l'adjoint à gauche à l'évaluation, ce qui implique qu'il prend ses valeurs dans les foncteurs dont $\{t\}$ est un support. Comme c'est un sous-objet de l'adjoint à droite à l'évaluation, il prend ses valeurs dans les foncteurs dont $\{t\}$ est un co-support.
\end{proof}

Lorsque les ensembles de morphismes sont finis dans $\C$, on peut améliorer l'énoncé précédent et ramener l'étude des objets simples de $\F(\C;k)$ à celle de représentations simples de \emph{groupes} finis d'automorphismes de $\C$. Pour ce faire, nous commençons par quelques rappels élémentaires sur les représentations des monoïdes -- lesquels seront toujours supposés \emph{unitaires} dans cet article.

\begin{nota}\label{nota-monoides} Soient $\M$ un monoïde et $k$ un anneau.
\begin{enumerate}
\item On note $\M^\times$ le groupe des éléments inversibles de $\M$, et  $\iota: k[\M^\times]\to k[\M]$ le morphisme d'anneaux induit par l'inclusion de $\M^\times$ dans $\M$.
\item On définit deux applications $k$-linéaires:
$$r: k[\M]\to k[\M^\times]\;,\qquad \delta:k[\M]\to k$$
nulles sur les éléments non inversibles de $\M$, et respectivement égales à l'identité et à $1$ sur $\M^\times$. Ainsi $r$ est une rétraction de $\iota$.
\end{enumerate}
\end{nota}

\begin{lm}\label{lm-mon-eltr}
Les trois conditions suivantes sont équivalentes :
\begin{enumerate}
\item[(i)] le produit de deux éléments de $\M$ est inversible si et seulement si les deux éléments sont inversibles,
\item[(ii)] l'application $r$ est un morphisme d'anneaux,
\item[(iii)] l'application $\delta$ est un morphisme d'anneaux.
\end{enumerate}
Lorsque ces conditions sont satisfaites, $\delta$ fait de $k$ un $k[\M]$-module encore noté $\delta$.
\end{lm}

\begin{ex}\label{ex-morannf}
Si $\M$ est un monoïde fini, alors (i), donc (ii) et (iii), sont vérifiées.
\end{ex}

\begin{lm}\label{lm-rep-mono-grp}
Supposons les conditions équivalentes du lemme~\ref{lm-mon-eltr} vérifiées. 
\begin{enumerate}
\item La restriction le long de $r$ induit un foncteur pleinement fidèle
$$r^*:k[\M^\times]\Md\to k[\M]\Md$$
De plus $r^*$ préserve les simples et induit une injection $\Irr(k[\M^\times])\hookrightarrow\Irr(k[\M])$ entre les classes d'isomorphisme de modules simples.
\item La restriction le long de $\iota$ induit un foncteur essentiellement surjectif  
$$\iota^*:k[\M]\Md\to k[\M^\times]\Md\;.$$
De plus $\iota^*$ préserve les simples si et seulement si $-\otimes\delta$ préserve les $k[\M]$-modules simples. Si c'est le cas, alors $\iota^*$ induit une bijection entre l'ensemble des classes d'isomorphisme des modules simples du type $M\otimes \delta$ et $\Irr(k[\M^\times])$.
\end{enumerate}
\end{lm}
\begin{proof}
L'énoncé est une conséquence directe des identités $\iota^*r^*M=M$ et $r^*\iota^*M=M\otimes\delta$.
\end{proof}

\begin{pr}\label{simples-sfin} Soient $k$ un anneau et $\C$ une petite catégorie dont les idempotents se scindent, et telle que $\C(x,x)$ soit un ensemble fini pour tout objet $x$ de $\C$. On note $\mathrm{Iso}(\C)$ l'ensemble des classes d'isomorphismes d'objets de $\C$. On a une bijection
\[\Phi:\bigsqcup_{[a]\in\mathrm{Iso}(\C)}\Irr(k[\mathrm{Aut}_\C(a)]\Md)\xrightarrow[]{\simeq}\Irr(\F(\C;k))\;.\]
De plus, $\Phi$ préserve les corps d'endomorphismes des simples.
\end{pr}

\begin{proof}
Pour tout objet $a$ de $\C$, soit $r_a:k[\mathrm{End}_\C(a)]\to k[\mathrm{Aut}_\C(a)]$ le morphisme d'anneaux qui est l'identité sur les éléments inversibles de $\mathrm{End}_\C(a)$ et nul sur les éléments non inversibles (cf. supra). Notons $T_a :k[\mathrm{End}_\C(a)]\Md\to\F(\C;k)$ le prolongement intermédiaire associé à l'évaluation d'un foncteur sur $a$. 
D'après le lemme \ref{lm-rep-mono-grp} et le corollaire \ref{simples-eval}, le foncteur
$\Phi_a=T_a\circ r_a^*$ est pleinement fidèle et induit une injection (toujours notée $\Phi_a$):
$$\Phi_a: \Irr(k[\mathrm{Aut}_\C(a)]\Md)\to\Irr(\F(\C;k))\qquad [M]\to [T_a(r_a^*M)] \;.$$
Pour démontrer la proposition, il reste à établir que tout élément de $\Irr(\F(\C;k))$ est dans l'image de $\Phi_a$ pour un élément $a$ unique à isomorphisme près. 

Pour voir que tout élément $[S]$ de $\Irr(\F(\C;k))$ est dans l'image d'un $\Phi_a$, prenons un objet $x$ de $\C$ tel que $S(x)\ne 0$. Par le corollaire~\ref{simples-eval},  $S(x)$ est un $k[\mathrm{End}_\C(x)]$-module simple. Comme le monoïde $\mathrm{End}_\C(x)$ est fini, \cite[thm~5.5]{BSt} donne un idempotent $e\in \mathrm{End}_\C(x)$ tel que $eS(x)e\ne 0$ et tel que les éléments non inversibles du monoïde $e \mathrm{End}_\C(x)e$ opèrent par $0$ sur $S(x)$. Soit $a$ l'image de $e$ dans $\C$ (qui existe par hypothèse sur $\C$). On en déduit que $S(a)$ est un $k[{\rm Aut}_\C(a)]$-module simple et que $S$ est isomorphe à l'image par $\Phi_a$ de $S(a)(\simeq eS(x))$.

Il reste à voir que si $a$ et $b$ sont deux objets non isomorphes de $\C$, les images de $\Phi_a$ et $\Phi_b$ sont disjointes. Cela peut se déduire facilement des résultats de \cite[§\,5.2]{BSt}, ou s'établir directement : la formule du prolongement intermédiaire $T_a$ montre que, si $M$ est un $k[\mathrm{Aut}_\C(a)]$-module (non nécessairement simple), alors $T_a(r_a^*M)(x)$ ne peut être non nul que si $a$ est facteur direct de $x$.
\end{proof}

\section{Foncteurs polynomiaux}\label{paragraphe-pol}

Dans cette section on se fixe une catégorie abélienne $\V$ arbitraire (dans la partie \ref{part-resfond} de l'article et les parties suivantes, $\V$ sera la catégorie des espaces vectoriels sur un corps $K$, mais on peut prendre une catégorie plus générale ici). Les foncteurs polynomiaux d'une petite catégorie additive  $\A$ vers $\V$, initialement introduits par Eilenberg et MacLane \cite[chapitre~II]{EML}, constituent une généralisation naturelle des foncteurs additifs.
Cette section rassemble les définitions et les principales propriétés des foncteurs polynomiaux dont nous aurons besoin dans l'article. Tous les résultats sont classiques, hormis ceux de la section \ref{ssctfp} où nous introduisons la notion nouvelle de \emph{foncteur \ph}, et certains énoncés de la section \ref{par-msec} qui ne sont pas explicités dans l'article de Pirashvili \cite{Pira88} dont ils sont inspirés.

\subsection{Foncteurs additifs}\label{ssctadd} Un foncteur $F$ de la catégorie additive $\A$ vers la catégorie abélienne $\V$ est \emph{additif} s'il vérifie les propriétés équivalentes suivantes.
\begin{enumerate}
\item\label{itadd1} pour tout couple $(x,y)$ d'objets de $\A$, la fonction $\A(x,y)\to\V(F(x),F(y))$ qu'induit $F$ est additive ;
\item $F(0)\simeq 0$ et pour tout couple $(x,y)$ d'objets de $\A$ le morphisme canonique $F(x\oplus y)\to F(x)\oplus F(y)$ de $\V$ est un monomorphisme ;
\item pour tout couple $(x,y)$ d'objets de $\A$, le morphisme canonique $F(x\oplus y)\to F(x)\oplus F(y)$ est un isomorphisme.
\end{enumerate}

La sous-catégorie pleine de $\fct(\A,\V)$ constituée des foncteurs additifs est épaisse, stable par limites et colimites. On la notera $\mathbf{Add}(\A,\V)$, et on notera $\mathbf{Add}(\A;k)$ pour $\mathbf{Add}(\A,k\Md)$, lorsque $k$ est un anneau. On rappelle un résultat classique de Eilenberg et Watts \cite{E-add,W-add}.

\begin{pr}\label{additifs-bimodules}
 Soient $A$ et $k$ deux anneaux. L'évaluation en $A$ induit une équivalence de catégories
 $$\mathbf{Add}(\mathbf{P}(A);k)\xrightarrow{\simeq} (A^{\op}\otimes_{\mathbb{Z}} k)\Md$$
 dont un quasi-inverse est donné par la tensorisation au-dessus de $A$ avec un $(k,A)$-bimodule.
\end{pr}

\subsection{Foncteurs polynomiaux}\label{subsec-fctpol} 
La notion de foncteur polynomial introduite par Eilenberg et MacLane \cite[chapitre~II]{EML} généralise celle de foncteur additif. Comme rappelé ci-dessus, les foncteurs additifs peuvent être définis du point de vue de l'effet du foncteur sur les morphismes, ou du point de vue de l'effet du foncteur sur les objets. 
Il en va de même de la notion de foncteur polynomial, pour lequel le point de vue de l'effet sur les morphismes repose sur la notion de fonction polynomiale, et le point de vue de l'effet sur les objets repose sur la notion d'effet croisé.

\paragraph*{Fonctions polynomiales}
Soient $U$ et $V$ des groupes abéliens, $f : U\to V$ une fonction et $d\in\mathbb{N}$.
La $d$-ième {\em déviation} de $f$ \cite[§\,8]{EML}\,\footnote{Nous ne suivons pas l'indexation d'Eilenberg et MacLane pour les déviations (et plus
tard, les effets croisés) : ce que nous nommons $d$-ième déviation (resp. effet croisé) correspond à leur $(d-1)$-ième déviation (resp. effet croisé).
Cela nous permet de nous conformer, plus bas, aux conventions les plus courantes pour les effets croisés.} est la fonction $\delta_d(f) : U^d\to V$ définie par
$$\delta_d(f)(u_1,\dots,u_d):=\underset{I\subset\{1,\dots,d\}}{\sum} (-1)^{d-|I|}f\Big(\underset{i\in I}{\sum} u_i\Big)\;;$$
où $|I|$ désigne le cardinal de l'ensemble $I$. On dit que $f$ est une {\em fonction polynomiale} de degré au plus $d$ si la fonction
$\delta_{d+1}(f)$ est identiquement nulle (ce qui implique que les déviations $\delta_i(f)$ sont également nulles pour $i\geq d+1$). Les fonctions polynomiales (resp. les fonctions polynomiales de degré au plus $d$) forment un sous-groupe noté $\mathrm{Pol}(U,V)$ (resp. $\mathrm{Pol}_d(U,V)$) du groupe abélien des fonctions $U\to V$.
Pour une étude systématique de ces fonctions polynomiales, on pourra consulter Passi \cite[Chap.~V]{Passi}.
\begin{ex}\label{ex-fctpol}
\begin{enumerate}
\item[(i)]
Une fonction est  polynomiale de degré au plus $0$ si et seulement si elle est constante. Une fonction $f$ est polynomiale de degré au plus $1$ si et seulement si la fonction $f-f(0)$ est additive.
\item[(ii)] Si $U=\mathbb{Z}=V$ la notion de fonction polynomiale coïncide avec la notion de polynôme numérique \cite[V.1.3(vi)]{Passi}. 
\item[(iii)] si $U$ est un $p$-groupe abélien fini et $V$ est un $\FF_p$-espace vectoriel, toute fonction $f:U\to V$ est polynomiale, d'après \cite[Chap.~VI, Th.~1.2]{Passi}.
\end{enumerate}
\end{ex}

\paragraph*{Effets croisés}
Si $F$ est un objet de $\fct(\A,\V)$ on note $cr_0(F)=F(0)$. Si $d$ est un entier strictement positif, le {\em $d$-ième effet croisé} est le foncteur $cr_d : \fct(\A,\V)\to\fct(\A^d,\V)$ défini par
$$cr_d(F)(x_1,\dots,x_d):={\rm Ker}\,\Big(F\Big(\bigoplus_{i=1}^d x_i\Big)\to\bigoplus_{i=1}^d F\Big(\underset{j\neq i}{\bigoplus}x_j\Big)\Big)$$
où les morphismes sont induits par les projections canoniques. On a une décomposition naturelle \cite[théorème~9.1]{EML}:
\begin{equation}\label{eq-dec-cr}F\Big(\bigoplus_{i=1}^d x_i\Big)\simeq\bigoplus_{s=0}^d\underset{1\leq j_1<\dots<j_s\leq d}{\bigoplus}cr_s(F)(x_{j_1},\dots,x_{j_s}).
\end{equation}

\begin{ex}\label{ex-cr}
\begin{enumerate}
\item[(i)] $cr_1(F)(x)$ est le noyau de $F(x)\xrightarrow[]{F(0)}F(0)$. Le foncteur $cr_1 F$ est souvent noté $\overline{F}$ et appelé {\em partie réduite} de $F$. On a $F\simeq\overline{F}\oplus F(0)$.
\item[(ii)] $cr_2(F)(x,y)$ est le noyau du morphisme canonique $F(x\oplus y)\to F(x)\oplus F(y)$.
\end{enumerate}
\end{ex}

\paragraph*{Foncteurs polynomiaux} Un foncteur $F$ de $\fct(\A,\V)$ est {\em polynomial de degré au plus $d$} s'il vérifie les conditions suivantes, qui sont équivalentes d'après \cite[Th.~9.11]{EML} :
\begin{enumerate}
\item pour tout couple $(x,y)$ d'objets de $\A$, la fonction $F_{x,y}:\A(x,y)\to\V(F(x),F(y))$ induite par $F$ est polynomiale de degré au plus $d$ ;
\item $cr_{d+1}(F)=0$.
\end{enumerate}
Un foncteur $F$ est \emph{de degré $d$} s'il est de degré au plus $d$ et si $cr_d(F)\ne 0$. Enfin, un foncteur est simplement dit \emph{polynomial} s'il existe un entier $d$ tel qu'il soit polynomial de degré au plus $d$.
\begin{ex}
\begin{enumerate}
\item[(i)] Les foncteurs constants sont les foncteurs polynomiaux de degré au plus $0$. Les foncteurs additifs sont les foncteurs polynomiaux de degré au plus $1$ et nuls sur l'objet nul.
\item[(ii)] Si $F$ (resp. $G$) est un foncteur de $\F(\A;k)$ polynomial de degré au plus $d$ (resp. $e$), alors le produit tensoriel $F\otimes_k G$ (défini par $(F\otimes_k G)(x)=F(x)\otimes_k G(x)$) est polynomial de degré au plus $d+e$. 
\end{enumerate}
\end{ex}

On note $\pol_d(\A,\V)$ la sous-catégorie pleine de $\fct(\A,\V)$ dont les objets sont les foncteurs polynomiaux de degré au plus $d$. C'est une sous-catégorie épaisse, stable par limites et colimites, car le foncteur $cr_{d+1}$ commute aux limites et colimites (il est en particulier exact). Cette dernière propriété résulte de la propriété analogue de tout foncteur de précomposition et de l'isomorphisme naturel \eqref{eq-dec-cr}.

Si $\Phi : \A\to\B$ est un foncteur additif entre petites catégories additives, le foncteur de précomposition $\Phi^*$ envoie $\pol_d(\B,\V)$ dans $\pol_d(\A,\V)$.

\subsection{Foncteurs \phs}\label{ssctfp} 
Pour les besoins de notre article, nous introduisons une généralisation de la notion de foncteur polynomial.

\begin{defi}\label{df-ph} Un foncteur $F : \A\to\V$ est \emph{\ph} si, pour tous objets $x$ et $y$ de $\A$, la fonction $F_{x,y}:\A(x,y)\to\V(F(x),F(y))$ induite par $F$ est polynomiale.
\end{defi}

Les foncteurs polynomiaux sont \phs\ -- par définition, les foncteurs polynomiaux sont exactement les foncteurs $F$ \,\phs\ pour lesquels  il existe un entier $d$ majorant le degré de toutes les fonctions polynomiales $F_{x,y}$. L'exemple~\ref{ex-ttph} ci-dessous montre que la classe des foncteurs \phs\ peut être bien plus vaste que celle des foncteurs polynomiaux. Contrairement aux foncteurs polynomiaux, pour lesquels on dispose du critère par annulation d'effets croisés, les foncteurs \phs\ ne semblent pas pouvoir se caractériser facilement autrement que par leur définition, qui est une condition polynomiale sur les Hom, ce qui motive la terminologie employée.

Le résultat suivant est immédiat.

\begin{pr}\label{pr-hpsq}La classe des foncteurs \phs\ est stable par sous-quotient et par sommes directes finies.
\end{pr}

\begin{ex}\label{ex-ttph} Soit $k$ un anneau de caractéristique première $p$. Supposons que les groupes abéliens $\A(x,y)$ sont toujours des $p$-groupes finis. D'après l'exemple \ref{ex-fctpol} (iii), tous les foncteurs de $\F(\A;k)$ sont \phs. Par contraste, si $x$ est un objet non nul de $\A$, le foncteur projectif standard $P^x_\A=k[\A(x,-)]$ n'est pas polynomial car  $cr_d P_\A^x(x_1,\dots,x_d) = \overline{P^x_\A}(x_1)\otimes_k\dots\otimes_k \overline{P^x_\A}(x_d)$ avec $\overline{P^x_\A}=cr_1 P^x_\A\ne 0$.
\end{ex}

Sur un corps de caractéristique nulle, ou plus généralement sur une $\mathbb{Q}$-algèbre, le lien entre foncteurs \phs\ et polynomiaux est simple grâce au résultat de décomposition suivant. Un foncteur $F$ de $\F(\A;k)$ est \emph{$d$-homogène} si $F(\lambda f)=\lambda^d F(f)$ pour tout morphisme $f$ de $\A$ et tout $\lambda\in \mathbb{Z}$.
\begin{pr}\label{pr-polph-car0} Soit $k$ une $\mathbb{Q}$-algèbre. Tout foncteur \ph\ $F$ de $\F(\A;k)$ est isomorphe à une somme directe $\bigoplus_{d\in\mathbb{N}}F_d$ où chaque foncteur $F_d$ est polynomial de degré au plus $d$ et $d$-homogène. De plus, les $F_d$ sont uniques à isomorphisme près.
\end{pr}

La démonstration de la proposition \ref{pr-polph-car0} repose sur le lemme suivant.

\begin{lm}\label{lm-fonction-pol-car0} Soient $U$ et $V$ des groupes abéliens, avec $V$ uniquement divisible, et $f : U\to V$ une fonction polynomiale. Alors $f$ se décompose de façon unique comme une somme presque nulle $f=\sum_{d\in\mathbb{N}}f_d$ où les $f_d : U\to V$ sont des fonctions polynomiales vérifiant la relation $f_d(\lambda.u)=\lambda^d.f_d(u)$ pour tout $(\lambda,u)\in\mathbb{Z}\times U$.
\end{lm}
\begin{proof}[Démonstration du lemme \ref{lm-fonction-pol-car0}] Le cas où $U=\mathbb{Z}$ se déduit directement de \cite[Ch.~V, Th.~2.1]{Passi}. Le cas général s'y ramène en considérant la fonction polynomiale
$\Psi_f:\mathbb{Z}\to\mathrm{Pol}(U,V)$ telle que $\Psi_f(\lambda):u\mapsto f(\lambda u)$.
\end{proof}

\begin{proof}[Démonstration de la proposition \ref{pr-polph-car0}]
L'unicité provient du fait que tout morphisme (en particulier tout isomorphisme) $F\to G$ doit nécessairement préserver les termes de décompositions en sommes directes de foncteurs homogènes.

Pour l'existence, pour chaque objet $x$ de $\A$ on considère l'application polynomiale $\Psi:\mathbb{Z}\to {\rm End}_k(F(x))$ telle que $\Psi(\lambda)= F(\lambda.\mathrm{Id}_x)$ et on note $\Psi_d$ ses composantes homogènes fournies par le lemme \ref{lm-fonction-pol-car0}. Comme $\Psi(\lambda\mu)=\Psi(\lambda)\circ\Psi(\mu)$ et $\Psi(1)=\mathrm{Id}_{F(x)}$, les $\Psi_d(1)$ forment une famille de projecteurs orthogonaux de $F(x)$, et on a une décomposition $F(x)=\bigoplus_{d\in\mathbb{N}} F(x)_d$  où $F(x)_d$ est l'image de $\Psi_d(1)$. Par fonctorialité de $F$ et par unicité des composantes homogènes d'une application polynomiale, la composante $d$-homogène de $F_{x,y}: \A(x,y)\to {\rm Hom}_k(F(x),F(y))$ est à valeurs dans ${\rm Hom}_k(F(x)_d,F(y)_d)$. On peut donc définir un foncteur $F_d$ adéquat comme le foncteur qui à $x$ associe $F(x)_d$ et dont l'effet sur les morphismes est donné par les composantes $d$-homogènes des applications polynomiales $F_{x,y}$.
\end{proof}

\begin{rem}\label{rq-algext} Les foncteurs $F_d$ de la proposition \ref{pr-polph-car0} peuvent être tous non nuls, même si $F$ est à valeurs de dimensions finies, comme l'illustre l'algèbre extérieure $\bigoplus_{d\in\mathbb{N}}\Lambda^d$ vue comme endofoncteur des $K$-espaces vectoriels de dimensions finies.
\end{rem}

La proposition~\ref{pr-polph-car0} implique qu'un foncteur \ph\ de type fini de $\F(\A;k)$ est nécessairement polynomial lorsque l'anneau $k$ est une $\mathbb{Q}$-algèbre. D'après l'exemple~\ref{ex-ttph}, ce n'est pas le cas pour une catégorie $\V$ arbitraire au but. Nous pouvons toutefois assurer qu'un foncteur \ph\  est polynomial sous une condition de finitude plus forte. Nous renvoyons à la définition \ref{def-supp} pour la notion de (co-)support. Par exemple les foncteurs de longueur finie, ou plus généralement les foncteurs à la fois de type fini et de type cofini, possèdent un support fini et un co-support fini.
\begin{pr}\label{pr-phpol-lf}
Tout foncteur \ph\ de $\fct(\A,\V)$ qui possède un support fini et un co-support fini  est polynomial.
\end{pr}

Cette proposition découle du lemme suivant.

\begin{lm}\label{lm-sucosu} Soient $F$ un foncteur de $\fct(\A,\V)$, $S$ un support et $T$ un co-support de $F$, et $d\in\mathbb{N}$. Supposons que, pour tous objets $s$ dans $S$ et $t$ dans $T$, l'application $F_{s,t}:\A(s,t)\to\V(F(s),F(t))$ est polynomiale de degré au plus $d$. Alors $F$ est polynomial de degré au plus~$d$.
\end{lm}

\begin{proof}
Soient $x$ et $y$ deux objets de $\A$. Considérons le carré commutatif
\[\xymatrix{\A(x,y)\ar[r]\ar[d] & \V(F(x),F(y))\ar[d]\\
\prod\A(s,t)\ar[r] & \prod \V(F(s),F(t))
}\]
où les produits sont pris sur tous les $s\in S$, $t\in T$ et toutes les flèches $f : s\to x$ et $g : y\to t$ de $\A$, les flèches horizontales sont données par l'effet de $F$ sur les morphismes, et les flèches verticales sont induites par l'image inverse le long de $f$ et l'image directe le long de $g$.

La flèche horizontale inférieure est une fonction polynomiale de degré au plus $d$. 

Les flèches verticales sont des morphismes de groupes abéliens, et celle de droite est {\em injective} puisque $S$ (resp. $T$) est un support (resp. co-support) de $F$. Il en résulte que la flèche horizontale supérieure est une fonction polynomiale de degré au plus $d$, ce qui achève la démonstration.
\end{proof}

Enfin, sans aucune hypothèse de finitude, l'existence de foncteurs \phs\ non constants est équivalente à l'existence de foncteurs polynomiaux non constants :

\begin{pr}\label{add-pol-cst}
Les assertions suivantes sont équivalentes.
\begin{enumerate}
 \item[{\rm (a)}] La catégorie $\mathbf{Add}(\A,\V)$ est réduite à $0$.
 \item[{\rm (b)}] Tout foncteur polynomial $\A\to\V$ est constant.
 \item[{\rm (c)}] Tout foncteur \ph\ $\A\to\V$ est constant.
\end{enumerate}
Si $\V=k\Md$ pour un anneau $k$, ces assertions équivalent aussi aux suivantes :
\begin{enumerate}
 \item[{\rm (d)}] Pour tous les objets $x$ et $y$ de $\A$, $\A(x,y)\otimes_{\mathbb{Z}} k=0$.
 \item[{\rm (e)}] Pour tout objet $x$ de $\A$, $\A(x,x)\otimes_{\mathbb{Z}} k=0$.
\end{enumerate}
\end{pr}
\begin{proof}
On a clairement (c)$\Rightarrow$(b)$\Rightarrow$(a). Soit $F$ un foncteur polynomial de degré au plus $d$. En utilisant \cite[Thm 9.6]{EML} et l'annulation de $cr_{d+1}F$, on constate que $cr_d F$ est additif par rapport à chacune de ses $d$ variables. Si $d$ est strictement positif et si l'on suppose (a), alors $cr_d F$ doit être nul, donc $F$ est de degré au plus $d-1$, et en itérant ce raisonnement on obtient que $F$ est degré au plus $0$, donc constant. Ainsi, (a)$\Rightarrow$(b). D'après les propositions~\ref{pr-phpol-lf} et~\ref{pr-hpsq}, pour montrer (b)$\Rightarrow$(c), il suffit de montrer que tout foncteur $F$ non constant de $\fct(\A,\V)$ admet un sous-quotient non constant de type fini et de type co-fini. 

Pour cela, quitte à remplacer $F$ par $cr_1 F$, on peut supposer $F(0)=0$. Comme $F$ est colimite de ses sous-foncteurs de type fini, il admet un sous-foncteur non nul de type fini $G$. Il existe un objet $M$ de $\V$, un objet $x$ de $\A$ et un morphisme non nul $G\to M^{\A(-,x)}$, grâce à l'isomorphisme \eqref{eq-injy} (page~\pageref{eq-injy}). L'image de ce morphisme fournit un sous-quotient  de $F$ non constant, de type fini et de type co-fini.

L'équivalence de (a) avec (d) et (e), pour $\V=k\Md$, résulte de ce que la catégorie $\mathbf{Add}(\A,\V)$ est alors engendrée par les foncteurs $\A(x,-)\otimes_{\mathbb{Z}} k$, dont les anneaux d'endomorphismes sont $\A(x,x)\otimes_{\mathbb{Z}} k$.
\end{proof}

\subsection{Fonctions de dimensions} 
Lorsque $\A=\mathbf{P}(A)$ et $\V$ est la catégorie des espaces vectoriels sur un corps $K$, les foncteurs polynomiaux à valeurs de dimensions finies peuvent être caractérisés par leurs fonctions de dimensions. 

\begin{defi}\label{defi-fct-dim} Soient $A$ un anneau et $F$ un foncteur de $\F^{\df}(A,K)$. La \emph{fonction de dimensions} de $F$ est la fonction   
\[\mathrm{d}_F : \mathbb{N}\to\mathbb{N}\qquad n\mapsto\dim_K F(A^n)\;.\]
\end{defi}

Une fonction $f:\mathbb{N}\to \mathbb{N}$ est dite \emph{polynomiale} de degré $d$ si elle est la restriction d'une fonction $\mathbf{f}:\mathbb{Z}\to\mathbb{Z}$ polynomiale de degré $d$.

La proposition suivante, conséquence de la décomposition de $F(A^n)$ donnée par \eqref{eq-dec-cr} (page~\pageref{eq-dec-cr}, avec tous les $x_i$ égaux à $A$), est établie par Kuhn \cite[§\,4]{Ku1} lorsque $A=K$ est un corps fini, mais la démonstration s'applique de façon inchangée au cas général.

\begin{pr}\label{pr-fct-dim-pol}
Soient $A$ un anneau et $F$ un foncteur de $\F^{\df}(A,K)$. Si $F$ est polynomial, alors $\mathrm{d}_F$ est une fonction polynomiale de même degré que $F$, sinon $\mathrm{d}_F$ croît plus vite que toute fonction polynomiale.
\end{pr}

\subsection{Lemme de Pirashvili}
Un foncteur d'une catégorie additive vers la catégorie abélienne 
$\V$ est dit \emph{réduit} s'il envoie l'objet nul sur $0$. Un foncteur depuis un produit de $d$ catégories additives vers $\V$ est dit $d$-\emph{multiréduit} s'il est réduit par rapport à chacune des $d$ variables. La sous-catégorie pleine de $\fct(\A^d,\V)$ des foncteurs $d$-multiréduits sera notée $\fct^{d\text{-}\rm{red}}(\A^d,\V)$. Si $k$ est un anneau et $\V=k\Md$ on notera cette catégorie $\F^{d\text{-}\rm{red}}(\A^d;k)$. 

Pour tout entier $d>0$, les effets croisés $cr_d(F)$ d'un foncteur $F$ de $\fct(\A,\V)$ sont des foncteurs multiréduits. Les effets croisés définissent donc des foncteurs : 
\begin{align}cr_d : \fct(\A,\V)\to\fct^{d\text{-}\rm{red}}(\A^d,\V)\;.
\label{eq-crn}
\end{align}
Si $\Delta_d:\A\to \A^{\times d}$ désigne la diagonale $d$-itérée, c'est-à-dire le foncteur additif défini par $\Delta_d(x)=(x,\dots,x)$, on notera encore 
\begin{align}
\Delta_d^* : \fct^{d\text{-}\rm{red}}(\A^d,\V)\to\fct(\A,\V)
\label{eq-Deltan}
\end{align} 
la restriction du foncteur de précomposition $\Delta_d^* : \fct(\A^d,\V)\to\fct(\A,\V)$ à la sous-catégorie $\fct^{d\text{-}\rm{red}}(\A^d,\V)$.
L'adjonction classique suivante est fondamentale.

\begin{pr}\label{adj-diagcr} 
Pour tout $d>0$, le foncteur \eqref{eq-crn} est adjoint des deux côtés au foncteur \eqref{eq-Deltan}.
\end{pr}

\begin{proof} Comme $\A$ est additive, le foncteur $\Sigma_d : \A^d\to\A$ défini par $\Sigma_d(x_1,\dots,x_d)=\bigoplus_{1\le i\le d}x_i$ est adjoint des deux côtés à $\Delta_d$. Il s'ensuit que les foncteurs de précomposition 
$$\Sigma_d^* : \fct(\A,\V)\leftrightarrows\fct(\A^d,\V):\Delta_d^*$$
sont adjoints des deux côtés. Pour tout foncteur $G:\A\to \V$ la décomposition \eqref{eq-dec-cr} de la section \ref{subsec-fctpol} donne un isomorphisme $\Sigma_d^*G\simeq cr_dG\oplus G'$ où $G'$ est une somme directe de foncteurs $G'_i$ (pour $1\le i\le d$) constants par rapport à la $i$-ème variable. Mais il n'y a pas de morphisme non nul entre un tel foncteur $G_i$ et un foncteur multiréduit $F$ (en effet, si $f$ est un morphisme entre $F$ et $G'_i$, alors $f$ factorise à travers le foncteur $(x_1,\dots,x_d)\mapsto F(x_1,\dots,x_{i-1},0,x_{i+1},\dots,x_d)$, et comme ce dernier est nul, $f$ est nul). Ainsi tout isomorphisme d'adjonction entre $\Sigma_d^*$ et $\Delta_d^*$ se restreint en un isomorphisme d'adjonction entre les foncteurs \eqref{eq-crn} et \eqref{eq-Deltan}.
\end{proof}

La proposition~\ref{adj-diagcr} est souvent utilisée par l'intermédiaire du résultat suivant.

\begin{cor}[Pirashvili]\label{cor-annul-pira} Soient $k$ un anneau commutatif et $F$ un foncteur de $\F(\A;k)$. Pour tout $d> 0$, les assertions suivantes sont équivalentes.
\begin{enumerate}
\item\label{icp1} Le foncteur $F$ est polynomial de degré au plus $d-1$.
\item\label{icp2} Pour tout objet $X$ de $\F^{d\text{-}\rm{red}}(\A^{d};k)$, on a $\F(\A;k)(\Delta_{d}^*X,F)=0$.
\item\label{icp3} Pour tout objet $G$ de $\F^{\rm{red}}(\A;k)$, on a $\F(\A;k)(G^{\otimes_k d},F)=0$.
\end{enumerate}
\end{cor}

\begin{proof}
Si $d=1$, l'équivalence entre les trois assertions découle de l'équivalence de catégories $\F(\A;k)\simeq \F^{\rm red}(\A;k)\times k\Md$ (voir l'exemple \ref{ex-cr}). Supposons $d>1$. L'implication \ref{icp1}$\Rightarrow$\ref{icp2} découle de la proposition~\ref{adj-diagcr}. L'implication \ref{icp2}$\Rightarrow$\ref{icp3} s'obtient en prenant pour $X$ le foncteur $X(x_1,\dots, x_d)= G(x_1)\otimes_k\dots\otimes_k G(x_d)$.
Montrons maintenant \ref{icp3}$\Rightarrow$\ref{icp1}. Le $k$-module $(\Delta_d^*cr_d F)(x_1\oplus\dots\oplus x_d)$ contient $cr_d F(x_1,\dots,x_d)$ comme facteur direct, donc $cr_d F=0$ si et seulement si $\Delta_d^*cr_d F=0$. Mais pour tout $x$ de $\A$:
$$(\Delta_d^*cr_d F)(x)\simeq \F(\A;k)(P^x_\A,\Delta_d^*cr_d F)\simeq \F(\A;k)(\Delta_d^*cr_d P^x_\A,F)\;,$$
le premier isomorphisme étant donné par le lemme de Yoneda, et le deuxième par la proposition~\ref{adj-diagcr}. Mais $\Delta_d^*cr_dP^x_\A = \overline{P^x_\A}^{\otimes_k d}$ où $\overline{P^x_\A}=cr_1P^x_\A$ est un foncteur réduit (cf. exemple~\ref{ex-ttph}). L'assertion \ref{icp3} implique ainsi la nullité de $cr_d F$, donc l'assertion \ref{icp1}.
\end{proof}

\subsection{Multifoncteurs symétriques et effets croisés}\label{par-msec} Les déviations d'une fonction $f : U\to V$ entre groupes  abéliens sont symétriques : pour tout $d\in\mathbb{N}$ et toute permutation $\sigma\in\Si_d$, on a $\delta_d(f)(u_1,\dots,u_d)=\delta_d(f)(u_{\sigma(1)},\dots,u_{\sigma(d)})$. Cette invariance par l'action du groupe symétrique a son analogue pour les effets croisés des foncteurs, mais de façon un peu plus subtile car les égalités doivent être remplacées par des isomorphismes cohérents au sens approprié, que nous détaillons maintenant. Le contenu de ce §\,\ref{par-msec} est connu des experts ; une partie figure dans Pirashvili \cite{Pira88}.

Pour tout $d\in\mathbb{N}$ on note $\mathbf{Add}_d(\A,\V)$ la sous-catégorie pleine de $\fct(\A^d,\V)$ des foncteurs qui sont additifs par rapport à chacune des $d$ variables. Le groupe symétrique $\Si_d$ agit \emph{à gauche} sur $\fct(\A^d,\V)$ et sur $\mathbf{Add}_d(\A,\V)$, via
\[(\sigma.F)(a_1,\dots,a_d):= F(a_{\sigma(1)},\dots,a_{\sigma(d)}).\]

La catégorie des foncteurs additifs multi-symétriques de $\A$ vers $\V$ est la catégorie $\Sigma\mathbf{Add}_d(\A,\V)$ définie de la manière suivante. Ses objets sont les objets $F$ de $\mathbf{Add}_d(\A,\V)$ munis d'une famille d'isomorphismes $\xi_\sigma : \sigma.F\simeq F$ pour $\sigma\in\Si_d$ telle que que pour tous $\sigma, \tau\in\Si_d$ la composée suivante soit égale à $\xi_{\sigma\tau}$:
$$(\sigma\tau).F=\sigma.(\tau.F)\xrightarrow[]{\sigma.\xi_\tau}\sigma.F\xrightarrow[]{\xi_\sigma} F\;.$$
Les morphismes de $\Sigma\mathbf{Add}_d(\A,\V)$ sont les morphismes des objets sous-jacents de $\mathbf{Add}_d(\A,\V)$ qui sont compatibles aux isomorphismes structuraux $\xi_\sigma$. 

On dispose d'une paire de foncteurs:
$$\mathcal{O}:\Sigma\mathbf{Add}_d(\A,\V)\leftrightarrows\mathbf{Add}_d(\A,\V):\mathcal{L}$$
où $\mathcal{O}$ désigne le foncteur d'oubli, et $\mathcal{L}$ envoie un foncteur $F$ sur le foncteur symétrisé $\mathcal{L}(F)=\bigoplus_{\tau\in\Si_d}\tau.F$ muni des isomorphismes
\[\xi_\sigma : \sigma.\mathcal{L}(F)=\bigoplus_{\tau\in\Si_d}(\sigma\tau).F\xrightarrow{\simeq}\bigoplus_{\tau\in\Si_d}\tau.F=\mathcal{L}(F)\]
fournis par l'isomorphisme d'échange des facteurs induit par la bijection $\tau\mapsto\sigma\tau$ de l'ensemble d'indexation $\Si_d$ de la somme directe.

\begin{lm}\label{lm-adj-cro}
Le foncteur $\mathcal{L}$ est adjoint des deux côtés au foncteur $\mathcal{O}$.
\end{lm}

\begin{proof} Soient $F$ et $X$ des objets de $\mathbf{Add}_d(\A,\V)$ et $\Sigma\mathbf{Add}_d(\A,\V)$ respectivement. Un morphisme $f : \mathcal{L}(F)\to X$ dans $\Sigma\mathbf{Add}_d(\A,\V)$ consiste en la donnée de morphismes $f_\tau : \tau.F\to\mathcal{O}(X)$ de $\mathbf{Add}_d(\A,\V)$ pour $\tau\in\Si_d$ de sorte que, pour tout $\sigma\in\Si_d$, le diagramme
\[\xymatrix{\sigma.(\tau.F)\ar[r]^-{\sigma.f_\tau}\ar@{=}[d] & \sigma.\mathcal{O}(X)\ar[d]^\simeq\\
(\sigma\tau).F\ar[r]_{f_{\sigma\tau}} & \mathcal{O}(X)
}\]
commute, où la flèche verticale de droite est l'isomorphisme structural de multi-symétrie. Il s'ensuit que l'application naturelle envoyant $f\in\mathrm{Hom}(\mathcal{L}(F),X)$ sur sa composante $f_1\in\mathrm{Hom}(F,\mathcal{O}(X))$ est une bijection, dont la réciproque envoie un morphisme $g$ sur le morphisme de composantes $f_\sigma : \sigma.F\xrightarrow{\sigma.g}\sigma.\mathcal{O}(X)\simeq\mathcal{O}(X)$.

Ainsi, $\mathcal{L}$ est adjoint à gauche à $\mathcal{O}$ ; la démonstration de l'adjonction dans l'autre sens est identique.
\end{proof}

Le foncteur de $d$-ième effet croisé se restreint en un foncteur (toujours noté $cr_d$) $\pol_d(\A,\V)\to\mathbf{Add}_d(\A,\V)$, qui se relève à travers $\mathcal{O}$ en un foncteur que nous noterons $\mathrm{Cr}_d$:
$$\xymatrix{
\pol_d(\A,\V)\ar[rrd]_-{cr_d}\ar[rr]^-{\mathrm{Cr}_d}&& \Sigma\mathbf{Add}_d(\A,\V)\ar[d]^-{\mathcal{O}}\\
&&\mathbf{Add}_d(\A,\V)
}\;.$$

\begin{pr}[Pirashvili, cf. \cite{Pira88}]\label{recollement-cr}
 Soient $\V$ une catégorie abélienne et $d\in\mathbb{N}$. Le foncteur $\mathrm{Cr}_d : \pol_d(\A,\V)\to\Sigma\mathbf{Add}_d(\A,\V)$ est réflexif.
\end{pr}

\begin{cor}\label{cor-recollement-cr}
Si $F$ est un foncteur simple de $\fct(\A,\V)$, polynomial de degré $d$, alors $\mathrm{Cr}_d(F)$ est un objet simple de $\Sigma\mathbf{Add}_d(\A,\V)$ et le foncteur $\mathrm{Cr}_d$ induit un isomorphisme de corps ${\rm End}(F)\simeq {\rm End}(\mathrm{Cr}_d(F))$.
\end{cor}

\begin{proof}
L'objet $\mathrm{Cr}_d(F)$ est non nul car  $\mathcal{O}(\mathrm{Cr}_d(F))=cr_d(F)\ne 0$. Le corollaire découle donc directement des propositions~\ref{pr-prolongement-interm} et~\ref{recollement-cr}. 
\end{proof}

Les résultats de structure sur les foncteurs $\mathrm{Cr}_d$ et $\mathcal{O}$ nous permettent d'obtenir dans la proposition suivante des renseignements importants sur la structure des foncteurs $cr_d(F)$ lorsque $F$ est un foncteur simple polynomial de degré~$d$.

\begin{pr}\label{pr-fin-cr} Soit $\E$ une catégorie de Grothendieck et soit $F$ un foncteur simple de $\fct(\A,\E)$, polynomial de degré $d$.
\begin{enumerate}
\item Le foncteur $cr_d(F)$ de $\fct(\A^d,\E)$ est semi-simple : il existe un foncteur simple $S$ de $\fct(\A^d,\E)$, un sous-ensemble $E$ de $\Si_d$, et un isomorphisme
$$ cr_d(F)\simeq \bigoplus_{\sigma\in E}\sigma.S\;.$$
\item Il existe des entiers naturels $n$ et $m$ vérifiant $0<nm\leq d!$ de sorte que l'anneau ${\rm End}(cr_d(F))$ soit isomorphe au produit d'anneaux de matrices $\M_n({\rm End}(S))^m$.
\item Le morphisme d'anneaux ${\rm End}(F)\to {\rm End}(cr_d(F))$ induit par $cr_d$ fait de ${\rm End}(cr_d(F))$ un ${\rm End}(F)$-espace vectoriel de dimension au plus $d!$.
\end{enumerate}
\end{pr}
\begin{proof}
Le foncteur $\mathcal{O}$ est adjoint à gauche au foncteur $\mathcal{L}$, qui commute aux colimites. Par conséquent, $\mathcal{O}$ préserve les objets de type fini. D'après le corollaire \ref{cor-recollement-cr}, $\mathrm{Cr}_dF$ est un simple de $\Sigma\mathbf{Add}_d(\A,\E)$, son image par $\mathcal{O}$ est donc non nulle et de type fini, et possède donc un quotient simple $S$. Par adjonction il existe un morphisme non nul, donc injectif, $\mathrm{Cr}_dF\to\mathcal{L}(S)$. D'où un monomorphisme :
\begin{equation*}
cr_d(F)=\mathcal{O}(\mathrm{Cr}_dF)\hookrightarrow\mathcal{O}\mathcal{L}(S)=\bigoplus_{\tau\in\Si_d}\tau.S.
\end{equation*}
Les foncteurs $\tau.S$ sont des objets simples de $\mathbf{Add}_d(\A,\E)$, tout comme $S$ (l'action de $\tau$ définissant un automorphisme de la catégorie $\mathbf{Add}_d(\A,\E)$). La première assertion en découle. En regroupant les termes isomorphes dans la décomposition de $cr_d(F)$ on obtient un isomorphisme $cr_d(F)\simeq\bigoplus_{i=1}^m \sigma_i.S^{\oplus n}$ pour des permutations $\sigma_i$ convenables, de sorte que $\sigma_i.S$ et $\sigma_j.S$ ne soient jamais isomorphes pour $i\neq j$. Cela fournit la deuxième assertion. 

Pour la troisième assertion, on considère le diagramme commutatif
\[\xymatrix{{\rm End}(F)\ar[d]_\simeq\ar[r] & {\rm End}(cr_d(F))\ar[d]^\simeq \\
{\rm End}(\mathrm{Cr}_d(F))\ar[r] & {\rm End}\,(\mathcal{O}\mathrm{Cr}_d(F))
}\]
dont la flèche horizontale supérieure (resp. inférieure) est induite par le foncteur $cr_d$ (resp. $\mathcal{O}$) et l'isomorphisme vertical de gauche par le corollaire~\ref{cor-recollement-cr}. Combiné à l'isomorphisme d'adjonction ${\rm End}\,(\mathcal{O}\mathrm{Cr}_d(F))\simeq {\rm Hom}(\mathrm{Cr}_d(F),\mathcal{L}\mathcal{O}\mathrm{Cr}_d(F))$ (lemme~\ref{lm-adj-cro}), il montre que la dimension de ${\rm End}(cr_d(F))$ sur le corps ${\rm End}(F)$ est la borne supérieure des $i\in\mathbb{N}$ tels qu'il existe un monomorphisme $\mathrm{Cr}_d(F)^{\oplus i}\hookrightarrow\mathcal{L}\mathcal{O}\mathrm{Cr}_d(F)$. L'existence d'un tel monomorphisme implique, en appliquant le foncteur exact $\mathcal{O}$, celle d'un monomorphisme $cr_d(F)^{\oplus i}\hookrightarrow\mathcal{O}\mathcal{L}cr_d(F)\simeq cr_d(F)^{\oplus d!}$. Comme $cr_d(F)$ est semi-simple, fini et non nul, par ce qui précède, cela entraîne $i\le d!$ et termine la démonstration.
\end{proof}

On remarquera que la proposition~\ref{pr-fin-cr} ne nécessite aucune autre hypothèse que la simplicité et la polynomialité de $F$ : elle s'applique aussi à des foncteurs à valeurs de dimensions infinies de $\F(\A;K)$, contrairement à la plupart des résultats ultérieurs que nous donnerons sur les foncteurs simples.

\section{Produits tensoriels et extension des scalaires}\label{sect-prelim-ptens}

Dans cette section, nous étendons aux catégories de foncteurs générales des résultats classiques sur les représentations simples des produits tensoriels d'algèbres, pour lesquels on pourra se référer à \cite[§\,12]{Bki}. Ces considérations constituent l'une des bases de nos décompositions tensorielles. Elles expliquent la nécessité fréquente d'imposer une hypothèse de dimension finie pour les valeurs des foncteurs simples, ainsi que celle de supposer le corps commutatif de base $K$ algébriquement clos, ou au moins assez gros, en un sens que nous allons préciser (définition~\ref{def-corps-dec}).

Après une partie de généralités sur la notion fondamentale de foncteur \emph{absolument} simple (§\,\ref{par-absimple}), on étudie (§\,\ref{par-simprod}) les foncteurs simples à valeurs de dimensions finies sur une catégorie produit. On termine (§\,\ref{sect-decomp-prim}) par une légère généralisation, correspondant à une source se décomposant en somme directe infinie de catégories additives en un sens approprié, qui s'applique en particulier à une décomposition tensorielle primaire issue de la décomposition primaire des groupes abéliens finis.

Dans toute la section, $\C$ et $\C'$ désignent deux petites catégories. \'Etant donnés un foncteur $F$ de $\F(\C;K)$ et un foncteur $G$ de $\F(\C';K)$, on note $F\boxtimes G$ leur produit tensoriel extérieur, c'est-à-dire le foncteur de $\F(\C\times\C';K)$ tel que $(F\boxtimes G)(c,c'):=F(c)\otimes G(c')$, les produits tensoriels étant pris comme d'habitude sur $K$.

\begin{rem}\label{chgt-base-gal} Une grande partie des résultats de changement de base de cette section s'étend (avec les mêmes démonstrations) à la situation plus intrinsèque suivante. Supposons que la catégorie de Grothendieck $\E$ est $k$-linéaire (où $k$ est un anneau commutatif) et que $A$ est une $k$-algèbre. Les objets $x$ de $\E$ munis d'une action de $A$ (c'est-à-dire d'un morphisme de $k$-algèbres $A\to\mathrm{End}_\E(x)$) forment une catégorie de Grothendieck $\E_A$. On dispose d'un foncteur d'extension des scalaires $\E\to\E_A$, qui est adjoint à gauche au foncteur d'oubli (et qui est exact si $k$ est un corps).

Si $\C$ est une petite catégorie, la catégorie $\F(\C;k)$ a une structure $k$-linéaire canonique pour tout anneau commutatif $k$, et pour toute $k$-algèbre $A$, la catégorie $\F(\C;k)_A$ s'identifie à $\F(\C;A)$. Via cette identification, le foncteur d'extension des scalaires est la post-composition par $-\otimes_k A$.

Il est également possible de donner une construction catégorique abstraite, appelée produit tensoriel, associant à deux catégories de Grothendieck une autre catégorie de Grothendieck d'une manière naturelle, de sorte que le produit tensoriel de $\F(\C;k)$ et $\F(\C';k)$ au-dessus de $k$ s'identifie à $\F(\C\times\C';k)$. Cette notion intrinsèque de produit tensoriel de catégories de Grothendieck, beaucoup plus délicate que l'extension des scalaires, a été introduite et étudiée dans l'article récent de Lowen, Ramos Gonz\'alez et Shoikhet \cite{LRGS}.
\end{rem}

\subsection{Foncteurs absolument simples}\label{par-absimple}

\begin{pr}\label{abs-semi-simple}
Soit $S$ un objet simple de $\F(\C;K)$. Les conditions suivantes sont équivalentes.
\begin{enumerate}
\item\label{it1} Le morphisme canonique $K\to\mathrm{End}_{\F(\C;K)}(S)$ est un isomorphisme.
\item\label{it1bis} Pour toute petite catégorie $\C'$ et tous objets $F$ et $G$ de $\F(\C';K)$, le produit tensoriel par $\mathrm{Id}_S$ induit un isomorphisme:
$$\F(\C';K)(F,G) \xrightarrow{\mathrm{Id}_S\boxtimes-}\F(\C\times\C';K)(S\boxtimes F,S\boxtimes G). $$
\item\label{it2} Pour toute petite catégorie $\C'$ et tout foncteur simple $T$ de $\F(\C';K)$, le produit tensoriel $S\boxtimes T$ est un foncteur simple de $\F(\C\times \C';K)$.
\item\label{it3} Pour toute extension $L$ de  $K$, $S\otimes_K L$ est un objet simple de $\F(\C;L)$. 
\end{enumerate}
\end{pr}

\begin{proof}
\eqref{it1}$\Rightarrow$\eqref{it2}. Via l'isomorphisme canonique de catégories $\F(\C\times\C';K)\simeq\fct(\C',\F(\C;K))$, le foncteur $S\boxtimes T$ prend ses valeurs dans la sous-catégorie pleine $\mathcal{S}$ de $\F(\C;K)$ constituée des foncteurs semi-simples isotypiques de type $S$. Comme celle-ci est stable par sous-quotients, il suffit de montrer que le foncteur qu'on obtient en voyant $S\boxtimes T$ comme objet de $\fct(\C',\mathcal{S})$ est simple. Mais $\mathcal{S}$ est équivalente  (via ${\rm Hom}(S,-)$) à la catégorie des espaces vectoriels sur $K$, et via cette équivalence, $S\boxtimes T$ s'identifie au foncteur $T$ de $\F(\C';K)$, qui est par hypothèse simple, d'où la simplicité de $S\boxtimes T$.
 
  \eqref{it2}$\Rightarrow$\eqref{it3} : soit $\C'$ la catégorie à un objet associée au monoïde multiplicatif $L_\mu$ sous-jacent à $L$. La catégorie $\F(\C;L)$ s'identifie à la sous-catégorie pleine, qui est stable par sous-quotient (donc préserve la simplicité), de $\F(\C\times\C';K)$ constituée des foncteurs sur lesquels les éléments de $K$ à la source (vus comme flèches de $\C'$) agissent par l'homothétie correspondante au but et sur lesquels l'action de $L$ à la source est additive, d'où le résultat en prenant pour $T$ le $K[L_\mu]$-module simple $L$.
  
  \eqref{it3}$\Rightarrow$\eqref{it1} : pour toute extension $L$ de $K$, le foncteur $S\otimes_K L$ prend naturellement ses valeurs dans les modules à gauche {\em libres} sur l'anneau $A_L:={\rm End}(S)\otimes_K L$. Comme par hypothèse $S\otimes_K L$ est simple dans $\F(\C;L)$, on en déduit que tout élément non nul de $A_L$ définit un automorphisme de ce foncteur. Comme une homothétie sur un module libre non nul ne peut être inversible que si son rapport l'est, on en déduit que $A_L$ est un {\em corps}. Supposons que ${\rm End}(S)$ contienne strictement $K$ et considérons un élément $\xi\in{\rm End}(S)\setminus K$ : si $\xi$ est algébrique sur $K$, $A_L$ n'est pas un corps si l'on prend pour $L$ un corps de rupture d'un polynôme irréductible sur $K$ annulant $\xi$, ce qui est absurde. Si $\xi$ est transcendant, $A_{K(x)}$ n'est pas un corps car son élément non nul $\xi\otimes 1-1\otimes x$ n'est pas inversible. On en déduit donc ${\rm End}(S)=K$.
  
  L'implication \eqref{it1bis}$\Rightarrow$\eqref{it1} est évidente (prendre pour $\C'$ la catégorie à un seul morphisme), et l'implication réciproque s'établit comme \eqref{it1}$\Rightarrow$\eqref{it2}.
\end{proof}

La définition suivante reprend, dans le cadre des catégories de foncteurs, la notion de module absolument simple classique en théorie des représentations \cite[§\,3.C, définition~3.42]{CuR}.

\begin{defi}
Un foncteur simple est {\em absolument simple} lorsqu'il vérifie les conditions équivalentes de la proposition \ref{abs-semi-simple}.
\end{defi}

\begin{defi}\label{def-corps-dec} Nous dirons que $K$ est un \emph{corps de décomposition} (resp. \emph{corps de décomposition non additif}) de la catégorie additive $\A$ (resp. de la catégorie $\C$)  si tout foncteur simple à valeurs de dimensions finies de $\mathbf{Add}(\A;K)$ (resp. $\F^{\df}(\C;K)$) est absolument simple.
\end{defi}

\begin{rem}\label{rq-df-cordec} Cette terminologie est motivée par l'observation suivante.

Si $A$ est un anneau, $K$ est un corps de décomposition de la catégorie $\mathbf{P}(A)$ si et seulement si $K$ est un corps de décomposition (\emph{splitting field} en anglais) de la $K$-algèbre $A^{\op}\otimes_{\mathbb{Z}}K$ au sens usuel de la théorie des représentations \cite[§\,7.B]{CuR}, en vertu de la proposition~\ref{additifs-bimodules}.
\end{rem}

\begin{ex}\label{excd}
 Si $k$ est un corps fini, tout surcorps commutatif de $k$ est un corps de décomposition de $\mathbf{P}(k)$. Tout corps commutatif est un corps de décomposition de $\mathbf{P}(\mathbb{Z})$ et de $\mathbf{P}(\mathbb{Q})$.
\end{ex}

\begin{pr}\label{crit-absimpl} 
 \begin{enumerate}
  \item Si $K$ est algébriquement clos, alors $K$ est un corps de décomposition non additif de la catégorie $\C$. Plus généralement,  tout foncteur simple de $\F(\C;K)$ prenant au moins une valeur de dimension finie non nulle est absolument simple.
  \item Si pour tout objet $a$ de $\C$, l'ensemble $\C(a,a)$ est fini, et si $K$ contient toutes les racines de l'unité, alors $K$ est un corps de décomposition non additif de $\C$.
 \end{enumerate}
\end{pr}

\begin{proof}
Soit $F$ un foncteur simple de $\F(\C;K)$ tel qu'il existe un objet $t$ de $\C$ tel que $F(t)$ soit non nul et de dimension finie. Alors ${\rm End}(F)$ est un sous-espace vectoriel de $F(t)$, car $F$ est un quotient de $P_\C^t$ (on a une flèche non nulle $P_\C^t\to F$ et $F$ est simple), donc est de dimension finie sur $K$. De plus, $K$ est inclus dans le centre de ${\rm End}(F)$. Par conséquent, si $K$ est algébriquement clos, ${\rm End}(F)$ est nécessairement réduit à $K$ (cf. par exemple \cite[§\,14, n°1, cor.~1]{Bki}).

La deuxième assertion se déduit de la proposition~\ref{simples-sfin} et de ce qu'un corps contenant assez de racines de l'unité est un corps de décomposition de l'algèbre de n'importe quel groupe fini. Ce dernier résultat constitue un corollaire classique du théorème d'induction de Brauer (voir par exemple \cite[{\em Theorem}~17.1]{CuR}).
\end{proof}

\begin{rem}
 La première propriété de la proposition~\ref{crit-absimpl} tombe en défaut si l'on omet l'hypothèse des valeurs de dimensions finies. Ainsi, si $L$ est un surcorps strict de $K$, le foncteur d'inclusion des $L$-espaces vectoriels de dimensions finies dans les $K$-espaces vectoriels définit un foncteur simple de $\F(L,K)$, mais il n'est pas absolument simple.
\end{rem}

\begin{rem} Nous verrons au corollaire~\ref{cor-corps-dec} que l'on peut, dans la deuxième assertion de la proposition~\ref{crit-absimpl}, s'affranchir de l'hypothèse de finitude sur $\C$ lorsque cette catégorie est additive et que $K$ en est un corps de décomposition.
\end{rem}

Il nous sera parfois utile de ramener l'étude de foncteurs simples à celle de foncteurs absolument simples, vu les bonnes propriétés dont jouissent ces derniers. On peut le faire, au prix d'une extension finie des scalaires et de dévissages, par l'intermédiaire de la propriété suivante.

\begin{pr}\label{pr-corpassegro}
Soit $S$ un foncteur simple de $\F^{\df}(\C;K)$. Il existe une extension finie de corps commutatifs $K\subset L$ telle que le foncteur $S\otimes_KL$ de $\F(\C;L)$ possède une filtration finie dont les sous-quotients sont \emph{absolument} simples.
\end{pr}

\begin{proof} Soit $Z$ le centre du corps ${\rm End}(S)$. Comme $S$ est à valeurs de dimensions finies, ${\rm End}(S)$ est de dimension finie sur $K$ (cf. le début de la démonstration de la proposition~\ref{crit-absimpl}) et donc sur $Z$. 

D'après \cite[§\,14.1, Théorème~1]{Bki}, il existe une extension finie $L$ de $Z$ et un entier $n>0$ tels que ${\rm End}(S)\otimes_Z L\simeq\mathcal{M}_n(L)$ comme $L$-algèbres. Quitte à remplacer $L$ par  sa clôture normale comme extension \emph{de $K$}, on peut supposer que l'extension finie $K\subset Z\subset L$ est normale.

Soient $A$ la $L$-algèbre commutative de dimension finie $Z\otimes_K L$ et $J$ son radical : comme l'extension $K\subset L$ est normale, la $L$-algèbre semi-simple $A/J$ est isomorphe à $L^m$ pour un entier $m>0$. Il s'ensuit que le quotient de la $L$-algèbre artinienne
\[R:={\rm End}(S)\otimes_K L\simeq {\rm End}(S)\otimes_Z(Z\otimes_K L)\]
par son radical est isomorphe à ${\rm End}(S)\otimes_Z L^m\simeq\big({\rm End}(S)\otimes_Z L\big)^m\simeq\M_n(L)^m$. Autrement dit, tous les $R$-modules simples sont absolument simples (sur $L$), et le $R$-module $R$ possède une filtration finie dont les quotients sont absolument simples.

En utilisant la correspondance entre sous-foncteurs de $S\otimes_K L$ dans $\F(\C;L)$ et sous-$R$-modules de $R$ (cf. le démonstration de la proposition~\ref{abs-semi-simple}), on en déduit le résultat souhaité.
\end{proof}

\subsection{Foncteurs simples sur une catégorie produit}\label{par-simprod}

La proposition ci-dessous constitue une variante fonctorielle de \cite[§\,12.1, proposition~2]{Bki}. 

\begin{pr}\label{simples-cat-prod}
 Soit $F$ un foncteur simple de $\F(\C\times\C';K)$. On suppose qu'il existe un objet $(x,y)$ de $\C\times \C'$ tel que  $F(x,y)$ soit de dimension finie non nulle sur $K$. Alors il existe des foncteurs simples $S$ de $\F(\C;K)$ et $T$ de $\F(\C';K)$, uniques à isomorphisme près, tels que $F$ soit isomorphe à un quotient de $S\boxtimes T$.
\end{pr}
\begin{proof} 
Si $S$ et $T$ existent, alors pour tous les objets $t$ de $\C$ et $u$ de $\C'$, le foncteur $F(t,-)$, resp. $F(-,u)$ est semi-simple isotypique de type $S$, resp. $T$. Ceci prouve l'unicité de $S$ et $T$. Montrons leur existence. 

Notons que $\{(x,y)\}$ est un support et un co-support (cf. définition~\ref{def-supp}) de $F$. Il s'ensuit que $\{x\}$ (resp. $\{y\}$) est un support et un co-support de $F(-,u)$
(resp. $F(t,-)$) pour tout objet $u$ de $\C'$ (resp. $t$ de $\C$). Soit $M$ un $K[{\rm End}_\C(x)]$-module simple contenu dans $F(x,y)$ (il en existe car
$F(x,y)$ est de dimension finie non nulle sur $K$) : alors $F(-,y)$ contient le foncteur simple $S$ de $\F(\C;K)$ obtenu en appliquant à $M$ le prolongement intermédiaire associé à la restriction de $\C$ à sa sous-catégorie pleine constituée du seul objet $t$ (cf. propositions~\ref{restr-fct-refl} et~\ref{pr-prolongement-interm}), car le foncteur de prolongement intermédiaire préserve les
monomorphismes, et l'hypothèse qu'un ensemble d'objets d'un foncteur en constitue un support et un co-support est équivalente à dire que ce foncteur
s'identifie au prolongement intermédiaire de sa restriction à la sous-catégorie pleine déterminée par cet ensemble.

Ainsi, vu comme objet de $\fct(\C',\F(\C;K))(\simeq\F(\C\times\C';K))$, $F$ prend une valeur contenant un objet simple $S$ de $\F(\C;K)$. La simplicité de $F$ montre alors que {\em toutes} les valeurs de ce foncteur sont des objets semi-simples isotypiques de type $S$ dans $\F(\C;K)$, puisque $\{y\}$ constitue un support de $F$ (vu dans $\fct(\C',\F(\C;K))$). On obtient de même un simple $T$ de $\F(\C';K)$ tel que $F(c,-)$ soit un foncteur semi-simple isotypique de type $T$ pour tout objet $c$ de $\C$. 

Pour voir que $F$ est un quotient de $S\boxtimes T$, on raisonne comme au début de la démonstration de la proposition~\ref{abs-semi-simple}: la sous-catégorie $\mathcal{S}$ de $\F(\C;K)$ des foncteurs semi-simples isotypiques de type $S$ est équivalente à ${\rm End}(S)^{\op}\Md$. Il existe donc un foncteur simple $T'$ de $\F(\C';{\rm End}(S)^{\op})$ tel que $F\simeq S\underset{{\rm End}(S)^{\op}}{\boxtimes} T'$. L'image de $T'$ dans $\F(\C';K)$ (par le foncteur de restriction des scalaires au but) contient manifestement $T$, de sorte qu'on dispose dans $\F(\C';{\rm End}(S)^{\op})$ d'un épimorphisme $T\otimes_K {\rm End}(S)\twoheadrightarrow T'$, d'où un épimorphisme $S\boxtimes T\twoheadrightarrow F$.
\end{proof}

\begin{rem}
 Le résultat tombe en général en défaut si l'on omet l'hypothèse d'une valeur de dimension finie. Pour le voir, il suffit de trouver des anneaux commutatifs $A$ et $B$ tels qu'existe, en posant $A_K:=A\otimes_\mathbb{Z}K$ et $B_K:=B\otimes_\mathbb{Z}K$, un $A_K\otimes B_K$-module simple qui ne s'obtienne pas comme produit tensoriel sur $K$ d'un $A$-module simple et d'un $B$-module simple, et d'utiliser la proposition~\ref{additifs-bimodules} (et l'équivalence $\mathbf{P}(A\times B)\simeq\mathbf{P}(A)\times\mathbf{P}(B)$).
 
 Un exemple d'une telle situation s'obtient en posant $A:=\mathbb{Z}[x]$ et en prenant pour $B$ l'algèbre de polynômes $\mathbb{Z}[t_P]$, où les indéterminées $t_P$ sont indexées par les polynômes irréductibles unitaires $P$ de $K[X]$. On fait du corps $K(x)$ un $A_K\otimes B_K$-module via les morphismes de $K$-algèbres $K[x]\hookrightarrow K(x)$ (inclusion canonique) et $K[t_P]\to K(x)$ envoyant $t_P$ sur $1/P$. Ce module est simple car le morphisme $A_K\otimes B_K\to K(x)$ ainsi défini est surjectif. Mais $K(x)$ ne contient aucun $K[x]$-module simple.
\end{rem}

\begin{cor}\label{cor-cas-sympa-tens}
Supposons que $K$ est un corps de décomposition non additif de $\C$. Alors les simples de $\F^{\df}(\C\times \C';K)$ sont les bifoncteurs de la forme $S\boxtimes T$ où $S$ et $T$ sont des simples de $\F^{\df}(\C;K)$ et $\F^{\df}(\C';K)$ respectivement. De plus, deux simples $S\boxtimes T$ et $S'\boxtimes T'$ sont isomorphes si et seulement si $S\simeq S'$ et $T\simeq T'$.
\end{cor}

\begin{proof} Cela découle des propositions~\ref{abs-semi-simple} et~\ref{simples-cat-prod}.
\end{proof}

\subsection{Application : décomposition tensorielle primaire}\label{sect-decomp-prim}

Cette section contient une généralisation (directe) aux catégories additives des décompositions classiques des algèbres données par des familles d'idempotents orthogonaux \cite[§\,6A]{CuR}. Nous obtenons en particulier la proposition~\ref{dec-prim}, qui s'appliquera notamment aux foncteurs simples antipolynomiaux,   permettant de raffiner la décomposition tensorielle du corollaire~\ref{cor-tens-St1} ci-après.

Appelons {\em idempotent} de $\A$ un foncteur additif $e:\A\to \A$ qui est l'identité sur les objets, et tel que $e^2=e$.
On note alors $e\A$ la catégorie ayant les mêmes objets que $\A$ avec ${\rm Hom}_{e\A}(x,y)=e{\rm Hom}_{\A}(x,y)$ et avec la loi de composition
$(ef)\circ (eg)= e(f\circ g)$. Cette loi est bien définie, $e\A$ est une catégorie additive, et on a un foncteur quotient canonique 
$$\pi_e:\A\to e\A\;.$$
Une \emph{famille localement finie d'idempotents orthogonaux} de $\A$ est une famille d'idempotents $(e_i)_{i\in I}$ de $\A$ telle que:
\begin{enumerate}
\item $e_ie_j=0=e_je_i$ si $i\ne j$ ;
\item\label{item2} pour tous objets $x,y$ de $\A$, seul un nombre fini d'idempotents induisent une application non nulle $e_i:\A(x,y)\to \A(x,y)$ ;
\item $\sum_{i\in I} e_i =1$ (la somme signifie la somme sur les morphismes, et l'identité sur les objets. Cette somme éventuellement infinie a un sens
car elle est en fait finie à cause de la condition précédente). 
\end{enumerate}
Soit $(\A_i)_{i\in I}$ une famille de catégories additives. La somme $\bigoplus_{i\in I}\A_i$ est la sous-catégorie pleine de $\prod_{i\in I}\A_i$, dont les objets sont les familles $(x_i)_{i\in I}$ telles que tous les $x_i$ sont nuls sauf un nombre fini d'entre eux. 

\begin{lm}
Pour toute famille $(e_i)_{i\in I}$ localement finie d'idempotents orthogonaux de $\A$, le foncteur
$$\prod_{i\in I}\pi_{e_i}:\A\to\prod_{i\in I} e_i\A$$
est à valeurs dans la sous-catégorie $\bigoplus_{i\in I} e_i\A$, et pour toute catégorie abélienne $\V$, il induit une équivalence de catégories:
$$\fct\left(\bigoplus_{i\in I} e_i\A,\V\right)\simeq\fct(\A,\V)\;.$$
\end{lm}
\begin{proof}
La deuxième condition dans la définition des familles localement finies d'idempotents orthogonaux assure que pour tout objet $x$ de $\A$, et pour tous les $i$ sauf un nombre fini, $\pi_{e_i}(x)$ est nul dans $e_i\A$. Ainsi $\prod_{i\in I}\pi_{e_i}$ est à valeurs dans $\bigoplus_{i\in I} e_i\A$. La troisième condition assure que $\prod_{i\in I}\pi_{e_i}$ est pleinement fidèle. Il induit donc une équivalence de catégories entre $\A$ et son image $\Delta$. L'équivalence de catégories du lemme découle alors du fait que tout objet de $\bigoplus_{i\in I} e_i\A$ est facteur direct d'un objet de $\Delta$.
\end{proof}

\begin{pr}\label{prop-dec-ortho}
Soit  $(e_i)_{i\in I}$ une famille localement finie d'idempotents orthogonaux de $\A$. On suppose donnés, pour tout $i\in I$, des foncteurs $S_i$ et $T_i$ de $\F(e_i\A;K)$, de sorte que seul un nombre fini des $S_i$ et des $T_i$ soient non isomorphes au foncteur constant en $K$. 
\begin{enumerate}
\item Le foncteur $\bigotimes_{i\in I}\pi_{e_i}^*S_i$ est absolument simple si et seulement si les foncteurs $S_i$ sont absolument simples. 
\item Si on a un isomorphisme entre deux foncteurs absolument simples :
$$\bigotimes_{i\in I}\pi_{e_i}^*S_i\simeq \bigotimes_{i\in I}\pi_{e_i}^*T_i$$
alors pour tout $i\in I$, les foncteurs $S_i$ et $T_i$ sont isomorphes.
\end{enumerate}
Enfin, $K$ est un corps de décomposition non additif de $\A$ si et seulement si c'est un corps de décomposition non additif de chaque catégorie $e_i\A$. Dans ce cas, tous les foncteurs simples de $\F^{\df}(\A;K)$ sont isomorphes à un produit tensoriel du type ci-dessus.
\end{pr}

\begin{proof} Le fait que $K$ est un corps de décomposition non additif des $e_j\A$ si c'est un corps de décomposition non additif de $\A$ découle de ce que $\F^{\df}(e_j\A;K)$ est une sous-catégorie pleine stable par sous-quotient de $\F^{\df}(\A;K)$. Les autres assertions découlent du lemme précédent et des propositions~\ref{abs-semi-simple} et~\ref{simples-cat-prod}, lorsque la famille d'idempotents est {\em finie}.
 Le cas général s'en déduit car tout foncteur de type fini sur une catégorie somme directe de catégories additives se factorise par la projection sur la somme directe d'une sous-famille finie.
\end{proof}

Si $V$ est un groupe abélien et $p$ un nombre premier, on note $_{(p)}V$ le sous-groupe de torsion $p$-primaire de $V$. On note $_{(p)}\A$ la sous-catégorie de $\A$ avec les mêmes objets que $\A$ et dont les groupes abéliens de morphismes sont donnés par $(_{(p)}\A)(x,y)={_{(p)}(\A(x,y))}$. On a un foncteur canonique (qui est l'identité sur les objets et qui envoie un morphisme sur sa composante $p$-primaire):
$$\pi_p:\A\to {_{(p)}\A}\;.$$

\begin{pr}\label{dec-prim}
Soit $S$ un foncteur de $\F(\A;K)$. Si les groupes abéliens de morphismes $\A(x,y)$ sont finis pour tous objets $x$ et $y$ de $\A$ et que $K$ est un corps de décomposition non additif de $\A$, alors $S$ est simple si et seulement s'il existe des simples $S_p$ de $\F({_{(p)}}\A;K)$, pour tout nombre premier $p$, dont seul un nombre fini sont non isomorphes à $K$, tels que
\[S\simeq \bigotimes_p \pi_p^*S_p\;.\]
De plus, les $S_p$ sont uniques à isomorphisme près.
\end{pr}

\begin{proof}
D'après la proposition \ref{pr-dim-finie}, tous les foncteurs simples de $\F(\A;K)$ sont à valeurs de dimensions finies. 
Par conséquent, le résultat découle de la proposition~\ref{prop-dec-ortho} et de la décomposition primaire des groupes abéliens finis.
\end{proof}

\part{Décompositions tensorielles à la Steinberg}\label{part-resfond}

Cette partie présente les résultats théoriques principaux de l'article. Pour les foncteurs simples d'une petite catégorie additive vers les espaces vectoriels de dimensions finies sur un corps assez gros, ils se traduisent par des décompositions tensorielles qui en facilitent grandement l'étude.

\section{Décomposition tensorielle globale}\label{section-generique}

Le but de cette section, inspirée de Steinberg \cite{RSt}, est de ramener l'étude des foncteurs simples à celle de foncteurs simples de natures très différentes : les foncteurs simples polynomiaux, et des foncteurs simples que nous nommons {\it antipolynomiaux}, dont nous commençons par donner la définition et des propriétés élémentaires.

\subsection{Foncteurs antipolynomiaux}\label{ssect-antipol}

\begin{defi}\label{def-CategKtriv}
Une petite categorie additive $\B$ est dite \emph{$K$-triviale} si elle vérifie l'une des deux conditions équivalentes suivantes. 
\begin{enumerate}
\item Pour tous $x$ et $y$ de $\B$, le groupe abélien $\B(x,y)$ est fini et $K\otimes_{\mathbb{Z}}\B(x,y)=0$.
\item Pour tout $x$ de $\B$, le groupe abélien $\B(x,x)$ est fini et
$K\otimes_{\mathbb{Z}}\B(x,x)=0$. 
\end{enumerate}
\end{defi}

La deuxième condition de la définition \ref{def-CategKtriv} est clairement une conséquence de la première, et réciproquement la première condition s'obtient à partir de la deuxième qu'on applique à $x\oplus y$.

\begin{defi}\label{def-antipol} Un foncteur $F$ de $\F(\A;K)$ est dit \emph{antipolynomial} s'il existe une catégorie $K$-triviale $\B$ et un foncteur additif $\pi:\A\to \B$ tel que $F$ appartienne à l'image essentielle du foncteur de précomposition $\pi^* : \F(\B;K)\to\F(\A;K)$.
\end{defi}

La manipulation des foncteurs antipolynomiaux sera plus aisée en utilisant les idéaux de la catégorie additive $\A$.  
Rappelons (cf. Mitchell \cite[§\,3, p.~18]{Mi72}, Street \cite{Street}) qu'un {\em idéal} $\I$ de $\A$ est un sous-foncteur de ${\rm Hom}_\A : \A^{\op}\times\A\to\mathbf{Ab}$.
Autrement dit, $\I$ consiste en la donnée de sous-groupes $\I(x,y)$ des $\A(x,y)$ stables par composition à droite et à gauche par des morphismes quelconques.
Un tel idéal détermine une catégorie additive $\A/\I$ ayant les mêmes objets que $\A$ et dont les morphismes sont donnés par $(\A/\I)(x,y):=\A(x,y)/\I(x,y)$,
la composition étant induite par celle de $\A$. On dispose ainsi d'un foncteur additif canonique plein et essentiellement surjectif $\pi_\I : \A\to\A/\I$, qui est l'identité sur les objets.
Réciproquement, tout foncteur additif depuis $\A$ détermine un idéal en considérant les noyaux des morphismes induits entre groupes abéliens de flèches.
Nous noterons parfois simplement $\I\triangleleft\A$ pour signifier que $\I$ est un idéal de $\A$.

\begin{defi}\label{df-ideaux}
 Soit $\I$ un idéal de $\A$. Nous dirons que $\I$ est {\em $K$-cotrivial} si la catégorie $\A/\I$ est $K$-triviale.
\end{defi}

\begin{ex}\label{rem-ideal}
Les idéaux bilatères d'un anneau $A$ sont en bijection avec les idéaux de la catégorie $\mathbf{P}(A)$. La bijection envoie un idéal $I$ sur le foncteur $I.{\rm Hom}_{\mathbf{P}(A)}$. Elle induit une bijection entre les idéaux bilatères $I$ de $A$ tels que $A/I$ est fini et $K\otimes_{\mathbb{Z}} A/I$ est nul (que nous appellerons également \emph{idéaux $K$-cotriviaux de $A$}) et les idéaux $K$-cotriviaux de $\mathbf{P}(A)$. 
\end{ex}

Les idéaux de $\A$ forment un ensemble ordonné par l'inclusion, on y dispose d'une notion d'intersection et de somme. La propriété suivante est immédiate.

\begin{pr}\label{prop-ev-cotriv}
\begin{enumerate}
\item Tout idéal de $\A$ contenant un idéal $K$-cotrivial est $K$-cotrivial.
\item Toute intersection finie d'idéaux $K$-cotriviaux de $\A$ est un idéal $K$-cotrivial.
\end{enumerate}
\end{pr}

\begin{rem}\label{ex-pas-de-cotriv} Dans un certain nombre de situations courantes, $\A$ ne possède pas d'autre idéal $K$-cotrivial que $\A$. C'est notamment le cas si la catégorie additive $\A$ est linéaire sur un corps $L$, lorsque $L$ est infini ou lorsque $\mathrm{car}(L)=\mathrm{car}(K)$.
\end{rem}

\begin{pr}\label{pr-antipol} Un foncteur $F$ de $\F(\A;K)$ est antipolynomial si et seulement s'il existe un idéal $K$-cotrivial $\I$ de $\A$ tel que $F$ appartienne à l'image essentielle du foncteur de précomposition $\pi_\I^* : \F(\A/\I;K)\to\F(\A;K)$.
\end{pr}
\begin{proof}
L'équivalence provient du fait que pour tout foncteur additif $\pi:\A\to \B$ de but $K$-trivial, l'idéal $\I=\ker \pi$ est $K$-cotrivial et $\pi$ factorise par le quotient $\pi_\I:\A\to \A/\I$.
\end{proof}

\begin{pr}\label{prel-antipol}
\begin{enumerate}
\item Les foncteurs antipolynomiaux simples sont à valeurs de dimensions finies.
\item La classe des foncteurs antipolynomiaux de $\F(\A;K)$ est stable par sous-quotients, sommes directes finies, et par produit tensoriel.
\item Tout foncteur \ph\ et antipolynomial de $\F(\A;K)$ est constant.
\end{enumerate}
\end{pr}

\begin{proof} La première assertion résulte de la proposition \ref{pr-dim-finie}. 
La proposition~\ref{prop-ev-cotriv} montre que tout ensemble \emph{fini} de foncteurs antipolynomiaux de $\F(\A;K)$ appartient à l'image essentielle de $\pi_\I^*$ pour un idéal $K$-cotrivial convenable $\I$ de $\A$. La deuxième assertion résulte donc de ce que l'image essentielle du foncteur $\pi_\I^*$ est stable par sommes directes, produits tensoriels, et sous-quotients (ce dernier fait venant de ce que $\pi_I$ est plein et essentiellement surjectif, en utilisant par exemple \cite[Prop.~A.2]{V2008}).
Enfin, la troisième assertion résulte de la proposition~\ref{add-pol-cst}.
\end{proof}

\subsection{Résultats de décomposition}\label{sous-section-generique} La proposition suivante, qu'on établira au §\,\ref{dm-thglob2}, montre comment construire des foncteurs simples à partir d'un idéal $K$-cotrivial. Elle peut aussi s'appliquer à des situations plus générales et, contrairement aux autres énoncés de cette section, elle ne nécessite aucune hypothèse de valeurs de dimensions finies.

On note dans ce qui suit $\psi_\I : \A\to\A/\I\times\A$ le foncteur dont les composantes sont la projection $\pi_\I : \A\to\A/\I$ et l'identité de $\A$.

\begin{pr}\label{pr-simglobi}  Soit $\I\triangleleft\A$ tel que $(\A/\I)(x,x)\otimes_\mathbb{Z} K=0$ pour tout objet $x$ de $\A$. Alors la restriction du foncteur de précomposition $\psi_\I^* : \F(\A/\I\times\A;K)\to\F(\A;K)$ à la sous-catégorie pleine des bifoncteurs de $\F(\A/\I\times\A;K)$ \phs\ par rapport à la deuxième variable est pleinement fidèle, et son image essentielle est stable par sous-quotient. En particulier, $\psi_\I^*$ envoie un foncteur simple de $\F(\A/\I\times\A;K)$ \ph\ par rapport à la première variable sur un foncteur simple de $\F(\A;K)$.
\end{pr}

Tous les résultats de cette section se déduiront du précédent et du théorème suivant, qu'on démontrera au §\,\ref{dm-thglob}.

\begin{thm}\label{th-glob-tf} Soit $F$ un foncteur de type fini (ou de type co-fini) de $\F^{\df}(\A;K)$. Il existe un bifoncteur $B$ de $\F^\df(\A\times\A;K)$ tel que :
\begin{enumerate}
\item $B$ est antipolynomial par rapport à la première variable, c'est-à-dire que pour tout $x$, $B(-,x)$ est un foncteur antipolynomial de $\F(\A;K)$ ;
\item $B$ est \ph\ par rapport à la deuxième variable, c'est-à-dire que pour tout $x$, $B(x,-)$ est un foncteur \ph\ de $\F(\A;K)$ ;
\item $F$ est isomorphe à la précomposition de $B$ par la diagonale $\A\to\A\times\A$.
\end{enumerate}

De plus, $B$ est unique à isomorphisme près.
\end{thm}

Cet énoncé et sa démonstration s'inspirent de résultats de Steinberg \cite[théorèmes~6 et~8]{RSt} mentionnés dans l'introduction.

En combinant le théorème~\ref{th-glob-tf} à la proposition~\ref{pr-phpol-lf}, on obtient le résultat suivant, qui s'applique notamment aux foncteurs de longueur finie de $\F^{\df}(\A;K)$.

\begin{cor}\label{cor-tftcf} Tout foncteur de type fini et de type cofini de $\F^{\df}(\A;K)$ s'exprime, de façon unique à isomorphisme près, comme la composée de la diagonale $\A\to\A\times\A$ et d'un bifoncteur de $\F^\df(\A\times\A;K)$ antipolynomial par rapport à la première variable et polynomial par rapport à la deuxième.
\end{cor}

\begin{thm}\label{th-simglob} Un foncteur $F$ de $\F^{\df}(\A;K)$ est simple si et seulement s'il est isomorphe à la composée de la diagonale $\A\to\A\times\A$ et d'un bifoncteur simple de $\F^{\df}(\A\times\A;K)$ antipolynomial par rapport à la première variable et polynomial par rapport à la deuxième. De plus, ce bifoncteur est uniquement déterminé par $F$, à isomorphisme près.
\end{thm}

\begin{proof} Le fait qu'une telle composée soit simple découle de la proposition~\ref{pr-simglobi}. Le reste de l'énoncé découle du corollaire~\ref{cor-tftcf}.
\end{proof}

En appliquant le corollaire~\ref{cor-cas-sympa-tens} on en tire l'importante conséquence suivante, qui ramène l'étude des simples de $\F^{\df}(\A;K)$ à celle des simples polynomiaux et celle des simples antipolynomiaux, au moins lorsque $K$ est assez gros -- par exemple si $K$ contient toutes les racines de l'unité, par la proposition~\ref{crit-absimpl}.

\begin{cor}\label{cor-tens-St1} Supposons que $K$ est un corps de décomposition non additif de $\A/\I$ pour tout idéal $K$-cotrivial $\I$ de $\A$.

Alors un foncteur de $\F^{\df}(\A;K)$ est simple si et seulement s'il est isomorphe au produit tensoriel d'un foncteur simple polynomial de $\F^{\df}(\A;K)$ et d'un foncteur simple antipolynomial de $\F(\A;K)$. De plus, une telle décomposition tensorielle est unique à isomorphisme près.
\end{cor}

\begin{rem}\label{rq-disc-Stein}
Lorsque $\A=\mathbf{P}(A)$ pour un anneau $A$, un foncteur donne par évaluation sur chaque $A^n$ une représentation de $\GL_n(A)$, et on peut donc penser à $\F(A,K)$ comme un avatar des représentations $K$-linéaires de $\GL_\infty(A)$. Avec ce point de vue, le corollaire~\ref{cor-tens-St1} pour $\A=\mathbf{P}(A)$ est formellement analogue à la décomposition tensorielle de Steinberg \cite{RSt} (la formule \eqref{eqn-2} de l'introduction de notre article, page~\pageref{eqn-2}) ou aux résultats similaires de Bass-Milnor-Serre \cite{BMS}. Cependant, contrairement à ces résultats, le corollaire \ref{cor-tens-St1} ne requiert aucune hypothèse sur l'anneau $A$. Ceci est à rapprocher de phénomènes de \og stabilisation\fg{} classiques : il est plus facile de donner une présentation du groupe $\GL_\infty(A)$ que des groupes $\GL_n(A)$, $n<\infty$. 
Il ne semble donc pas qu'on puisse déduire directement les résultats de \cite{RSt,BMS} de ceux pour $\F(A,K)$ ou inversement. \`A cet égard, la situation est différente de l'article \cite{Harman} où Harman utilise les résultats de \cite{BMS} pour obtenir des résultats de structure sur des catégories de foncteurs dont la source est les groupes abéliens, mais avec des morphismes modifiés (monomorphismes munis d'un scindement).
\end{rem}

Dans certains cas, les résultats précédents prennent une forme plus forte, du fait de l'absence de foncteurs \phs\ ou antipolynomiaux non constants, par exemple. Nous donnons quelques énoncés en ce sens ci-dessous.

\begin{cor}\label{cor-pasdecotr} Supposons que $\A$ n'a pas d'autre idéal $K$-cotrivial que $\A$. Alors tout foncteur de type fini (ou de type cofini) de $\F^{\df}(\A;K)$ est \ph. Tout foncteur de type fini et de type cofini de $\F^{\df}(\A;K)$ est polynomial.
\end{cor}

\begin{cor}\label{cor-tfap} Supposons que, pour tout objet $x$ de $\A$, le groupe abélien $\A(x,x)\otimes_\mathbb{Z}K$ est nul. Alors tout foncteur de type fini (ou de type cofini) de $\F^{\df}(\A;K)$ est  antipolynomial.
\end{cor}

\begin{proof} Combiner le théorème~\ref{th-glob-tf} et la proposition~\ref{add-pol-cst}.
\end{proof}

\begin{cor}\label{cor-dfcst} Supposons que $\A$ est $L$-linéaire, où $L$ est un corps commutatif infini, avec $\mathrm{car}(L)\ne\mathrm{car}(K)$. Alors la catégorie $\F^{\df}(\A;K)$ est réduite aux foncteurs constants.
\end{cor}

\begin{proof} Cela découle du corollaire~\ref{cor-tfap} et de la remarque~\ref{ex-pas-de-cotriv}.
\end{proof}

\begin{ex}\label{ex-St-Q} Si $k$ est un corps \emph{infini}, tout foncteur simple de $\F^{\df}(k,K)$ est polynomial. Si de plus $\mathrm{car}(k)\ne\mathrm{car}(K)$, le foncteur constant $K$ est le seul foncteur simple de $\F^{\df}(k,K)$, et $\F^{\df}(k,K)$ est réduite aux foncteurs constants.

Steinberg \cite[\emph{theorem} p.~343]{RSt} a montré que toute représentation complexe de dimension finie d'un groupe $\SL_n(\mathbb{Q})$ (pour $n\in\mathbb{N}$) est polynomiale : la polynomialité des foncteurs simples de $\F^{\df}(\mathbb{Q},\mathbb{C})$ est à mettre en parallèle de ce résultat.
\end{ex}

\begin{cor}\label{cor-glob-car0} Supposons $K$ de caractéristique $0$. Alors tout foncteur de $\F^{\df}(\A;K)$ se décompose en une somme directe
\[\bigoplus_{d\in\mathbb{N}}\delta^* B_d\]
où $\delta : \A\to\A\times\A$ désigne la diagonale, et $B_d$ est un bifoncteur de $\F(\A\times\A;K)$ antipolynomial par rapport à la première variable et polynomial homogène de degré $d$ par rapport à la deuxième. De plus, les bifoncteurs $B_d$ sont uniques à isomorphisme près.
\end{cor}

\begin{proof} Lorsque $F$ est de type fini, le résultat se déduit du théorème~\ref{th-glob-tf} et de la proposition~\ref{pr-polph-car0}. Le cas général s'y ramène en écrivant $F$ comme la colimite de ses sous-foncteurs de type fini.
\end{proof}

\subsection{Démonstration de la proposition~\ref{pr-simglobi}}\label{dm-thglob2}

Celle-ci repose sur le lemme suivant sur les fonctions polynomiales (à la Eilenberg-MacLane).

\begin{lm}\label{lm-polbi}
Soient $E$ un $K$-espace vectoriel, $V$ un groupe abélien et $U$ un sous-groupe de $V$ tel quel $K\otimes_{\mathbb{Z}} V/U=0$. Alors la restriction de la fonction $E^{V/U\times V}\to E^V$ induite par l'application linéaire canonique $V\to V/U\times V$ (dont les composantes sont la projection et l'identité) à l'espace vectoriel des fonctions $V/U\times V\to E$ qui sont polynomiales par rapport à la deuxième variable est injective.
\end{lm}

\begin{proof} Montrons d'abord que la restriction de $V$ à $U$ définit une application linéaire {\em injective} $\mathrm{Pol}(V,E)\to\mathrm{Pol}(U,E)$. Supposons en effet qu'il existe une fonction polynomiale non nulle $f : V\to E$ dont la restriction à $U$ soit nulle ; soit $n\geq 0$ le degré de $f$. Alors la $n$-déviation $\delta_n(f) : V^n\to E$ est une fonction additive par rapport à chacune des $n$ variables et non identiquement nulle. On voit par récurrence descendante sur l'entier $r\in\{0,\dots,n\}$ que $\delta_n(f)(a_1,\dots,a_n)=0$ si au moins $r$ des éléments $a_1,\dots,a_n$ de $V$ appartiennent à $U$ : pour $r=n$ cela résulte de la nullité de $f$ sur $U$, et le pas de la récurrence s'obtient en utilisant que toute fonction {\em additive} $V/U\to E$ est nulle, en raison de la nullité de $K\otimes_{\mathbb{Z}} V/U$. Pour $r=0$, on voit que $\delta_n(f)$ est identiquement nulle, contradiction qui établit l'injectivité souhaitée.

Considérons maintenant une fonction $f : V/U\times V\to E$ polynomiale par rapport à la deuxième variable et telle que $f(\bar{x},x)=0$ pour tout $x\in V$, où $\bar{x}$ désigne la classe de $x$ modulo $U$. Pour tout $x\in V$, la fonction $V\to E\quad t\mapsto f(\bar{x},x+t)$ est polynomiale et nulle sur $U$, elle est donc identiquement nulle d'après ce qui précède, d'où le lemme.
\end{proof}

\begin{lm}\label{lm-vraimev} Soient $\I\triangleleft\A$ et $\B$ la sous-catégorie pleine de $\A/\I\times\A$ image essentielle du foncteur $\psi_\I : \A\to\A/\I\times\A$. Alors le foncteur de restriction $\F(\A/\I\times\A;K)\to\F(\B;K)$ est une équivalence de catégories.
\end{lm}

\begin{proof} En effet, tout objet $(x,y)$ de $\A/\I\times\A$ est facteur direct de l'objet $(x\oplus y,x\oplus y)$, qui appartient à $\B$.
\end{proof}

\begin{proof}[Démonstration de la proposition~\ref{pr-simglobi}]
Soient $X$ un foncteur de $\F(\A/\I\times\A;K)$ \ph\ par rapport à la deuxième variable et $F$ un sous-foncteur de $\psi_\I^*X$ (dans $\F(\A;K)$). On va montrer que $F$ est isomorphe  à $\psi_\I^*Y$ pour un sous-foncteur $Y$ de $X$. À cette fin, considérons, pour des objets $a$ et $b$ de $\A$, le diagramme commutatif ensembliste (en traits pleins)
\[\xymatrix{\A(a,b)\ar[r]^-f\ar[d]_-j & {\rm Hom}_K(F(a),F(b))\ar[r]^-\alpha & {\rm Hom}_K(F(a),\psi_\I^*X(b)) \\
(\A(a,b)/\I(a,b))\times\A(a,b)\ar[r]_-g\ar@{-->}[ru] & {\rm Hom}_K(\psi_\I^*X(a),\psi_\I^*X(b))\ar[ru]_-\beta & 
}\]
où les flèches $f$ et $g$ sont données par la fonctorialité de $F$ et $X$ respectivement, $\alpha$ et $\beta$ sont induites par les inclusions $F(b)\subset\psi_\I^*X(b)$ et $F(a)\subset\psi_\I^*X(a)$ respectivement, et $j$ est induite par $\psi_\I$. L'existence d'une fonction ensembliste en pointillé faisant commuter le diagramme, et naturelle en $a$ et $b$, est équivalente au fait que $F$ est isomorphe à la précomposition par $\psi_\I$ d'un sous-foncteur de $X$. Comme $\alpha$ est un monomorphisme de $K$-espaces vectoriels, l'existence de la flèche en pointillé équivaut à la nullité de la composée de $\beta g$ et de la projection de ${\rm Hom}_K(F(a),\psi_\I^*X(b))$ sur ${\rm Coker}\,\alpha$. Or la précomposition de cette composée avec $\A(a,b)\xrightarrow{j}(\A(a,b)/\I(a,b))\times\A(a,b)$ est nulle (grâce à la commutativité du diagramme), de sorte que l'hypothèse faite sur $X$ et le lemme~\ref{lm-polbi} garantissent l'existence de la flèche en pointillé. L'unicité et la naturalité en $a$ et $b$ de la factorisation obtenue découle de l'injectivité de $\alpha$. Ainsi, $F$ est bien isomorphe à $\psi_\I^*Y$ pour un sous-foncteur $Y$ de $X$.

Le lemme~\ref{lm-vraimev} montre que le foncteur $\psi_\I^* : \F(\A/\I\times\A;K)\to\F(\A;K)$ est fidèle.

Montrons que la restriction aux foncteurs \phs\ par rapport à la deuxième variable de $\psi_\I^*$ est pleine. Considérons pour cela une transformation naturelle $\psi_\I^*X\to\psi_\I^*Y$, où $X$ et $Y$ sont des bifoncteurs de $\F(\A/\I\times\A;K)$ \phs\ par rapport à la deuxième variable : c'est une collection d'applications linéaires $\psi_\I^*X(a)\to\psi_\I^*Y(a)$, pour $a$ objet de $\A$, telle que le diagramme ensembliste évident
\begin{equation}\label{dctf1}
\xymatrix{\A(a,b)\ar[r]\ar[d] & {\rm Hom}_K(\psi_\I^*X(a),\psi_\I^*X(b))\ar[d] \\
{\rm Hom}_K(\psi_\I^*X(a),\psi_\I^*X(b))\ar[r] & {\rm Hom}_K(\psi_\I^*X(a),\psi_\I^*Y(b))
}
\end{equation}
commute pour tous objets $a$ et $b$ de $\A$. Le fait que cette transformation naturelle appartient à l'image du morphisme $\F(\A/\I\times\A;K)(X,Y)\to\F(\A;K)(\psi_\I^*X,\psi_\I^*Y)$ équivaut à la commutativité du diagramme évident
\begin{equation}\label{dctf2}
\xymatrix{(\A/\I)(a,b)\times\A(a,b)\ar[r]\ar[d] & {\rm Hom}_K(\psi_\I^*X(a),\psi_\I^*X(b))\ar[d] \\
{\rm Hom}_K(\psi_\I^*X(a),\psi_\I^*X(b))\ar[r] & {\rm Hom}_K(\psi_\I^*X(a),\psi_\I^*Y(b))
},
\end{equation}
dans lequel les deux composées $(\A/\I)(a,b)\times\A(a,b)\to {\rm Hom}_K(\psi_\I^*X(a),\psi_\I^*Y(b))$ sont polynomiales par rapport à la deuxième variable. Comme le diagramme commutatif (\ref{dctf1}) s'obtient en composant (\ref{dctf2}) avec $\A(a,b)\xrightarrow{j}(\A/\I)(a,b)\times\A(a,b)$, le lemme~\ref{lm-polbi} implique la commutativité de (\ref{dctf2}), ce qui achève la démonstration.
\end{proof}

Les lemmes~\ref{lm-polbi} et~\ref{lm-vraimev} impliquent également le résultat suivant, dont nous nous servirons à la fin de la démonstration du théorème~\ref{th-glob-tf}.

\begin{lm}\label{lm-factoris} Sous les hypothèses de la proposition~\ref{pr-simglobi}, supposons que $F$ est un foncteur de $\F(\A;K)$ tel que, pour tous objets $x$ et $y$ de $\A$, il existe une factorisation :
$$\xymatrix{\A(x,y)\ar[rr]^-{(\psi_\I)_{x,y}}\ar[rrd]_-{F_{x,y}} & & (\A/\I)(x,y)\times\A(x,y)\ar@{-->}[d]^-{g_{x,y}}\\
& & \mathrm{Hom}_K(F(x),F(y))
}$$
où $g_{x,y}$ est une fonction polynomiale par rapport à la deuxième variable. Alors $F$ est isomorphe à $\psi_\I^*G$ pour un foncteur $G$ de $\F(\A/\I\times\A;K)$ \ph\ par rapport à la deuxième variable.
\end{lm}

\begin{proof} Le lemme~\ref{lm-polbi} implique que l'application $g_{x,y}$ réalisant la factorisation est unique. Le même lemme appliqué avec $V=\A(x,y)\times\A(y,z)$ et $U=\I(x,y)\times\I(y,z)$ montre que les diagrammes suivants commutent
$$\xymatrix{\big((\A/\I)(x,y)\times\A(x,y)\big)\times\big((\A/\I)(y,z)\times\A(y,z)\big)\ar[r]\ar[d]_-{g_{x,y}\times g_{y,z}} & (\A/\I)(x,z)\times\A(x,z)\ar[d]^-{g_{x,z}} \\
\mathrm{Hom}_K(F(x),F(y))\times\mathrm{Hom}_K(F(y),F(z))\ar[r] & \mathrm{Hom}_K(F(x),F(z))
}$$
où les flèches horizontales sont données par la composition des morphismes. Autrement dit, les applications $g_{x,y}$ définissent une factorisation de $F$ à travers le foncteur essentiellent surjectif $\A\to\B$ induit par $\psi_\I$, où $\B$ est la catégorie du lemme~\ref{lm-vraimev} ; le foncteur de $\F(\B;K)$ ainsi obtenu est par construction \ph\ par rapport à la deuxième variable. La conclusion suit du lemme~\ref{lm-vraimev}.
\end{proof}

\subsection{Démonstration du théorème~\ref{th-glob-tf}}\label{dm-thglob}

La démonstration du théorème~\ref{th-glob-tf} repose sur la décomposition de Jordan multiplicative des automorphismes linéaires. Plus précisément, pour tout morphisme $f\in\A(x,y)$, notons $u[f]$ l'automorphisme de $x\oplus y$ donné matriciellement par
$$u[f]=\left(\begin{array}{cc} 1 & 0\\
                                                f & 1
                                               \end{array}\right)\;.$$
Nous établirons comme corollaire de notre démonstration une caractérisation des foncteurs \phs\ et antipolynomiaux en termes de leur effet sur les automorphismes $u[f]$ : un foncteur $F$ de type fini de $\F^\df(\A;K)$ est \ph\ si et seulement s'il transforme les automorphismes $u[f]$ en automorphismes unipotents, et il est antipolynomial si et seulement s'il transforme les morphismes $u[f]$ en automorphismes absolument semi-simples. (Cette caractérisation est énoncée à la proposition~\ref{pr-unidiag} en fin de section).

La stratégie de la démonstration se comprend bien à l'aune de cette caractérisation. Pour tout $f$, on dispose d'une décomposition de Jordan multiplicative $F(u[f])=D(f)U(f)$ où $D(f)$ est absolument semi-simple et $U(f)$ est unipotent. Le bifoncteur $B\in\F(\A\times\A;K)$ correspondant à $F$ dans le théorème~\ref{th-glob-tf} sera alors obtenu comme l'unique  bifoncteur (à isomorphisme près) tel que $B(u[f],1)=D(f)$ et $B(1,u[f])=U(f)$. Si l'on définit $B$ de cette façon, sa partie antipolynomiale envoie sur l'identité tous les morphismes $u[f]$ dont l'image par $F$ est unipotente. Ainsi, la partie antipolynomiale de $B$ se factorise à travers un quotient $\A\to\A/\I$ où $\I$ est un certain idéal $K$-cotrivial de $\A$ dont tous les éléments $f$ sont tels que $F(u[f])$ est unipotent.

Toute la démonstration du théorème~\ref{th-glob-tf} reposant sur la manipulation des automorphismes $u[f]$, nous commençons par rassembler, dans le lemme immédiat suivant, les propriétés élémentaires des morphismes $u[f]$ dont nous aurons besoin.

\begin{lm}\label{propunit} Soient $x$, $y$, $z$ des objets de $\A$, $f, f'\in\A(x,y)$ et $g\in\A(y,z)$.
\begin{enumerate}
\item\label{lpp} On a $u[f]u[f']=u[f']u[f]=u[f+f']$.
\item\label{lci1}\label{lci2} On a des diagrammes commutatifs dans $\A$:
$$\xymatrix{x\oplus y\ar[d]_-{x\oplus g}\ar[rr]^-{u[f]} & & x\oplus y\ar[d]^-{x\oplus g}\\
x\oplus z\ar[rr]_-{u[gf]} & & x\oplus z
}\;,\qquad\xymatrix{x\oplus z\ar[rr]^-{u[gf]}\ar[d]_-{f\oplus z} & & x\oplus z\ar[d]^-{f\oplus z} \\
y\oplus z\ar[rr]_-{u[g]} & & y\oplus z
}.$$
\item\label{lpi1} Notons $p(f,g) : x\oplus y\oplus z\to x\oplus z$ le morphisme dont  les composantes $x\to x$ et $z\to z$ sont les identités, la composante $y\to z$ est $g$ et les autres composantes sont nulles. Alors $p(f,g)$ est un épimorphisme scindé, et le diagramme suivant commute:
$$\xymatrix{x\oplus y\oplus z\ar[rr]^-{p(f,g)}\ar[d]_{u[f]\oplus z} & & x\oplus z\ar[d]^{u[gf]}\\
x\oplus y\oplus z\ar[rr]_-{p(f,g)} & & x\oplus z
}.$$
\item\label{lpi2} Notons $i(f,g) : x\oplus y\oplus z\to x\oplus z$ le morphisme dont  les composantes $x\to x$ et $z\to z$ sont les identités, la composante $x\to y$ est $f$ et les autres composantes sont nulles. Alors $i(f,g)$ est un monomorphisme scindé, et le diagramme suivant commute:
\[\xymatrix{x\oplus z\ar[rr]^-{i(f,g)}\ar[d]_{u[gf]} & & x\oplus y\oplus z\ar[d]^{x\oplus u[g]}\\
x\oplus z\ar[rr]_-{i(f,g)} & & x\oplus y\oplus z
}.\]
\end{enumerate}
\end{lm}

Dans la suite de cette section, on fixe un foncteur $F$ de $\F^\df(\A;K)$, de type fini (le cas de type co-fini est analogue, ou s'en déduit par dualité). Étant donné des objets $x$ et $y$ de $\A$, on note $\I(x,y)$ l'ensemble des morphismes $f$ de $\A(x,y)$ tels que l'automorphisme $F(u[f]\oplus t)$ de $F(x\oplus y\oplus t$) soit unipotent pour tout objet $t$ de $\A$. Nous ne nous servirons en fait que de l'unipotence de $F(u[f])$ pour montrer la factorisation du théorème~\ref{th-glob-tf}, mais celle de $F(u[f]\oplus t)$ pour tout $t$ est nécessaire pour le lemme suivant, qui constitue une étape cruciale de la démonstration du théorème.

\begin{lm}\label{lm-Icotriv} $\I$ est un idéal $K$-cotrivial de $\A$.

De plus, pour tout morphisme $f$ de $\A$, les valeurs propres de $F(u[f])$ dans une clôture algébrique $\bar{K}$ de $K$ sont des racines de l'unité.
\end{lm}

\begin{proof} Comme le produit de deux automorphismes unipotents permutables est unipotent, le lemme~\ref{propunit}.\ref{lpp} montre que $\I(x,y)$ est un sous-groupe de $\A(x,y)$.

Soient $f\in\I(x,y)$ et $g\in\A(y,z)$, montrons que $gf\in\I(x,z)$. Soit $t$ un objet de $\A$. En utilisant le lemme~\ref{propunit}.\ref{lpi1}, on obtient un diagramme commutatif
\begin{equation}\label{eqej}
\xymatrix{F(x\oplus y\oplus z\oplus t)\ar[rr]^-{F(p(f,g)\oplus t)}\ar[d]_{F(u[f]\oplus z\oplus t)} & & F(x\oplus z\oplus t)\ar[d]^{F(u[gf]\oplus t)}\\
F(x\oplus y\oplus z\oplus t)\ar[rr]_-{F(p(f,g)\oplus t)} & & F(x\oplus z\oplus t)
}
\end{equation}
dans lequel $F(p(f,g)\oplus t)$ est un épimorphisme. Comme $F(u[f]\oplus z\oplus t)$ est unipotent puisque $f$ appartient à $\I(x,y)$, on en déduit que $F(u[gf]\oplus t)$ est également unipotent. On a donc bien $gf\in\I(x,z)$.

On établit de façon analogue, à partir du lemme~\ref{propunit}.\ref{lpi2}, que, pour $f\in\A(x,y)$ et $g\in\I(y,z)$, on a $gf\in\I(x,z)$.

Par conséquent, $\I$ est un idéal de $\A$.

Comme $F$ est de type fini et $\A$ additive, il existe un objet $s$ de $\A$ qui constitue un support de $F$. Montrons qu'un morphisme $f\in\A(x,y)$ appartient à $\I(x,y)$ dès que $F(u[f]\oplus s)$ est unipotent. En effet, pour tout objet $t$ de $\A$, le morphisme $\xi : F(x\oplus y\oplus s)^{\oplus\A(s,t)}\to F(x\oplus y\oplus t)$ dont la composante indexée par un morphisme $\alpha\in\A(s,t)$ est $F(x\oplus y\oplus\alpha) : F(x\oplus y\oplus s)\to F(x\oplus y\oplus t)$ est un épimorphisme (cela se vérifie par inspection lorsque $F=P^s_\A$, le cas général s'en déduit par naturalité de $\xi$ en $F$ en utilisant la définition du support). Le diagramme commutatif
$$\xymatrix{F(x\oplus y\oplus s)^{\oplus\A(s,t)}\ar[r]^-\xi\ar[d]_{F(u[f]\oplus s)^{\oplus\A(s,t)}} & F(x\oplus y\oplus t)\ar[d]^{F(u[f]\oplus t)}\\
F(x\oplus y\oplus s)^{\oplus\A(s,t)}\ar[r]_-\xi & F(x\oplus y\oplus t)
}$$
montre donc que $F(u[f]\oplus t)$ est unipotent si c'est le cas de $F(u[f]\oplus s)$, de sorte que $f\in\I(x,y)$ si $F(u[f]\oplus s)$ est unipotent.

Pour tous objets $x$ et $y$ de $\A$, notons $d(x,y):=\dim_K F(x\oplus y\oplus s)$. On dispose d'un morphisme
de groupes $\A(x,y)\to (\bar{K}^\times)^{d(x,y)}$
obtenu en trigonalisant simultanément les automorphismes permutables $\bar{K}\otimes_K F(u[f]\oplus s)$ (pour $f\in\A(x,y)$) de $\bar{K}\otimes_K F(x\oplus y\oplus s)$, les composantes $\A(x,y)\to\bar{K}^\times$ du morphisme étant les différents coefficients diagonaux des matrices ainsi obtenues.

Le noyau de ce morphisme est exactement $\I(x,y)$ d'après ce qu'on vient de montrer, de sorte qu'il induit un monomorphisme de groupes
\[\phi_{x,y}=(\phi_{x,y}^1,\dots,\phi_{x,y}^{d(x,y)}) : (\A/\I)(x,y)\to (\bar{K}^\times)^{d(x,y)}.\]

Soient $f\in\A(x,y)$ et $g\in\A(y,z)$ des morphismes de $\A$. Le lemme~\ref{propunit}.\ref{lci1} procure un diagramme commutatif
$$\xymatrix{F(x\oplus y\oplus s)\ar[d]_-{F(x\oplus g\oplus s)}\ar[rr]^-{F(u[f]\oplus s)} & & F(x\oplus y\oplus s)\ar[d]^-{F(x\oplus g\oplus s)}\\
F(x\oplus z\oplus s)\ar[rr]_-{F(u[gf]\oplus s)} & & F(x\oplus z\oplus s)
\;.}$$
Si $g$ est un épimorphisme scindé, alors $F(x\oplus g\oplus s)$ est un épimorphisme, ce qui implique l'existence d'une fonction injective $\sigma : \{1,\dots,d(x,z)\}\to\{1,\dots,d(x,y)\}$ telle que $\phi^i_{x,z}(gf)=\phi^{\sigma(i)}_{x,y}(f)$ pour tout $i\in\{1,\dots,d(x,z)\}$. De même, en utilisant l'autre diagramme commutatif du lemme~\ref{propunit}.\ref{lci1}, on voit que, si $f$ est un monomorphisme scindé, il existe une fonction injective $\tau : \{1,\dots,d(x,z)\}\to\{1,\dots,d(y,z)\}$ telle que $\phi^i_{x,z}(gf)=\phi^{\tau(i)}_{y,z}(g)$ pour tout $i\in\{1,\dots,d(x,z)\}$.

Tout morphisme $h\in\A(x,y)$ se factorise en $x\xrightarrow{\alpha_h}x\oplus y\xrightarrow{\beta} y$,
où $\alpha_h$ a pour composantes $\mathrm{Id}_x$ et $h : x\to y$ et $\beta$ est la projection. Comme $\beta$ est un épimorphisme scindé, en appliquant le fait ci-avant à la composée $x\oplus y\xrightarrow{\mathrm{Id}_{x\oplus y}}x\oplus y\xrightarrow{\beta}y$, on obtient une fonction injective $\sigma : \{1,\dots,d(x\oplus y,y)\}\to\{1,\dots,d(x\oplus y,x\oplus y)\}$ telle que $\phi^i_{x\oplus y,y}(\beta)=\phi^{\sigma(i)}_{x\oplus y,x\oplus y}(\mathrm{Id}_{x\oplus y})$ pour tout $i$. De même, comme $\alpha_h$ est un monomorphisme scindé et que $h=\beta\alpha_h$, il existe une fonction injective $\tau : \{1,\dots,d(x,y)\}\to\{1,\dots,d(x\oplus y,y)\}$ telle que $\phi^j_{x,y}(h)=\phi^{\tau(j)}_{x\oplus y,y}(\beta)$ pour tout $j$.

Il s'ensuit que l'image de $\phi_{x,y}$ est {\em finie}, de cardinal majoré par le nombre de fonctions injectives de $\{1,\dots,d(x,y)\}$ vers $\{1,\dots,d(x\oplus y,x\oplus y)\}$. Le groupe abélien $(\A/\I)(x,y)$ est donc isomorphe à un sous-groupe fini d'un produit de copies de $\bar{K}^\times$. La nullité de $K\otimes_\mathbb{Z}(\A/\I)(x,y)$ découle de la finitude de $(\A/\I)(x,y)$ si $\mathrm{car}(K)=0$. Si $\mathrm{car}(K)=p>0$, elle vient de ce que $p$ est inversible dans $\bar{K}^\times$. Ainsi, l'idéal $\I$ de $\A$ est $K$-cotrivial. De plus, les valeurs propres dans $\bar{K}$ des $F(u[f])$ appartiennent à des sous-groupes finis de $\bar{K}^\times$, ce sont donc des racines de l'unité.
\end{proof}

Le lemme suivant montre comment déduire la polynomialité d'applications appropriées à partir de propriétés d'unipotence.

\begin{lm}\label{lm-polunipo} Soient $M$ un groupe abélien, $V$ un $K$-espace vectoriel de dimension finie $d$ et $\rho : M\to E:=\mathrm{End}_K(V)$ une fonction telle que $\rho(u+v)=\rho(u).\rho(v)$ pour tout $(u,v)\in M^2$. On suppose que $\rho$ prend ses valeurs dans les automorphismes unipotents de $V$. Alors $\rho$ définit une fonction polynomiale de degré au plus $d-1$ de $M$ vers le groupe additif sous-jacent à $E$.
\end{lm}

\begin{proof} Les automorphismes unipotents $\rho(u)$ de $V$ commutent deux à deux, donc il existe une base de $V$ dans laquelle leurs matrices $(\rho_{i,j}(u))_{1\le i,j\le d}$ sont triangulaires supérieures strictes : $\rho_{i,j}(u)$ est nul pour $i>j$ et égal à $1$ pour $i=j$. La relation $\rho(u+v)=\rho(u).\rho(v)$ s'écrit donc, pour $0<l\le d-i$,  $$\rho_{i,i+l}(u+v)=\rho_{i,i+l}(u)+\rho_{i,i+l}(v)+\sum_{0<n<l}\rho_{i,i+n}(u).\rho_{i+n,i+l}(v)\,,$$
ce dont on déduit par récurrence sur $l$ que les fonctions $\rho_{i,i+l} : M\to K$ sont polynomiales de degré au plus $l$, d'où le lemme.
\end{proof}

\begin{proof}[Démonstration du théorème~\ref{th-glob-tf}] 
Soit $f\in\A(x,y)$ un morphisme de $\A$. Par le lemme~\ref{lm-Icotriv}, les valeurs propres de $F(u[f])$ dans $\bar{K}$ sont des racines de l'unité, elles sont donc {\em séparables} sur $K$. Ainsi, $F(u[f])$ possède une décomposition de Jordan multiplicative en produit permutable d'un automorphisme unipotent $U(f)$ et d'un automorphisme absolument semi-simple $D(f)$ \cite[chap.~7, §\,5.9, th.~1 et prop.~17]{Bki2}.

 Comme les $u[f]$ sont deux à deux permutables, $U$ et $D$ définissent des morphismes de groupes $\A(x,y)\to {\rm Aut}_K(F(x\oplus y))$ dont les images commutent. De plus, par définition de $\I$, le noyau de $D$ contient $\I(x,y)$, de sorte que ces morphismes induisent un morphisme de groupes $(\A/\I)(x,y)\times\A(x,y)\to {\rm Aut}_K(F(x\oplus y))$. La composée de $U$ avec l'inclusion de ${\rm Aut}_K(F(x\oplus y))$ dans le groupe additif ${\rm End}_K(F(x\oplus y))$ est polynomiale (de degré au plus $\dim_K F(x\oplus y)-1$) grâce au lemme~\ref{lm-polunipo}.
 
 Le diagramme commutatif
 $$\xymatrix{\A(x,y)\ar[rr]^-{f\mapsto F(u[f])}\ar[rrd]_-{F_{x,y}} & & \mathrm{Hom}_K(F(x\oplus y),F(x\oplus y))\ar[d]^{{\rm Hom}_K(F(x\hookrightarrow x\oplus y),F(x\oplus y\twoheadrightarrow y))} \\
& & {\rm Hom}_K(F(x),F(y))
}$$
permet d'en déduire que $F_{x,y}$ se factorise comme la composée de l'application linéaire canonique
 $\A(x,y)\to (\A/\I)(x,y)\times\A(x,y)$
 et d'une fonction
 $(\A/\I)(x,y)\times\A(x,y)\to {\rm Hom}_K(F(x),F(y))$ polynomiale par rapport à la deuxième variable.
 Le lemme~\ref{lm-factoris} montre donc que $F$ se factorise comme composée de $\psi_\I$ et d'un foncteur de $\F(\A/\I\times\A;K)$ \ph\ par rapport à la deuxième variable. Comme $\I$ est $K$-cotrivial (lemme~\ref{lm-Icotriv}), cela montre que $F$ est isomorphe à la composée de la diagonale $\A\to\A\times\A$ et d'un bifoncteur $B$ de $\F(\A\times\A;K)$ antipolynomial par rapport à la première variable et \ph\ par rapport à la deuxième variable.
 
 L'unicité à isomorphisme près de cette factorisation découle de la proposition~\ref{pr-simglobi} et du fait que l'intersection de deux idéaux $K$-cotriviaux de $\A$ est un idéal $K$-cotrivial (proposition~\ref{prop-ev-cotriv}). Cela achève la démonstration.
\end{proof}

La définition du bifoncteur $B$ dans la démonstration du théorème~\ref{th-glob-tf} montre que pour tout morphisme $f$ de $\A$, $B(u[f],1)$ est absolument semi-simple et $B(1,u[f])$ est unipotent. Comme le bifoncteur $B$ est unique, nous en déduisons la caractérisation suivante des foncteurs antipolynomiaux et \phs.

\begin{pr}\label{pr-unidiag} Soit $F$ un foncteur de type fini (ou de type cofini) de $\F^\df(\A;K)$. 
\begin{enumerate}
\item Le foncteur $F$ est antipolynomial si et seulement si l'automorphisme $F(u[f])$ est absolument semi-simple pour tout morphisme $f$ de $\A$.
\item  Le foncteur $F$ est \ph\ si et seulement si l'automorphisme $F(u[f])$ est unipotent pour tout morphisme $f$ de $\A$.
\end{enumerate}
\end{pr}

\begin{rem}\label{rem-Stunipo} Cette caractérisation des foncteurs \phs\ est à rapprocher du fait bien connu \cite[thm~2.4.8]{Spri} que les morphismes $GL_n(K)\to GL_m(K)$ (dont l'effet sur les éléments du groupe est donné par des polynômes) préservent la décomposition de Jordan, en particulier les éléments unipotents.
La caractérisation des foncteurs antipolynomiaux peut quant à elle être rapprochée du fait que les éléments unipotents d'un groupe fini $GL_n(k)$, où $k$ est un corps fini de caractéristique $p$ différente de celle de $K$, agissent par des automorphismes absolument semi-simples sur toute représentation $K$-linéaire, car leur ordre est une puissance de $p$.
\end{rem}

\section{Décomposition tensorielle des simples polynomiaux}\label{sect-stein-equi}

Le résultat principal que nous allons établir constitue une forme fonctorielle très générale du théorème du produit tensoriel de Steinberg \cite{St-TPT}. Nous en discuterons dans la section~\ref{s-tpts} les liens avec les représentations des monoïdes et groupes linéaires, mais les démonstrations employées dans le cadre fonctoriel reposent plutôt sur les liens avec les représentations des groupes symétriques, par l'intermédiaire des effets croisés.

\subsection{Définitions et énoncés} \label{sous-sec-dec-tenso}

Pour tout $K[\Si_d]$-module $M$ on définit $E_M:K\Md\to K\Md$ par:
$$E_M(V)\;:=\; \mathrm{Im}\left( \left(V^{\otimes d}\otimes M\right)_{\Si_d} \xrightarrow[]{N} \left(V^{\otimes d}\otimes M\right)^{\Si_d} \right)\;.$$
Le morphisme $N$ apparaissant dans cette formule est la norme, qui est l'application $M_G\to M^G$ définie pour tout groupe fini $G$ agissant $K$-linéairement sur $M$, et qui à un co\"invariant $m$ associe l'invariant $N(m)=\sum_{g\in G} gm$. L'action de $\Si_d$ sur le $K$-module $V^{\otimes d}\otimes M$ est donnée par
$$\sigma(v_1\otimes\dots\otimes v_d\otimes m)= v_{\sigma^{-1}(1)}\otimes\dots\otimes v_{\sigma^{-1}(d)}\otimes\sigma m\;.$$

\begin{defi}\label{df-felem}
Un \emph{foncteur élémentaire}  $E$ est un endofoncteur de $K\Md$ tel qu'il existe un $K[\Si_d]$-module simple $M$ et un isomorphisme $E\simeq E_M$.
\end{defi}

\begin{rem}
Les foncteurs élémentaires sont intimement reliés aux foncteurs de Schur classiques. Nous n'utiliserons dans cette section que la définition élémentaire donnée ci-dessus, et renvoyons le lecteur à l'appendice \ref{app-elt} pour des détails sur les relations avec les foncteurs de Schur et les représentations de $\GL_n(K)$.
\end{rem}

Le but de la section est d'établir le théorème suivant de décomposition tensorielle à la Steinberg pour les foncteurs absolument simples polynomiaux. Contrairement au théorème \ref{th-simglob}, il ne nécessite pas d'hypothèse de dimension finie sur les foncteurs.

\begin{thm}\label{thm-st}
Soit $S$ un foncteur polynomial de degré $d\ge 0$ de $\F(\A;K)$. Les conditions suivantes sont équivalentes.
\begin{enumerate}
\item[(1)] Le foncteur $S$ est simple, et il existe des foncteurs additifs absolument simples $\pi_i$ de $\mathbf{Add}(\A;K)$ tels que $S$ soit un quotient de $\pi_1\otimes\dots\otimes \pi_d$. 
\item[(2)]\label{it-st-2} Le foncteur $S$ est simple, et il existe des foncteurs additifs absolument simples $\pi_i$ de $\mathbf{Add}(\A;K)$ tels que $S$ soit un sous-foncteur de $\pi_1\otimes\dots\otimes \pi_d$.  
\item[(3)] Il existe une famille de foncteurs élémentaires $(\mathrm{E}_\pi)$, indexée par les classes d'isomorphisme de foncteurs absolument simples $\pi$ de $\mathbf{Add}(\A;K)$, tous constants sauf un nombre fini, telle que
\[S\simeq\bigotimes_\pi \pi^*\mathrm{E}_\pi\;.\]
\end{enumerate}
De plus, si $S$ vérifie les conditions équivalentes précédentes, les foncteurs $\mathrm{E}_\pi$ sont uniques à isomorphisme près, et $S$ est absolument simple.
\end{thm}

\begin{rem}\label{rq-dectensmax} Le deuxième auteur a montré \cite[Th.~B.12]{Touze} que les foncteurs élémentaires ne se décomposent pas en produit tensoriel non trivial d'autres foncteurs. Il s'ensuit que la décomposition tensorielle fournie par le théorème précédent est \emph{maximale}.
Par contraste, nous ignorons comment obtenir une décomposition tensorielle maximale pour les foncteurs simples antipolynomiaux.
\end{rem}

Si l'on suppose que le foncteur polynomial simple considéré est à valeurs de dimensions finies, et que le corps $K$ est assez gros (par exemple algébriquement clos, d'après la proposition \ref{crit-absimpl}), alors les conditions précédentes sont vérifiées et l'on obtient le résultat suivant.

\begin{thm}\label{steinberg-pol-gal} Supposons que $K$ est un corps de décomposition de $\A$.
Un foncteur polynomial $S$ de $\F^{\df}(\A;K)$ est simple si et seulement s'il est isomorphe à un produit tensoriel
\[S\simeq\bigotimes_\pi \pi^*\mathrm{E}_\pi\]
indexé par les classes d'isomorphisme de foncteurs additifs simples $\pi$ de $\F^{\df}(\A;K)$, où les $\mathrm{E}_\pi$ sont des foncteurs élémentaires, tous constants sauf un nombre fini. De plus, les foncteurs élémentaires $\mathrm{E}_\pi$ sont uniques à isomorphisme près.

Enfin les foncteurs polynomiaux simples de $\F^{\df}(\A;K)$ sont absolument simples.
\end{thm}

\begin{proof} 
Soit $d$ le degré de $S$. L'adjonction diagonale-effets croisés de la proposition \ref{adj-diagcr} donne un morphisme non nul $\Delta^*cr_d(S)\to S$ qui est un épimorphisme par simplicité de $S$. D'après la proposition \ref{pr-fin-cr}, $cr_d(S)$ est semi-simple. Comme $K$ est un corps de décomposition de $\A$, il découle des propositions \ref{abs-semi-simple} et \ref{simples-cat-prod} que les facteurs directs de $cr_d(S)$ sont de la forme $\pi_1\boxtimes\dots\boxtimes\pi_d$ pour des foncteurs $\pi_i$ absolument simples. Ainsi $S$ est un quotient d'un produit tensoriel $\pi_1\otimes\dots\otimes\pi_d$ et le théorème \ref{steinberg-pol-gal} découle du théorème~\ref{thm-st}.
\end{proof}

\begin{rem}\label{rq-abss-st} Dans l'énoncé précédent, on ne peut généralement pas s'affranchir de l'hypothèse que $K$ est un corps de décomposition de $\A$, même si $S$ est \emph{absolument} simple.
  Considérons par exemple le foncteur quadratique $S$ de $\F(\mathbb{C},\mathbb{R})$ associant à un $\mathbb{C}$-espace vectoriel $V$ le $\mathbb{R}$-espace vectoriel des  formes hermitiennes sur $V^*$. On vérifie facilement que $S$ est absolument simple, mais il n'admet pas de décomposition à la Steinberg, car la catégorie $\mathbf{Add}(\mathbf{P}(\mathbb{C});\mathbb{R})$ ne contient \emph{aucun} foncteur absolument simple. En effet, un tel foncteur additif absolument simple $F$ fournirait un morphisme de corps $\mathbb{C}\to\mathrm{End}(F)\simeq\mathbb{R}\quad\lambda\mapsto F(.\lambda)$.
\end{rem}

Le théorème~\ref{steinberg-pol-gal}, le corollaire~\ref{cor-tens-St1} et la proposition~\ref{crit-absimpl} entraînent aussitôt le résultat suivant, dont nous ne connaissons pas de démonstration directe.

\begin{cor}\label{cor-corps-dec} Tout corps de décomposition de $\A$ qui contient toutes les racines de l'unité est un corps de décomposition non additif de $\A$.
\end{cor}

Si le corps $K$ n'est pas assez gros, on a le résultat plus faible suivant.

\begin{thm}\label{st-pol-petit-corps} Soit $F$ un foncteur polynomial simple de $\F^{\df}(\A;K)$. Il existe une extension finie de corps commutatifs $K\subset L$ telle que le foncteur $F\otimes_KL$ de $\F(\A;L)$ possède une filtration finie dont les sous-quotients sont du type
\[\pi_1^*E_1\otimes\dots\otimes \pi_d^*E_d\]
où les $\pi_i$ sont des foncteurs additifs absolument simples de $\F^{\df}(\A;L)$ et les $E_i$ des foncteurs élémentaires.
\end{thm}

\begin{proof} Soit $S$ un sous-foncteur simple de $cr_d(F)$ (cf. proposition~\ref{pr-fin-cr}), où $d$ est le degré de $F$. Par la proposition~\ref{pr-corpassegro}, il existe une extension finie $K\subset L$ telle que $S\otimes_KL$ possède une filtration finie dont les sous-quotients sont absolument simples (dans $\F(\A^d;L)$).

Comme $cr_d(F)$ est somme directe d'un nombre fini de foncteurs du type $\sigma.S$ pour $\sigma\in\Si_d$ (proposition~\ref{pr-fin-cr}), et que le foncteur $-\otimes_KL : \F(\A;K)\to\F(\A;L)$ commute aux effets croisés, on voit que tous les sous-quotients simples de $cr_d(F\otimes_KL)$ (dans $\mathbf{Add}_d(\A;L)$) sont absolument simples. Le foncteur $F\otimes_KL$ possède une filtration finie (de longueur bornée par le degré de $L$ sur $K$) dont les sous-quotients $T_i$ sont simples, polynomiaux, à valeurs de dimensions finies. Par construction, les sous-quotients simples de $cr_d(T_i)$ sont absolument simples dans $\mathbf{Add}_d(\A;L)$. On peut donc conclure que $T_i$ a la forme souhaitée en procédant comme dans la démonstration du théorème \ref{steinberg-pol-gal}. 
\end{proof}

\begin{rem}\label{rq-taille-corps} Si $F$ est \emph{absolument} simple, la proposition~\ref{pr-fin-cr} montre qu'on peut choisir  $L$ de degré au plus $d!$ sur $K$, où $d$ est le degré de $F$.
\end{rem}

\subsection{La catégorie auxiliaire $\Sigma_d(\A;K)$}\label{ssct-cataux} Dans cette section, on se donne un entier $d\ge 0$. Les notations $\Sigma\mathbf{Add}_d$, $\mathbf{Add}_d$, $\mathcal{O}$ et $\mathcal{L}$ sont celles du §\,\ref{par-msec}.

\begin{nota}\begin{enumerate}
\item On désigne par $\s$ un ensemble de représentants des classes d'isomorphisme de foncteurs absolument simples de $\mathbf{Add}(\A;K)$. Si $\underline{\pi}=(\pi_1,\dots,\pi_d)$ est un élément de $\s^d$, on note $t(\underline{\pi})$ l'objet $\pi_1\boxtimes\dots\boxtimes\pi_d$ de $\mathbf{Add}_d(\A,K\Md)$.
\item On note $\Sigma_d(\A;K)$ la sous-catégorie pleine de $\Sigma\mathbf{Add}_d(\A,K\Md)$ constituée des objets dont l'image par le foncteur d'oubli $\mathcal{O} : \Sigma\mathbf{Add}_d(\A,K\Md)\to\mathbf{Add}_d(\A,K\Md)$ est isomorphe à une somme directe de foncteurs du type $t(\underline{\pi})$ avec $\underline{\pi}\in\s^d$.
\end{enumerate}
\end{nota}

Comme un multifoncteur du type $t(\underline{\pi})$ est (absolument) simple (proposition~\ref{abs-semi-simple}), $\Sigma_d(\A;K)$ est une sous-catégorie de $\Sigma\mathbf{Add}_d(\A,K\Md)$ stable par sous-quotients et par colimites. C'est en particulier une catégorie de Grothendieck.

\begin{lm}\label{lm-mor-pte} Soient $\underline{\pi}$ et $\underline{\chi}$ des éléments de $\s^d$. Le $K$-espace vectoriel $\mathbf{Add}_d(\A;K)(t(\underline{\pi}),t(\underline{\chi}))$ est de dimension $1$ si $\underline{\pi}=\underline{\chi}$ et $0$ sinon.
\end{lm}

\begin{proof} Le cas $\underline{\pi}=\underline{\chi}$ provient de ce que les $\pi_i$ sont absolument simples (cf. proposition~\ref{abs-semi-simple}). Si $\underline{\pi}\ne\underline{\chi}$, soit $i\in\{1,\dots,d\}$ tel que $\pi_i$ ne soit pas isomorphe à $\chi_i$ : alors ${\rm Hom}(\pi_i,\chi_i)=0$, puisque $\pi_i$ et $\chi_i$ sont simples. La conclusion en résulte, en fixant toutes les variables sauf la $i$-ème.
\end{proof}

Dans ce qui suit, on fait usage de l'action à droite du groupe symétrique $\Si_d$ sur $\s^d$ par permutation des facteurs.

\begin{nota} Soit $\underline{\pi}$ un élément de $\s^d$.
\begin{enumerate}
\item On note $\Si_{\underline{\pi}}$ le sous-groupe de $\Si_d$ stabilisateur de $\underline{\pi}$.
\item On note $\mathrm{L}(\underline{\pi})$ l'objet $\mathcal{L}(t(\underline{\pi}))$ de $\Sigma\mathbf{Add}_d(\A,K\Md)$.
\end{enumerate}
\end{nota}

On constate que $\mathrm{L}(\underline{\pi})$ appartient à $\Sigma_d(\A;K)$.

\begin{lm}\label{lm-endsym} Soient $\underline{\pi}$ et $\underline{\chi}$ des éléments de $\s^d$.
\begin{enumerate}
\item Si $[\underline{\pi}]=[\underline{\chi}]$ dans $\s^d/\Si_d$, alors $\mathrm{L}(\underline{\chi})\simeq\mathrm{L}(\underline{\pi})$. Dans le cas contraire, $\Sigma\mathbf{Add}_d(\A,K\Md)(\mathrm{L}(\underline{\pi}),\mathrm{L}(\underline{\chi}))$ est nul.
\item La $K$-algèbre $\mathrm{End}(\mathrm{L}(\underline{\pi}))$ est isomorphe à $K[\Si_{\underline{\pi}}]$.
\end{enumerate}
\end{lm}

\begin{proof} Il est immédiat que $\mathrm{L}(\underline{\chi})\simeq\mathrm{L}(\underline{\pi})$ s'il existe $\sigma\in\Si_d$ tel que  $\chi_i=\pi_{\sigma(i)}$ pour tout $i$.

Par ailleurs, le lemme~\ref{lm-adj-cro} fournit des isomorphismes naturels
\[\Sigma\mathbf{Add}_d(\A,K\Md)(\mathrm{L}(\underline{\pi}),\mathrm{L}(\underline{\chi}))\simeq\mathbf{Add}_d(\A,K\Md)(t(\underline{\pi}),\mathcal{O}\mathrm{L}(\underline{\chi}))\]
\[\simeq\bigoplus_{\sigma\in\Si_d}\mathbf{Add}_d(\A,K\Md)\big(t(\underline{\pi}),t(\sigma.\underline{\chi})\big).\]
Il s'ensuit que $\Sigma\mathbf{Add}_d(\A,K\Md)(\mathrm{L}(\underline{\pi}),\mathrm{L}(\underline{\chi}))$ est nul si $\sigma.\underline{\chi}\ne\underline{\pi}$ pour toute permutation $\sigma$, grâce au lemme~\ref{lm-mor-pte}, et que $\mathrm{End}(\mathrm{L}(\underline{\pi}))$ est isomorphe comme $K$-espace vectoriel à $K[\Si_{\underline{\pi}}]$. Pour conclure, il suffit de remarquer que cet isomorphisme linéaire coïncide avec le morphisme de $K$-algèbres $K[\Si_{\underline{\pi}}]\to\mathrm{End}(\mathrm{L}(\underline{\pi}))$ fourni par l'action canonique du groupe $\Si_{\underline{\pi}}$ sur $t(\underline{\pi})$, donc sur $\mathrm{L}(\underline{\pi})$.
\end{proof}

\begin{lm}\label{lm-projs} Soit $\underline{\pi}\in\s^d$. L'objet $\mathrm{L}(\underline{\pi})$ est projectif dans $\Sigma_d(\A;K)$.
\end{lm}

\begin{proof} Cette propriété résulte de l'adjonction entre les foncteurs exacts $\mathcal{O}$ et $\mathcal{L}$ (lemme~\ref{lm-adj-cro}) et de ce que la restriction de $\mathcal{O}$ à la sous-catégorie $\Sigma_d(\A;K)$ de $\Sigma\mathbf{Add}_d(\A,K\Md)$ prend par définition ses valeurs dans la sous-catégorie des objets semi-simples de $\mathbf{Add}_d(\A,K\Md)$, comme $t(\underline{\pi})$.
\end{proof}

\begin{lm}\label{lm-gens} Les objets $\mathrm{L}(\underline{\pi})$ engendrent la catégorie abélienne $\Sigma_d(\A;K)$ lorsque $\underline{\pi}$ parcourt $\s^d$.
\end{lm}

\begin{proof} C'est une conséquence directe de la définition de $\Sigma_d(\A;K)$ et de l'adjonction entre $\mathcal{L}$ et $\mathcal{O}$.
\end{proof}

\begin{pr}\label{identif-cataux} Les foncteurs
\[\Sigma_d(\A;K)\to\prod_{[\underline{\pi}]\in\s^d/\Si_d}K[\Si_{\underline{\pi}}]\Md\]
dont les composantes sont les $\mathrm{Hom}(\mathrm{L}(\underline{\pi}),-)$ et
\[\prod_{[\underline{\pi}]\in\s^d/\Si_d}K[\Si_{\underline{\pi}}]\Md\to\Sigma_d(\A;K)\qquad (M_{\underline{\pi}})_{[\underline{\pi}]\in\s^d/\Si_d}\mapsto\bigoplus_{[\underline{\pi}]\in\s^d/\Si_d}\mathrm{L}(\underline{\pi})\underset{\Si_{\underline{\pi}}}{\otimes}M_{\underline{\pi}}\]
sont des équivalences de catégories quasi-inverses l'une de l'autre.
\end{pr}

\begin{proof} Cela découle des lemmes~\ref{lm-projs}, \ref{lm-gens} et~\ref{lm-endsym} par équivalence de Morita à plusieurs objets (voir le théorème~3.1 de \cite{Mi72}, dû à Freyd).
\end{proof}

\subsection{Démonstration du théorème~\ref{thm-st}}

Dans ce qui suit, si $[\underline{\pi}]$ est un élément de $\s^d/\Si_d$, quitte à réordonner les composantes de $\underline{\pi}$, on peut supposer que le sous-groupe $\Si_{\underline{\pi}}$ de $\Si_d$ est du type $\Si_{i_1}\times\dots\times\Si_{i_r}$ avec $i_1+\dots+i_r=d$ et écrire $\underline{\pi}=(\pi'_1,\dots,\pi'_1,\pi'_2,\dots,\pi'_2,\dots,\pi'_r,\dots,\pi'_r)$, où $\pi'_j$ est répété $i_j$ fois et les $\pi'_j$ sont deux à deux distincts.

\begin{lm}\label{lm-ident-printer} L'application
$$K[\Si_{i_1}]\Md\times\dots\times K[\Si_{i_r}]\Md\to\pol_d(\A;K)$$
donnée par la composée 
\begin{itemize}
\item du produit tensoriel extérieur $K[\Si_{i_1}]\Md\times\dots\times K[\Si_{i_r}]\Md\to K[\Si_{\underline{\pi}}]\Md$,
\item du plongement $K[\Si_{\underline{\pi}}]\Md\to\Sigma_d(\A;K)$ de la proposition~\ref{identif-cataux},
\item de l'inclusion $\Sigma_d(\A;K)\hookrightarrow\Sigma\mathbf{Add}_d(\A;K\Md)$,
\item du prolongement intermédiaire $\Sigma\mathbf{Add}_d(\A;K\Md)\to\pol_d(\A;K)$ associé au foncteur réflexif $\mathrm{Cr}_d$ (proposition~\ref{recollement-cr}),
\end{itemize}
envoie $(M_1,\dots,M_r)$ sur $(\pi'_1)^*E_{M_1}\otimes\dots\otimes (\pi'_r)^*E_{M_r}$.
\end{lm}

\begin{proof} Le prolongement intermédiaire associe à un multifoncteur $d$-multiadditif symétrique $X$ l'image de la norme $(\Delta_d^*X)_{\Si_d}\to(\Delta_d^*X)^{\Si_d}$, où $\Delta_d : \A\to\A^d$ est la diagonale itérée. Si $X$ est du type $\mathrm{L}(\underline{\pi})\underset{\Si_{\underline{\pi}}}{\otimes}(M_{i_1}\otimes\dots\otimes M_{i_r})$, alors le $\Si_d$-module $\Delta_d^*X$ s'identifie à $({\pi'_1}^{\otimes i_1}\otimes M_{i_1})\otimes\dots\otimes({\pi'_r}^{\otimes i_r}\otimes M_{i_r})\uparrow^{\Si_d}_{\Si_{i_1}\times\dots\times\Si_{i_r}}$, de sorte que l'application précédente s'identifie au produit tensoriel pour $t=1,\dots,r$ des normes $({\pi'_t}^{\otimes i_t}\otimes M_{i_t})_{\Si_{i_t}}\to ({\pi'_t}^{\otimes i_t}\otimes M_{i_t})^{\Si_{i_t}}$, d'où le lemme.
\end{proof}

\begin{lm}\label{lm-eq-ptcr} Soient $S$ un foncteur simple de $\F(\A;K)$, polynomial de degré $d$, et $\pi_1,\dots,\pi_d$ des foncteurs absolument simples de $\mathbf{Add}(\A;K)$. Les assertions suivantes sont équivalentes.
\begin{enumerate}
\item\label{itSq1} Le foncteur $S$ est un quotient de $\pi_1\otimes\dots\otimes\pi_d$.
\item\label{itSq2}  Le foncteur $S$ est un sous-objet de $\pi_1\otimes\dots\otimes\pi_d$.
\item\label{itcrm1} Le multifoncteur $\pi_1\boxtimes\dots\boxtimes\pi_d$ est facteur direct de $cr_d(S)$.
\item\label{itcrm2} Le multifoncteur multisymétrique $\mathrm{Cr}_d(S)$ est facteur direct de $\mathcal{L}(\pi_1\boxtimes\dots\boxtimes\pi_d)$.
\end{enumerate}
\end{lm}

\begin{proof} Comme $M:=\pi_1\boxtimes\dots\boxtimes\pi_d$ est simple, la proposition~\ref{pr-fin-cr} montre que les conditions \ref{itcrm1} et \ref{itcrm2} sont équivalentes entre elles, ainsi qu'à chacune des suivantes :
\begin{enumerate}
\item[\it{5.}] $M$ est un sous-objet de $cr_d(S)$ ;
\item[\it{6.}] $M$ est un quotient de $cr_d(S)$.
\end{enumerate} 

Maintenant, l'équivalence entre \ref{itSq1} et 5 d'une part, et entre \ref{itSq2} et 6 d'autre part, résulte de la simplicité de $M$ et de $S$ et de l'adjonction entre diagonale et effets croisés (proposition~\ref{adj-diagcr}).
\end{proof}

\begin{proof}[Démonstration du théorème~\ref{thm-st}] Les objets simples sont préservés par le foncteur d'inclusion $\Sigma_d(\A;K)\hookrightarrow\Sigma\mathbf{Add}_d(\A,K\Md)$, et tout simple de $\Sigma\mathbf{Add}_d(\A,K\Md)$ vérifiant les conditions équivalentes du lemme~\ref{lm-eq-ptcr} appartient à $\Sigma_d(\A;K)$ grâce à la proposition~\ref{pr-fin-cr}. En conséquence, la conclusion résulte des propositions~\ref{identif-cataux}, \ref{pr-prolongement-interm}, des lemmes~\ref{lm-ident-printer} et~\ref{lm-eq-ptcr} et de l'absolue simplicité des représentations simples des groupes symétriques \cite[théorème~11.5]{Ja-sym}.
\end{proof}

\part{Foncteurs simples de source $\mathbf{P}(A)$ et applications}\label{part-PA}

Dans cette partie, nous nous concentrons sur les foncteurs de source la catégorie $\mathbf{P}(A)$ des modules à gauche projectifs de type fini sur un anneau $A$. Nous résumons les résultats des sections précédentes dans la section~\ref{sec-PA}, et nous y apportons quelques compléments spécifiques à la source $\mathbf{P}(A)$. Puis nous en déduisons des décompositions tensorielles de représentations de monoïdes multiplicatifs $\M_n(A)$ de matrices carrées sur un anneau $A$ et de leurs sous-monoïdes $\GL_n(A)$ et $\SL_n(A)$.

Plus précisément, le foncteur réflexif (définition~\ref{lm-2adj} et proposition~\ref{restr-fct-refl}) $\mathrm{ev}_n$ d'évaluation sur $A^n$ et le prolongement intermédiaire associé (définition~\ref{df-prolint}), qu'on notera simplement $T_n$ dans cette partie, sont deux foncteurs dont la composée $\mathrm{ev}_n\circ T_n$
\[K[\M_n(A)]\Md\xrightarrow[]{T_n} \F(A,K) \xrightarrow[]{\mathrm{ev_n}} K[\M_n(A)]\Md\]
est isomorphe à l'identité (proposition~\ref{pr-prolongement-interm}).

Le foncteur $T_n$ envoie les modules simples sur des foncteurs simples (proposition~\ref{pr-prolongement-interm}), et le foncteur $\mathrm{ev}_n$ préserve les produits tensoriels.
Pour obtenir une décomposition tensorielle d'un $K[\M_n(A)]$-module simple $M$, il suffit donc  d'évaluer sur $A^n$ une décomposition tensorielle du foncteur simple $T_n(M)$. 
Cette approche pour obtenir des décompositions tensorielles est développée dans la section \ref{sec-dec-tens-mon}, où nous introduisons également les notations utiles dans les sections suivantes.

La restriction principale à cette approche est que nos théorèmes de décomposition tensorielle s'appliquent seulement si le foncteur simple $T_n(M)$ est à valeurs de dimensions finies sur le corps $K$.

Si l'anneau $A$ est fini, comme les foncteurs simples sont toujours à valeurs de dimensions finies, cette approche permet de montrer que {\it tous} les $K[\M_n(A)]$-modules simples admettent une décomposition tensorielle. Nous en déduisons une démonstration du théorème classique du produit tensoriel de Steinberg pour le groupe $\GL_n(A)$ sur un $p$-anneau fini $A$ dans la section~\ref{s-tpts}.

En revanche, si $A$ est un anneau infini, la restriction sur la dimension des valeurs de $T_n(M)$ est loin d'être anodine : supposer que $M$ est de dimension finie ne garantit pas que $T_n(M)$ soit à valeurs de dimensions finies. Pour cette raison, nous introduisons la notion de représentation \emph{polynomiale à la Eilenberg-MacLane} dans la section~\ref{seml}. Cette notion est un analogue naturel des représentations polynomiales au sens classique \cite{Green} où les polynômes sont remplacés par des applications polynomiales au sens de la section \ref{subsec-fctpol}. On montre que si $A$ est commutatif et si $M$ est une représentation polynomiale à la Eilenberg-MacLane de dimension finie, alors $T_n(M)$ est à valeurs de dimensions finies. Ceci nous assure de l'existence d'une décomposition tensorielle pour les $K[\M_n(A)]$-modules simples polynomiaux à la Eilenberg-MacLane, et nous en déduisons des décompositions tensorielles pour les groupes spéciaux linéaires.

De façon surprenante, le cas $n=1$ de la décomposition tensorielle des $K[\M_n(A)]$-modules, explicité dans la section~\ref{sta}, est un énoncé non trivial, qui prend une coloration arithmétique puisqu'il traite de fonctions multiplicatives (d'un anneau commutatif vers un corps commutatif) vérifiant {\it par ailleurs} une condition purement additive (la polynomialité au sens d'Eilenberg-MacLane). Ce résultat semble élémentaire mais nous ne savons pas l'établir sans recours aux catégories de foncteurs.

\section{Structure des foncteurs simples de source $\mathbf{P}(A)$}\label{sec-PA}

Si $p$ est un nombre premier, on appelle \emph{$p$-anneau fini} un anneau fini de caractéristique une puissance de $p$. Dans l'énoncé suivant, un foncteur $S_p$ de $\F(A,K)$ est dit \emph{$p$-antipolynomial} s'il existe un morphisme d'anneaux surjectif $A\to B$, où $B$ est un $p$-anneau fini, avec $p$ inversible dans $K$, tel que $S_p$ soit isomorphe à une composée :
\[\mathbf{P}(A)\xrightarrow{B\otimes_A -}\mathbf{P}(B)\xrightarrow{\overline{S_p}}K\Md\;.\]
Le théorème suivant résume le corollaire~\ref{cor-tens-St1}, le théorème~\ref{steinberg-pol-gal} et la proposition \ref{dec-prim} lorsque la source des foncteurs est la catégorie $\mathbf{P}(A)$.

\begin{thm}\label{thf-PA} Soient $A$ un anneau, $K$ un corps de décomposition non additif de $\mathbf{P}(A)$, et $F$ un foncteur de $\F^{\df}(A,K)$. Notons $(\pi_i)_{i\in I}$ une famille de représentants des classes d'isomorphisme des foncteurs additifs simples de $\F^{\df}(A,K)$, et $\mathrm{P}\cap K^\times$ l'ensemble des nombres premiers inversibles dans $K$. Alors $F$ est simple si et seulement s'il est isomorphe à un produit tensoriel
\[\left(\bigotimes_{p\in\mathrm{P}\cap K^\times}S_p\right)\otimes\left(\bigotimes_{i\in I}\pi^*_i\mathrm{E}_i\right)\]
où les $S_p$ sont des foncteurs simples $p$-antipolynomiaux et les $\mathrm{E}_i$ sont des endofoncteurs élémentaires des $K$-espaces vectoriels, tous constants sauf un nombre fini.

De plus, si $F$ est simple, alors $F$ est absolument simple, et les $S_p$ et les $\mathrm{E}_i$ apparaissant dans la décomposition sont uniques à isomorphisme près. 
\end{thm}
\begin{rem}
Si $F$ est un foncteur polynomial, alors l'énoncé du théorème reste valable sous l'hypothèse plus faible que $K$ est un corps de décomposition de $\mathbf{P}(A)$, c'est-à-dire que $K$ est un corps de décomposition de la $K$-algèbre $A^{\op}\otimes_{\mathbb{Z}}K$ au sens classique \cite[\S 7.B]{CuR}. 
De plus, si $F$ est polynomial simple, alors tous les $S_p$ apparaissant dans sa décomposition tensorielle sont constants.
\end{rem}

Dans la suite de la section, nous commentons l'énoncé du théorème \ref{thf-PA} en apportant quelques compléments spécifiques au cas de la source $\mathbf{P}(A)$.

\subsection{Hypothèse des valeurs de dimensions finies}\label{phdfx}

Le théorème \ref{thf-PA} ne concerne que les foncteurs $F$ à valeurs de dimensions finies. Si $A$ est un anneau fini, ceci n'est pas une restriction car tous les foncteurs simples sont à valeurs de dimensions finies d'après la proposition~\ref{pr-dim-finie}.

En revanche, si $A$ est un anneau infini, il existe de nombreux foncteurs simples dont toutes les valeurs ne sont pas de dimension finie, voir l'appendice~\ref{apa-1}. Toutefois, lorsque $A$ est commutatif, on dispose du résultat suivant, qui permet d'étendre le champ d'application du 
théorème~\ref{thf-PA} à des foncteurs simples polynomiaux dont on sait seulement qu'ils prennent \emph{une} valeur de dimension finie non nulle.

\begin{pr}\label{ann-com-cr} Soient $A$ un anneau commutatif et $F$ un foncteur simple polynomial de degré $d$ de $\F(A,K)$. Les assertions suivantes sont équivalentes.
\begin{enumerate}
\item\label{it-com1} Le foncteur $F$ est à valeurs de dimensions finies.
\item\label{it-com2} Le foncteur $F$ prend une valeur de dimension finie non nulle.
\item\label{it-com3} Le corps ${\rm End}(F)$ est de dimension finie sur $K$.
\item\label{it-com4} Le foncteur $cr_d(F)$ est à valeurs de dimensions finies.
\end{enumerate}
Si on suppose de plus que $A$ est un anneau de type fini, alors $F$ vérifie automatiquement les assertions  précédentes.
\end{pr}

\begin{proof} La proposition \ref{adj-diagcr} donne un morphisme $\Delta^*cr_d(F)\to F$  non nul, qui est un épimorphisme par simplicité de $F$. Ceci prouve l'implication \ref{it-com4}$\Rightarrow$\ref{it-com1}. L'implication \ref{it-com1}$\Rightarrow$\ref{it-com2} est immédiate, et \ref{it-com2}$\Rightarrow$\ref{it-com3} est une conséquence du corollaire \ref{simples-eval}.  Montrons \ref{it-com3}$\Rightarrow$\ref{it-com4}. La proposition~\ref{pr-fin-cr} et l'hypothèse sur ${\rm End}(F)$ donnent des extensions d'anneaux (non nécessairement commutatifs) de degrés finis
\[K\subset {\rm End}(F)\subset {\rm End}(cr_d(F))\simeq\M_n({\rm End}(S))^m\]
pour des entiers $n, m>0$ convenables (où $S$ est un sous-foncteur simple de $cr_dF$). Ainsi, ${\rm End}(S)$ est de degré fini sur $K$. Or $S$ est un objet simple de $\mathbf{Add}_d(A,K)\simeq (A^{\underset{\mathbb{Z}}{\otimes} d}\otimes_\mathbb{Z} K)\Md$, par une application itérée de la proposition~\ref{additifs-bimodules}. Comme l'anneau $A^{\underset{\mathbb{Z}}{\otimes} d}\otimes_\mathbb{Z} K$ est {\em commutatif}, un module simple sur celui-ci est un quotient qui est un corps \emph{commutatif} $L$, et le corps des endomorphismes de ce module simple s'identifie à $L$. Par conséquent, le fait que ${\rm End}(S)$ soit de degré fini sur $K$ implique que $S$ prend des valeurs de dimensions finies. Comme $cr_d(F)$ est une somme directe \emph{finie} de foncteurs du type $\sigma.S$ pour $\sigma\in\Si_d$ (proposition~\ref{pr-fin-cr}), on en déduit que $cr_d(F)$ prend également des valeurs de dimensions finies, ce qui achève la démonstration de \ref{it-com3}$\Rightarrow$\ref{it-com4}. Finalement si $A$ est de type fini, alors $A^{\underset{\mathbb{Z}}{\otimes} d}\otimes_\mathbb{Z} K$ est une $K$-algèbre commutative de type fini, donc tout module simple sur cette algèbre est de dimension finie. Ainsi $cr_d(F)$ est à valeurs de dimensions finies, car c'est une somme directe finie de simples de $\mathbf{Add}_d(A,K)\simeq (A^{\underset{\mathbb{Z}}{\otimes} d}\otimes_\mathbb{Z} K)\Md$.
\end{proof}

\begin{rem} Les implications \ref{it-com4}$\Rightarrow$\ref{it-com1}$\Rightarrow$\ref{it-com2}$\Rightarrow$\ref{it-com3} ne nécessitent pas la commutativité de $A$. En revanche, \ref{it-com3}$\Rightarrow$\ref{it-com1} est généralement fausse si $A$ n'est pas commutatif, même si $d=1$, car un $A^{\op}\otimes_\mathbb{Z}K$-module absolument simple peut être de dimension infinie sur $K$. Par exemple, si $k$ est un corps de dimension infinie sur son centre $K$, le $(k,k)$-bimodule $k$ est absolument simple mais de dimension infinie sur $K$.
Nous soupçonnons que l'implication \ref{it-com2}$\Rightarrow$\ref{it-com1} tombe en défaut pour certains anneaux non commutatifs $A$ mais n'en connaissons pas de contre-exemple.
\end{rem}

\begin{rem} Dans la démonstration de la proposition \ref{ann-com-cr}, la commutativité du corps des endomorphismes d'un foncteur simple de $\mathbf{Add}_d(A,K)$ joue un rôle important. On notera toutefois que le corps des endomorphismes d'un foncteur simple polynomial de $\F^{\df}(A,K)$ (avec $A$ et $K$ commutatifs) n'est pas nécessairement commutatif. Donnons-en un exemple quadratique dans $\F^{\df}(\mathbb{C},\mathbb{R})$ : on a un morphisme surjectif $\mathbb{C}\otimes_{\mathbb{R}}\mathbb{C}\to\mathbb{C}$ (donné par l'identité et la conjugaison) de $\mathbb{R}[\Si_2]$-algèbres, où l'action de $\Si_2$ est donnée par l'échange des facteurs à la source et la conjugaison au but. On en déduit un morphisme surjectif de $\mathbb{R}$-algèbres de $\Si_2\int\mathbb{C}$ (où $\mathbb{C}$ est vu comme $\mathbb{R}$-algèbre) sur l'algèbre tordue (relativement à l'action de conjugaison) du groupe $\Si_2$ sur $\mathbb{C}$. Or cette algèbre tordue est isomorphe à la $\mathbb{R}$-algèbre $\mathbb{H}$ des quaternions. Cela montre qu'il existe un $\Si_2\int\mathbb{C}$-module simple dont le corps d'endomorphismes est $\mathbb{H}$, d'où par prolongement intermédiaire un foncteur quadratique simple de $\F^{\df}(\mathbb{C},\mathbb{R})$ ayant la même propriété.
\end{rem}

\subsection{Structure des simples polynomiaux}

Les facteurs simples polynomiaux  $\pi_i^*E_i$ apparaissant dans le théorème \ref{thf-PA} sont relativement bien compris. En effet, on dispose de beaucoup d'informations sur les foncteurs élémentaires $E_i$, voir l'appendice \ref{app-elt}. D'autre part, la proposition \ref{additifs-bimodules} montre que les foncteurs additifs simples $\pi_i$ sont de la forme $\pi_i(P)=B_i\otimes_A P$, pour des $(K,A)$-bimodules simples $B_i$. Ces bimodules simples sont souvent accessibles à une classification complète, comme l'illustre le fait suivant, où $\mathbf{Ann}$ désigne la catégorie des anneaux.

\begin{pr}\label{pr-bimod-com}
Soit $A$ un anneau commutatif et $K$ un corps de décomposition de $\mathbf{P}(A)$. Pour tout morphisme d'anneaux $\phi:A\to K$ on note $K_\phi$ le groupe abélien $K$ muni de la structure de $(K,A)$-bimodule donnée par $\lambda\cdot x\cdot a:= \lambda x \phi(a)$. Les bimodules $(K_{\phi})_{\phi\in \mathbf{Ann}(A,K)}$ forment une famille de représentants des classes d'isomorphisme de $(K,A)$-bimodules simples de dimension finie.
\end{pr}

\begin{proof}
Soient $B=A\otimes_\mathbb{Z}K$ et $S$ un $B$-module simple de dimension finie sur $K$. Alors $S\simeq B/I$ où $I$ est un idéal maximal de $B$, donc $S$ admet une structure de corps commutatif. L'hypothèse sur $K$ garantit alors que $S\simeq {\rm End}_B(S)$ est de dimension $1$ sur $K$. Fixons un isomorphisme $K$-linéaire $S\simeq K$. Alors la structure de $B$-module de $S$ est donnée par un morphisme d'anneaux $\psi:B\to K$, dont on note $\psi_A$ et $\psi_K$ les restrictions à $A$ et $K$. Donc $S$ est isomorphe à $K_\phi$ avec $\phi=\psi_K^{-1}\circ\psi_A$.
\end{proof}

\begin{ex}\label{ex-kuhn} Prenons $A=K=\FF_q$ avec $q=p^r$ ; $\FF_q$ est un corps de décomposition de $\mathbf{P}(\FF_q)$, et les endomorphismes de $\FF_q$ sont les $\mathrm{Fr}^i$ pour $0\le i\le r-1$, où $\mathrm{Fr}$ désigne le morphisme de Frobenius $x\mapsto x^p$. Si l'on note $F^{\{i\}}$ la précomposition d'un foncteur $F$ de $\F(\FF_q,\FF_q)$ par le foncteur additif associé au bimodule $(\FF_q)_{\mathrm{Fr}^i}$, le théorème dit donc que les simples de $\F(\FF_q,\FF_q)$ sont les produits  tensoriels $\bigotimes_{i=0}^{r-1}E_i^{\{i\}}$ où les $E_i$ sont des foncteurs élémentaires. Le théorème~\ref{thf-PA} fournit ainsi une troisième démonstration d'un théorème établi par Kuhn  par deux méthodes différentes (\cite[Thm~5.23]{Ku2} et \cite[Thm~7.11]{Ku-strat}).
\end{ex}

\subsection{Structure des simples antipolynomiaux}
La structure des foncteurs simples $p$-antipolynomiaux $S_p$ apparaissant dans le théorème \ref{thf-PA} est moins bien comprise que celle des foncteurs simples polynomiaux $\pi_i^*E_i$. 

Par définition, comprendre ces simples antipolynomiaux revient à comprendre les foncteurs simples de $\F(A,K)$ lorsque $A$ est un $p$-anneau fini avec $p$ inversible dans $K$. Cependant, même pour un cas aussi élémentaire que $A=\mathbb{Z}/p^2\mathbb{Z}$ et $K$ de caractéristique nulle, la classification de ces simples semble hors d'atteinte. En effet, par la proposition \ref{simples-sfin}, cette classification est équivalente à la classification des représentations $K$-linéaires simples des $\GL_n(\mathbb{Z}/p^2\mathbb{Z})$, qui est un problème sauvage d'après un théorème de Bondarenko-Nagornyi cité dans \cite[§\,4]{VS-survol}.

Pour contourner ce genre de difficulté, on peut s'intéresser plutôt à des propriétés qualitatives globales des foncteurs antipolynomiaux, par exemple à leurs fonctions de dimensions (définition~\ref{defi-fct-dim}). On sait que la fonction de dimensions d'un foncteur polynomial est polynomiale, et que celle d'un foncteur non polynomial croît plus vite que tout polynôme (proposition~\ref{pr-fct-dim-pol}) ; dans les cas usuels, on constate une croissance exponentielle. Nous conjecturons que cette dichotomie entre croissances polynomiale et exponentielle est générale :

\begin{conj}\label{conj-dim}
Soit $A$ un $p$-anneau fini. On suppose $\mathrm{car}(K)\neq p$. Alors pour tout foncteur simple $F$ de $\F(A,K)$, il existe une fonction polynomiale $f : \mathbb{Z}\to\mathbb{Z}$ telle que $\forall n\in\mathbb{N}\quad\mathrm{d}_F(n)=f(p^n)$.
\end{conj}

\begin{pr}\label{pr-kud} La conjecture~\ref{conj-dim} est vraie si l'anneau $A$ est semi-simple.
\end{pr}

\begin{proof} Si $A$ est un \emph{corps}, les résultats de Kuhn \cite{Ku-adv} montrent que la fonction de dimensions d'un foncteur de type fini de $\F(A,K)$ est une combinaison linéaire de fonctions associant à $n$ le cardinal de l'ensemble des sous-espaces de dimension $i$ d'un $A$-espace vectoriel de dimension $n$, $i$ étant un entier fixé. Ce cardinal est bien polynomial en $p^n$. Par équivalence de Morita, elle vaut aussi dans le cas où $A$ est un anneau de matrices sur un corps fini de caractéristique $p$.

Dans le cas général, $A$ est isomorphe à un produit fini de tels anneaux $R_1,\dots,R_n$. Quitte à remplacer $K$ par une extension finie appropriée de ce corps (cf. proposition~\ref{pr-corpassegro}), on peut se ramener au cas où $F$ est \emph{absolument} simple : $F$ est alors isomorphe (via l'équivalence de catégories $\mathbf{P}(A)\simeq\mathbf{P}(R_1)\times\dots\times\mathbf{P}(R_n)$) à un produit tensoriel extérieur de foncteurs simples des catégories $\F(R_i,K)$, grâce aux propositions~\ref{simples-cat-prod} et~\ref{abs-semi-simple}, ce qui achève la démonstration.
\end{proof}

\begin{rem} Nous montrerons ultérieurement à la proposition \ref{corsam}, que dans la situation de la conjecture \ref{conj-dim}, la catégorie $\F(A,K)$ est localement finie. Ainsi, la conjecture équivaut au fait que pour tout foncteur $F$ \emph{de type fini}, on peut trouver une fonction polynomiale $f : \mathbb{Z}\to\mathbb{Z}$ telle que $\forall n\in\mathbb{N}\quad\mathrm{d}_F(n)=f(p^n)$. Nagpal a démontré \cite[théorème~1.1]{Nag}, à l'aide de méthodes différentes, un résultat analogue à la proposition~\ref{pr-kud} pour les foncteurs de type fini des espaces vectoriels de dimensions finies sur un corps fini \emph{avec injections} vers les espaces vectoriels sur un corps d'une autre caractéristique.
La dichotomie entre croissance polynomiale et exponentielle pour des foncteurs vérifiant des propriétés de finitude appropriées pourrait ainsi constituer un phénomène encore plus général que celui pressenti pour une source additive raisonnable.
\end{rem}

\section{Décompositions à la Steinberg des représentations des monoïdes de matrices}\label{sec-dec-tens-mon}

Nous fixons un corps commutatif $K$ et un anneau $A$, et nous désignons par $G_n(A)$ l'un des monoïdes $\M_n(A)$, $\GL_n(A)$, ou encore $\SL_n(A)$ si $A$ est commutatif. On parlera souvent de représentation de $G_n(A)$ pour désigner un $K[G_n(A)]$-module.
Pour tout foncteur $F$ de $\F(A,K)$, le $K$-espace vectoriel  $F(A^n)$ est naturellement une représentation de $\M_n(A)={\rm End}_A(A^n)$, donc une représentation de $G_n(A)$ par restriction.

\begin{defi}
Une représentation non nulle $M$ de $G_n(K)$ est dite \emph{élémentaire} si elle est isomorphe à une représentation de la forme $E(K^{n})$, où $E$ est un foncteur élémentaire. De manière équivalente (voir le début de l'appendice~\ref{app-elt}), il existe un $\Si_d$-module simple $S$ tel que $M$ soit isomorphe à l'image du morphisme de norme :
$$(K^n)^{\otimes d}\otimes_{\Si_d} S\xrightarrow[]{N} ((K^n)^{\otimes d}\otimes S)^{\Si_d}\;.$$
\end{defi}

Se donner un $(K,A)$-bimodule $B$ de dimension $d$ sur $K$ équivaut à se donner un morphisme d'anneaux $\phi_B:A\to \M_d(K)^{\op}\simeq\M_d(K)$. On en déduit pour tout $n$ strictement positif un morphisme d'anneaux (qui est le morphisme d'anneaux définissant le $(K,\M_n(A))$-bimodule $B^{\oplus n}$)
$$\M_n(\phi_B):\mathcal{M}_n(A)\to \mathcal{M}_n(\mathcal{M}_{d}(K))\simeq \mathcal{M}_{n d}(K)\;.$$
\begin{lm}
Le morphisme $\M_n(\phi_B)$ se restreint en un morphisme de monoïdes $G_n(\phi_B):G_n(A)\to G_{nd}(K)$.
\end{lm}
\begin{proof}
Pour $G_n=\M_n$ ou $\GL_n$ c'est évident. Si $A$ est commutatif et  $G_n=\SL_n$, cela suit de \cite[Thm 1]{Silv-det_bloc}.
\end{proof}

\begin{nota}\label{nota-pbck}
Si $B$ est un $(K,A)$-bimodule, on note $M^{[B]}$ la représentation de  $G_n(A)$ obtenue par restriction d'une représentation $M$ de  $G_{nd}(K)$ le long du morphisme de monoïdes $G_n(\phi_B)$.
\end{nota}

\begin{nota}\label{nota-pbckanneau}
Si $\phi:A\to K$ est un morphisme d'anneaux, on note $K_\phi$ le groupe abélien $K$ muni de la structure de bimodule donnée par $\lambda\cdot x\cdot a:= \lambda x\phi(a)$, comme dans la proposition \ref{pr-bimod-com}. Dans ce cas, $M^{[K_\phi]}$ est la restriction de $M$ le long du morphisme de monoïdes $G_n(\phi)$ et sera plus simplement notée $M^{[\phi]}$.
\end{nota}

\begin{defi} Soit $(B_i)_{i\in I}$ une famille complète de représentants des classes d'isomorphisme de $(K,A)$-bimodules simples de dimension finie sur $K$, et soit $d_i$ la dimension de $B_i$. 
On dit qu'une représentation $M$ de $G_n(A)$ \emph{admet une décomposition à la Steinberg} s'il existe des représentations élémentaires $M_i$ de $G_{n d_i}(K)$, toutes égales à $K$ sauf un nombre fini, et un isomorphisme:
\[ M\simeq \bigotimes_{i\in I}M_i^{[B_i]}\;.\] 
La décomposition tensorielle à la Steinberg est dite \emph{unique} si les représentations élémentaires $M_i$ sont uniques à isomorphisme près. 
\end{defi}

\begin{rem}\label{dec-St-dim1} Toute représentation élémentaire de dimension $1$ de $G_n(K)$ factorise par le déterminant, et est en particulier triviale sur $\SL_n(K)$. Il s'ensuit que, si $A$ est commutatif, la seule représentation de dimension $1$ de $\SL_n(A)$ admettant une décomposition à la Steinberg est la représentation triviale.
\end{rem}

On rappelle que $T_n:K[\M_n(A)]\Md\to \F(A,K)$ désigne le prolongement intermédiaire associé au foncteur d'évaluation sur $A^n$.
Le théorème suivant établit l'existence et l'unicité des décompositions à la Steinberg d'une représentation $M$ de $\M_n(A)$, sous réserve que le foncteur $T_n(M)$ soit polynomial à valeurs de dimensions finies.  Cette hypothèse sur $T_n(M)$ sera traduite purement en termes de la représentation $M$ dans les sections suivantes si $A$ est fini ou si $A$ est commutatif.

\begin{thm}\label{Mn}
Soient $A$ un anneau et $K$ un corps de décomposition de $\mathbf{P}(A)$. Pour toute représentation $M$ de $\M_n(A)$, les assertions suivantes sont équivalentes.
\begin{enumerate}
\item[(1)] La représentation $M$ est simple et le foncteur $T_n(M)$ est polynomial, à valeurs de dimensions finies.
\item[(2)] La représentation $M$ admet une décomposition à la Steinberg.
\end{enumerate}
De plus, si $M$ satisfait ces assertions alors $M$ est absolument simple, et sa décomposition à la Steinberg est unique.
\end{thm}

\begin{proof}
Supposons (1) vérifié, alors $T_n(M)$ est simple d'après la proposition \ref{pr-prolongement-interm}, donc admet une décomposition tensorielle d'après le théorème \ref{thf-PA}. On en déduit par évaluation sur $A^n$ une décomposition tensorielle à la Steinberg de $T_n(M)(A^n)\simeq M$, ce qui montre (2). Réciproquement, supposons (2), notons $\pi_i:\mathbf{P}(A)\to K\Md$ le foncteur additif tel que $\pi_i(P)=B_i\otimes_A P$, et notons $E_i$ un foncteur élémentaire tel que $E_i(A^n)\simeq M_i$. Alors $M$ est isomorphe à l'évaluation sur $A^n$ de $\bigotimes_{i\in I}\pi_i^*E_i$. Ce dernier est un foncteur simple d'après le théorème \ref{thf-PA}, donc $M$ est simple d'après le corollaire \ref{simples-eval}. De plus $T_n(M)$ et $\bigotimes_{i\in I}\pi_i^*E_i$ sont deux foncteurs simples dont l'évaluation sur $A^n$ est isomorphe à $M$. D'après la proposition \ref{pr-prolongement-interm} ces deux foncteurs sont donc isomorphes, ce qui prouve que $T_n(M)$ est simple à valeurs de dimensions finies. On a donc prouvé (1). Enfin, si $M$ admet deux décompositions à la Steinberg données par des modules $M_i$ et $M'_i$, on note $E_i$ et $E_i'$ les foncteurs élémentaires correspondants. Par l'argument précédent $\bigotimes_{i\in I}\pi_i^*E_i$ et $\bigotimes_{i\in I}\pi_i^*E_i'$ sont tous deux isomorphes à  $T_n(M)$. L'unicité dans le théorème \ref{thf-PA} donne alors des isomorphismes $E_i\simeq E'_i$, donc des isomorphismes $M_i\simeq M'_i$.
\end{proof}

\section{Représentations des monoïdes et groupes de matrices finis}\label{s-tpts}

Dans cette section, nous supposons que $A$ est un $p$-anneau fini et fixons un corps commutatif $K$ de caractéristique $p$. On rappelle que $T_n:K[\M_n(A)]\to\F(A,K)$ désigne le prolongement intermédiaire associé à l'évaluation sur $A^n$.

Le lemme fondamental suivant nous permettra d'utiliser notre décomposition à la Steinberg polynomiale afin d'étudier les représentations des monoïdes $\M_n(A)$.

\begin{lm}
Si $K$ est de caractéristique $p$ et si $A$ est un $p$-anneau fini, alors pour toute représentation $M$ de $\M_n(A)$ de dimension finie, le foncteur $T_n(M)$ est polynomial, à valeurs de dimensions finies.
\end{lm}

\begin{proof}
Comme $M$ est de dimension finie, $T_n(M)$ est de type fini, donc à valeurs de dimensions finies d'après la proposition~\ref{pr-dim-finie}. L'hypothèse sur la caractéristique de $A$ assure ensuite que les foncteurs antipolynomiaux de $\F(A,K)$ sont constants, donc $T_n(M)$ est polynomial d'après le corollaire \ref{cor-tftcf}.
\end{proof}

Le lemme précédent permet d'appliquer le théorème~\ref{Mn} aux $p$-anneaux finis sous la forme suivante.
\begin{thm}\label{Mn-fini}
Soit $A$ un $p$-anneau fini, et $K$ un corps de décomposition de $\mathbf{P}(A)$, de caractéristique $p$. Une représentation de $\M_n(A)$ est simple si et seulement si elle admet une décomposition à la Steinberg.
Si c'est le cas, la représentation simple est absolument simple, et sa décomposition à la Steinberg est unique.
\end{thm}

Nous pouvons déduire du théorème \ref{Mn-fini} l'existence d'une décomposition à la Steinberg pour les représentations de $\GL_n(A)$. Plus précisément, comme $A$ est \emph{fini}, l'inclusion $\iota:K[\GL_n(A)]\to K[\M_n(A)]$ possède une rétraction $r$ qui est un morphisme d'anneaux, voir le lemme \ref{lm-mon-eltr} et l'exemple \ref{ex-morannf}. Le corollaire suivant s'obtient en restreignant à $\GL_n(A)$ la décomposition à la Steinberg de la représentation simple $r^*M$.

\begin{cor}\label{cor-GLfin}
Soit $K$ un corps de décomposition de $\mathbf{P}(A)$. Toute représentation simple de $\GL_n(A)$ est absolument simple et admet une décomposition à la Steinberg. 
\end{cor}

Dans le reste de la section, nous montrons réciproquement que les représentations de $\GL_n(A)$ admettant une décomposition à la Steinberg sont simples et étudions la question de l'unicité d'une telle décomposition.

Nous commençons par le cas où $A$ un corps fini de cardinal $q=p^r$.  
Rappelons qu'une partition $\lambda=(\lambda_1,\dots,\lambda_n)$ est \emph{$q$-restreinte} si on a $\lambda_n<q$ et $\lambda_i-\lambda_{i+1}<q$  pour tout $i<n$. Toute partition $q$-restreinte $\lambda$ se décompose de manière unique en une somme 
$\lambda=\lambda^0+p\lambda^1+\dots+p^{r-1}\lambda^{r-1}$
où les $\lambda^i$ sont des partitions $p$-restreintes. D'après la proposition~\ref{prop-L-p} de l'appendice~\ref{app-elt}, les foncteurs élémentaires sont isomorphes aux socles $\mathbb{L}_\lambda$ des foncteurs de Schur $\mathbb{S}_\lambda$ indexés par les partitions $p$-restreintes.

\begin{nota}\label{nota-qrest}
Soit $\FF_q$ un sous-corps de $K$ de cardinal $q=p^r$.
Soient $\lambda$ une partition $q$-restreinte et $\lambda^0+\dots+p^{r-1}\lambda^{r-1}$ sa décomposition à l'aide de partitions $p$-restreintes $\lambda^i$. On note:
\[\mathbb{L}_\lambda(K^n)=\mathbb{L}_{\lambda^0}(K^n)^{\{0\}}\otimes \dots \otimes \mathbb{L}_{\lambda^{r-1}}(K^n)^{\{r-1\}}\]
où l'exposant $^{\{k\}}$ indique la restriction d'une représentation de  module le long du morphisme $\FF_q\to K$, $x\mapsto x^{p^k}$. Ainsi $\mathbb{L}_\lambda(K^n)$ est une représentation de $\M_n(\FF_q)$. 
\end{nota}

\begin{lm}\label{lm-repMnFq}
Soit $\FF_q$ un sous-corps de $K$ de cardinal $q=p^r$. Les  restrictions à $\M_n(\FF_q)$ des représentations $\mathbb{L}_\lambda(K^n)$ indexées par les partitions $q$-restreintes $\lambda$ forment une famille complète de représentants des classes d'isomorphisme de représentations simples de $\M_n(\FF_q)$.

De plus, notons $\delta$ la représentation de dimension $1$ de $\M_n(\FF_q)$ sur laquelle les matrices inversibles agissent par l'identité et les matrices non inversibles par $0$ (cf. lemme~\ref{lm-mon-eltr}). Pour toute partition $q$-restreinte $\lambda=(\lambda_1,\dots,\lambda_n)$, posons $\mu=\lambda+(q-1,\dots,q-1)$ si $\lambda_n=0$ et $\mu=\lambda$ sinon. Alors on a un isomorphisme $L_{\lambda}(K^n)\otimes\delta \simeq L_\mu(K^n)$ en tant que représentations de $\M_n(\FF_q)$.
\end{lm}

\begin{proof}
La théorie de Galois des corps finis fournit un isomorphisme de $(K,\FF_q)$-bimodules 
$K\otimes_{\mathbb{Z}}{\FF_q}={}^{(0)}K\oplus \dots\oplus{}^{(r-1)}K$
où chaque ${}^{(i)}K$ désigne le $K$-espace vectoriel $K$, muni de l'action à droite de $\FF_q$ par $x\cdot a= xa^{p^i}$. En d'autres termes, les ${}^{(i)}K$ forment un système complet de représentants des $(K,\FF_q)$-modules simples. De plus, pour toute représentation $M$ de $\M_n(K)$, la représentation $M^{[{}^{(i)}K]}$ (notation \ref{nota-pbck}) est isomorphe à la restriction à $\M_n(\FF_q)$ de $M^{\{i\}}$ (notation \ref{nota-qrest}). La première partie du lemme découle donc directement du théorème \ref{Mn-fini}. La deuxième partie du lemme découle de la proposition \ref{pr-technique1} de l'appendice \ref{app-elt} et du fait que $\delta$ coïncide avec la restriction à $\M_n(\FF_q)$ de $\det^{q-1}$.
\end{proof}

\begin{thm}\label{GLn-fini-Fq}
Soit $\FF_q$ un sous-corps de $K$ de cardinal $q=p^r$. Alors une représentation de $\GL_n(\FF_q)$ est simple si et seulement si elle admet une décomposition à la Steinberg. 
De plus, les représentations simples sont absolument simples. 

Enfin, $\bigotimes_{0\le i< r}{M_i}^{\{i\}} $ et $\bigotimes_{0\le i< r}{M_i'}^{\{i\}}$  sont deux décompositions à la Steinberg de la même représentation simple si et seulement si ($M_i\otimes\det^{p-1}\simeq M_i'$ pour tout $i$), ou ($M_i\simeq M_i'\otimes\det^{p-1}$ pour tout $i$), ou encore ($M_i\simeq M_i'$ pour tout $i$).
\end{thm}

\begin{proof}
Le corollaire \ref{cor-GLfin} montre que toute représentation simple de $\GL_n(\FF_q)$ est absolument simple et admet une décomposition à la Steinberg. Réciproquement, soit $M$ une représentation de $\GL_n(\FF_q)$ admettant une décomposition à la Steinberg. Comme toute représentation élémentaire de $\GL_n(\FF_q)$ est la restriction d'une représentation élémentaire de $\M_n(\FF_q)$, on en déduit que $M$ est la restriction d'une représentation de $\M_n(\FF_q)$, qui est simple d'après le théorème \ref{Mn-fini}. Mais les lemmes \ref{lm-rep-mono-grp} et \ref{lm-repMnFq} montrent que la restriction à $\GL_n(\FF_q)$ préserve la simplicité, donc $M$ est simple. 

Il reste à montrer l'énoncé d'unicité des décompositions à la Steinberg. Deux décompositions à la Steinberg différentes d'une même représentation de $\GL_n(\FF_q)$ correspondent à deux représentations simples non isomorphes de $\M_n(\FF_q)$ qui ont même restriction à $\GL_n(\FF_q)$. 
D'après les lemmes \ref{lm-rep-mono-grp} et \ref{lm-repMnFq}, la seule possibilité est que l'une des représentations soit $\mathbb{L}_\lambda(K^n)$ pour une partition $q$-restreinte $\lambda$ telle que $\lambda_n=0$ et que l'autre soit $\mathbb{L}_\lambda(K^n)\otimes\delta=\mathbb{L}_\mu(K^n)$ avec $\mu=\lambda+(q-1,\dots,q-1)$. Ceci implique que $\mu^i=\lambda_i+(p-1,\dots,p-1)$ pour tout $i$, donc que $\mathbb{L}_{\mu^i}(K^n)=\mathbb{L}_{\lambda^i}\otimes\det^{p-1}$ pour tout $i$.
Réciproquement si pour tout $i$ on a $M_i\simeq M'_i\otimes\det^{p-1}$ alors  $M=\bigotimes_{0\le i< r}{M_i}^{\{i\}} $ est isomorphe au produit tensoriel de $M'=\bigotimes_{0\le i< r}{M_i'}^{\{i\}}$ par la représentation $\bigotimes_{0\le i< r}(\det^{p-1})^{\{i\}}=\det^{q-1}$ qui coïncide avec la représentation triviale de dimension $1$, donc $M\simeq M'$.
\end{proof}

Nous revenons maintenant au cas d'un $p$-anneau fini quelconque $A$. 
Comme $A$ est un $p$-anneau fini, le théorème d'Artin-Wedderburn (cf. \cite[§\,8, n°1, th.~1 et §\,9, n°2, prop.~5.c]{Bki}) donne un isomorphisme d'anneaux
$\Phi:A/{\rm Rad}\,A\xrightarrow{\simeq}\prod_{i=1}^m\M_{d_i}(\FF_{p^{r_i}})$.
On a un diagramme commutatif, où $\pi : A\twoheadrightarrow A/{\rm Rad}\,A$ désigne la projection :
\begin{equation}
\xymatrix{
K[\prod_{i=1}^m\M_{nd_i}(\FF_{p^{r_i}})]\Md \ar[rr]^-{\iota^*_3}\ar[d]^-{\M_n(\Phi)^*}_-{\simeq}
&& K[\prod_{i=1}^m \GL_{nd_i}(\FF_{p^{r_i}})]\Md \ar[d]^-{\GL_n(\Phi)^*}_-{\simeq}
\\
K[\M_n(A/{\rm Rad}\,A)]\Md \ar[rr]^-{\iota^*_2}\ar@{_{(}->}[d]^-{\M_n(\pi)^*}
&& K[\GL_n(A/{\rm Rad}\,A)]\Md \ar@{_{(}->}[d]^-{\GL_n(\pi)^*}
\\
K[\M_n(A)]\Md \ar[rr]^-{\iota^*_1}
&& K[\GL_n(A)]\Md
}\;.\label{diagr1}
\end{equation}
En raisonnant sur ce diagramme commutatif, on peut étendre le théorème \ref{GLn-fini-Fq} à tous les $p$-anneaux finis.

\begin{lm}\label{lm-reduc}
Soient $K$ un corps de décomposition de $\mathbf{P}(A)$ et $G_n=\GL_n$ ou $\M_n$.
Le foncteur $G_n(\pi)^*$ est pleinement fidèle et préserve la simplicité des représentations. Il induit une bijection entre les classes d'isomorphisme de représentations simples de $G_n(A/\mathrm{Rad}\,A)$ et celles de $G_n(A)$.
\end{lm}
\begin{proof}
Comme $G_n(\pi)$ est un morphisme surjectif, le foncteur de restriction associé est pleinement fidèle et préserve la simplicité des représentations. En particulier, il envoie deux représentations simples non isomorphes sur des représentations simples non isomorphes. 
Il reste à montrer que toute représentation simple de $G_n(A)$ s'obtient comme restriction le long de $G_n(\pi)$ d'une représentation (forcément simple) de $G_n(A/\mathrm{Rad}\,A)$. 
Pour cela, on rappelle d'abord que les $(K,A)$-bimodules simples s'obtiennent tous comme restriction de $(K,A/{\rm Rad}\,A)$-bimodules simples le long de $\pi: A\twoheadrightarrow A/{\rm Rad}\,A$. Ceci provient du fait que pour tout $p$-anneau fini $A'$ on a un isomorphisme 
$$K\otimes_\mathbb{Z}(A'/\mathrm{Rad}A')\simeq (K\otimes_\mathbb{Z}A')/\mathrm{Rad}(K\otimes_{\mathbb{Z}}A')\;.$$
(Un tel isomorphisme s'obtient de l'isomorphisme $A'/{\rm Rad}\,A'\simeq (A'/p)/\mathrm{Rad}(A'/p)$ induit par le morphisme d'anneaux $A'\twoheadrightarrow A'/p$, en appliquant le corollaire de 
\cite[§\,12.7 (page~220)]{Bki}). 
Ainsi si on note $B^{[\pi]}$ la restriction d'un $(K,A/{\rm Rad}\,A)$-bimodule le long de $\pi$, le théorème \ref{Mn-fini} et le corollaire \ref{cor-GLfin} montrent, que les représentations simples de $G_n(A)$ sont de la forme $\bigotimes_{i\in I}{M_i}^{[{B_i}^{[\pi]}]}= G_n(\pi)^*(\bigotimes_{i\in I}{M_i}^{[{B_i}]})$.
\end{proof}

\begin{lm}\label{lm-reductens} Soit $K$ un corps de décomposition de $\mathbf{P}(A)$, et $G_n=\GL_n$ ou $\M_n$.
Une représentation de $\prod_{1\le i\le m}G_{nd_i}(\FF_{p^{r_i}})$ est simple si et seulement si c'est un produit tensoriel de représentations simples des $G_{nd_i}(\FF_{p^{r_i}})$. Une telle décomposition tensorielle est unique.
\end{lm}
\begin{proof} D'après le théorème \ref{Mn-fini} et le corollaire \ref{cor-GLfin}, les représentations simples de $G_n(A)$ sont absolument simples. D'après le lemme \ref{lm-reduc} c'est aussi le cas des représentations de $G_n(A/{\rm Rad}\,A)$, donc des facteurs $G_{nd_i}(\FF_{p^{r_i}})$. Le lemme découle alors du corollaire \ref{cor-cas-sympa-tens}.
\end{proof}

\begin{thm}\label{classif-St}
Soit $A$ un $p$-anneau fini, et $K$ un corps de décomposition de $\mathbf{P}(A)$, de caractéristique $p$. Une représentation de $\GL_n(A)$ est simple si et seulement si elle admet une décomposition à la Steinberg.
Si c'est le cas, la représentation simple est absolument simple. Les représentations élémentaires apparaissant dans une décomposition à la Steinberg sont uniques à tensorisation par une puissance $(p-1)$-ième du déterminant près.
\end{thm}

\begin{proof}
Il s'agit d'établir la réciproque du corollaire \ref{cor-GLfin}. Nous utilisons dans la suite les notations du diagramme~\eqref{diagr1}.

Commençons par montrer que les représentations de $\GL_n(A)$ admettant une décomposition à la Steinberg sont simples, ce qui est équivalent à montrer que $\iota_1^*$ préserve la simplicité.
D'après le théorème \ref{GLn-fini-Fq}, si $K$ est un corps de décomposition du corps fini $\FF_q$, le foncteur de restriction:
$\iota^*:K[\M_n(\FF_q)]\Md\to K[\GL_n(\FF_q)]\Md$
préserve la simplicité. D'après le lemme \ref{lm-reductens}, on en déduit que $\iota^*_3$ préserve la simplicité, puis en utilisant le lemme \ref{lm-reduc} que $\iota_1^*$ préserve la simplicité.

Il reste à établir la propriété d'unicité. D'après le théorème \ref{GLn-fini-Fq} et le lemme \ref{lm-reductens}, elle est vérifiée pour $\prod_{i=1}^m\GL_{nd_i}(\FF_{p^{r_i}})=\GL_{n}\big(\prod_{i=1}^n\M_{d_i}(\FF_{p^{r_i}})\big)$. Nous allons maintenant la démontrer pour $\GL_n(A)$.

Pour cela, notons $B_i$ le $(K,A)$-bimodule égal au $K$-espace vectoriel $K^{d_i}$ avec l'action à droite de $A$ donnée par le morphisme d'anneaux $A\twoheadrightarrow \M_n(\FF_{p^{r_i}})\subset \M_n(K)$ induit par l'isomorphisme $\Phi$ du théorème d'Artin-Wedderburn et l'inclusion de $\FF_{p^{r_i}}$ dans $K$ (comme $K$ est un corps de décomposition de $\mathbf{P}(A)$, il contient des corps finis de cardinal $p^{r_i}$ pour tout $i$ et on identifie $\FF_{p^{r_i}}$ à un sous-corps de $K$). On note $^{(j)}B_i$ le bimodule $K_{\phi^j}\otimes_{\mathbb{Z}}B_i$ où $\phi^j:K\to K$ est le morphisme de Frobenius d'élévation à la puissance $p^j$, et $K_{\phi^j}$ est le $(K,K)$-bimodule égal au groupe abélien $K$, muni de l'action $\lambda\cdot x\cdot \mu := \lambda x\phi^j(\mu)$.

\medskip
\noindent
{\em Fait : } les bimodules $^{(j)}B_i$, $1\le i\le m$, $0\le j<r_i$, forment une famille complète de représentants des $(K,A)$-bimodules simples.
\medskip

En effet, l'isomorphisme $\Phi$ du théorème  d'Artin-Wedderburn dit que $A$ possède $m$ modules à droite
simples $S_i$, que chaque $S_i$ a un corps d'endomorphismes $A$-linéaires isomorphe à $\FF_{p^{r_i}}$, et que $S_i$ est isomorphe en tant que
$(\FF_{p^{r_i}},A)$-bimodule au $\FF_{p^{r_i}}$-espace vectoriel ${\FF_{p^{r_i}}}^{\oplus d_i}$ avec action de $A$ donnée par la surjection $A\twoheadrightarrow A/{\rm Rad}\,A \twoheadrightarrow \M_n(\FF_{p^{r_i}})$. On sait alors que tous les $(K,A)$-bimodules simples sont les facteurs de composition des $K\otimes_{\mathbb{Z}}A$-modules $K\otimes_{\mathbb{Z}}S_i$, et on utilise l'isomorphisme de $(K,\FF_{p^r_i})$-bimodules $K\otimes_{\mathbb{Z}}\FF_{p^{r_i}}\simeq K\oplus {}^{(1)}K\oplus\dots \oplus {}^{(r_i-1)}K$ pour conclure que 
$K\otimes_{\mathbb{Z}}S_i\simeq B_i\oplus {}^{(1)}B_i\oplus\dots \oplus {}^{(r_i-1)}B_i$, d'où le fait.

Supposons maintenant que $\bigotimes_{0\le j<r_i}M_{i,j}^{\{j\}}$ est une décomposition tensorielle à la Steinberg d'une représentation $M_i$ de $\GL_{nd_i}(\FF_{p^{r_i}})$. Alors, par définition des bimodules $^{(j)}B_i$, on a
$$G_n(\pi)^*G_n(\Phi)^*M_i\simeq \bigotimes_{0\le j<r_i}M_{i,j}^{[^{(j)}B_i]}\;.$$
En d'autres termes, $G_n(\pi)^*G_n(\Phi)^*$ préserve les décompositions à la Steinberg, et le résultat d'unicité pour $\GL_n(A)$ se déduit donc de celui pour $\prod_{i=1}^m\GL_{nd_i}(\FF_{p^{r_i}})$.
\end{proof}

\section{Représentations EML-polynomiales}\label{seml}

\subsection{Généralités}\label{subsec-generalites}
Si le corps commutatif $K$ est infini, une représentation $V$ de dimension finie du monoïde multiplicatif $\M_n(K)$ est dite \emph{polynomiale} si le morphisme d'action $\rho:\M_n(K)\to {\rm End}_K(V)$ est un polynôme des coefficients des matrices. Les restrictions des représentations polynomiales de $\M_n(K)$ à $\GL_n(K)$ sont les \emph{représentations polynomiales de $\GL_n(K)$}. Cette notion, dont l'idée remonte à la thèse de Schur et dont l'ouvrage de Green \cite{Green} mène une étude systématique, a fait l'objet d'une littérature intensive. La définition suivante donne une notion plus générale et plus souple de représentation \emph{polynomiale à la Eilenberg-MacLane} (\emph{EML-polynomiale} en abrégé).

Dans la suite de cette section \ref{subsec-generalites}, nous fixons un anneau $A$, et nous désignons par $G_n(A)$ l'un des monoïdes $\M_n(A)$, $\GL_n(A)$ ou encore $\SL_n(A)$ si $A$ est commutatif. On parlera de représentation de $G_n(A)$ pour désigner un $K[G_n(A)]$-module.

\begin{defi}\label{def-rep-pol} Une représentation $M$ de $\M_n(A)$ est dite \emph{polynomiale à la Eilenberg-MacLane} si le morphisme d'action $\rho: \M_n(A)\to\mathrm{End}_K(M)$ est une fonction polynomiale au sens d'Eilenberg-MacLane, comme dans la section \ref{subsec-fctpol}. Une représentation de $G_n(A)$ est dite \emph{polynomiale à la Eilenberg-MacLane} si elle est restriction d'une représentation EML-polynomiale de $\M_n(A)$.
\end{defi}

La polynomialité à la Eilenberg-MacLane de $\rho$ fait intervenir la structure \emph{additive} des anneaux $\M_n(A)$ et $\mathrm{End}_K(V)$ (la structure multiplicative ne joue aucun rôle), tandis que le fait que $\rho$ est une représentation signifie la compatibilité à leur structure \emph{multiplicative} (la structure additive ne joue aucun rôle). La polynomialité à la Eilenberg-MacLane est donc une forme de propriété \emph{arithmétique} des représentations.

Si $A=K$ est un corps infini, toute représentation polynomiale de $\M_n(A)$ ou de $\GL_n(A)$ au sens classique \cite{Green} est EML-polynomiale. La proposition suivante permet de construire d'autres exemples à partir de celui-ci.

\begin{pr}\label{prel-eml}
\begin{enumerate}
\item Si $A$ est un $p$-anneau fini et $K$ est de caractéristique $p$, toute représentation de $\M_n(A)$ ou de $\GL_n(A)$ est EML-polynomiale.
\item Si $\phi : B\to A$ est un morphisme d'anneaux, la restriction le long du morphisme de monoïdes $G_n(\phi)$ d'une représentation EML-polynomiale de $G_n(A)$ est une représentation EML-polynomiale de $G_n(B)$.
\item La classe des représentations EML-polynomiales de $G_n(A)$ est stable par sous-quotient et par produit tensoriel.
\end{enumerate}
\end{pr}

\begin{proof}
Pour la première assertion, toute représentation de $\M_n(A)$ est polynomiale d'après l'exemple \ref{ex-fctpol}, et toute représentation $M$ de $\GL_n(A)$ s'obtient comme restriction d'une représentation $r^*M$ de $\M_n(A)$ d'après le lemme \ref{lm-rep-mono-grp}.  

Pour les deux assertions suivantes, on peut supposer $G_n=\M_n$, le cas de $\GL_n$ et de $\SL_n$ s'obtenant par restriction.
La deuxième assertion résulte de ce que la précomposition d'une fonction polynomiale (au sens d'Eilenberg-MacLane) par une fonction additive est polynomiale.
Pour la troisième, montrons par exemple que si $U$ est une sous-représentation d'une représentation EML-polynomiale $V$ de $\M_n(A)$, alors $U$ est polynomiale. Notons $i : U\to V$ l'inclusion. Le diagramme
$$\xymatrix{\M_n(A)\ar[r]^{\rho_U}\ar[d]_-{\rho_V} & \mathrm{End}_K(U)\ar[d]^-{i_*}\\
\mathrm{End}_K(V)\ar[r]^-{i^*} & \mathrm{Hom}_K(U,V),
}$$
où $\rho_U$ et $\rho_V$ sont les morphismes d'action, commute. Comme $\rho_V$ est par hypothèse polynomiale, et $i^*$ additive, $i_*\circ\rho_U=i^*\circ\rho_V$ est polynomiale. La polynomialité de $\rho_U$ découle maintenant de ce que $i_*$ est une fonction additive {\it injective}. La stabilité par produit tensoriel provient de ce que, si $f : M\to U$ et $g : M\to V$ sont des fonctions polynomiales de même source $M$ (où $U$ et $V$ sont des $K$-espaces vectoriels), la fonction $M\to U\otimes V\quad v\mapsto f(v)\otimes g(v)$ est polynomiale.
\end{proof}

La proposition suivante donne le lien entre représentations EML-polynomiales et foncteurs polynomiaux. Nous notons $\pol(A,K)$ la sous-catégorie pleine des foncteurs polynomiaux de $\F(A,K)$, et $K[\M_n(A)]\Mpol$ la sous-catégorie pleine des $K[\M_n(A)]$-modules qui sont EML-polynomiaux.

\begin{pr}\label{pr-pol-pol}
Soit $A$ un anneau. L'évaluation sur $A^n$ induit un foncteur essentiellement surjectif
$$ev_n : \pol(A,K)\to K[\M_n(A)]\Mpol\;.$$
De plus, toute représentation EML-polynomiale simple de $\M_n(A)$ est l'évaluation sur $A^n$ d'un foncteur polynomial simple unique à isomorphisme près.  
\end{pr}

\begin{proof}  Soit $T_n: K[\M_n(A)]\Md\to \F(A,K)$ le prolongement intermédiaire associé à l'évaluation sur $A^n$. Pour tout $K[\M_n(A)]$-module $M$, l'objet $A^n$ est un support et un co-support du foncteur $T_n(M)$. Si $M$ est EML-polynomial, la polynomialité du foncteur $T_n(M)$ s'ensuit grâce au lemme~\ref{lm-sucosu}. L'essentielle surjectivité d'$ev_n$ découle alors de l'isomorphisme $T_n(M)(A^n)\simeq M$ donné par la proposition~\ref{pr-prolongement-interm}. Celle-ci montre également que, si $M$ est simple, $T_n(M)$ est l'unique simple (à isomorphisme près) dont l'évaluation sur $A^n$ est $M$.
\end{proof}

\begin{rem}\label{rq-comp-pol} 
La proposition \ref{pr-pol-pol} est l'analogue du fait que si le corps $K$ est infini, les représentations polynomiales de $\M_n(K)$  au sens classique \cite{Green} s'obtiennent par évaluation sur $K^n$ des foncteurs \emph{strictement polynomiaux} de Friedlander et Suslin \cite{FS} (qui forment une sous-catégorie pleine de $\pol(K,K)$ si $K$ est un corps infini).
\end{rem}

\begin{rem}\label{polEML-ext} Bien que similaires et reliées par la proposition \ref{pr-pol-pol}, les notions de polynomialité pour les foncteurs et pour les représentations présentent des différences significatives. Ainsi, à l'inverse des sous-catégories $\pol_d(A,K)$ qui sont épaisses dans $\F(A,K)$ (voir la section \ref{paragraphe-pol}), une extension de deux représentations EML-polynomiales de $\M_n(A)$ n'est pas nécessairement EML-polynomiale. Ceci se voit déjà pour $n=1$, $A=K=\mathbb{R}$. Soit $\mathbb{R}$ la représentation tautologique du monoïde multiplicatif $\M_1(\mathbb{R})$. C'est une représentation EML-polynomiale qui possède des auto-extensions $V=(\mathbb{R}^2,\rho)$ qui ne sont pas EML-polynomiales. Il suffit de prendre $\rho:\M_1(\mathbb{R})\to\mathcal{M}_2(\mathbb{R})$ tel que $\rho(0)=0$ et  pour $x\ne 0$,  
$$\rho(x)=\left(\begin{array}{cc} x & x.\ln(|x|)\\
                                                0 & x
                                               \end{array}\right)\;.$$ 
\end{rem}

\subsection{Décompositions à la Steinberg}

On fixe jusqu'à la fin de la section un corps commutatif $K$ et un anneau \emph{commutatif} $A$. On rappelle que $T_n:K[\M_n(A)]\to\F(A,K)$ désigne le prolongement intermédiaire associé à l'évaluation sur $A^n$.

Le résultat suivant constitue la clef de l'utilisation des foncteurs pour l'étude de représentations simples de monoïdes ou groupes linéaires ; il motive l'introduction de la notion de polynomialité à la Eilenberg-MacLane pour ces représentations.

\begin{lm}\label{lm-vdf}
Si $K$ est un corps commutatif et $A$ est un anneau commutatif, alors pour toute représentation $M$ de $\M_n(A)$, EML-polynomiale, simple, et de dimension finie sur $K$, le foncteur $T_n(M)$  est polynomial à valeurs de dimensions finies.
\end{lm}

\begin{proof}
Les propositions~\ref{pr-prolongement-interm} et~\ref{pr-pol-pol} montrent la polynomialité et la simplicité de $T_n(M)$. Ce foncteur prend sur $A^n$ la valeur $M$, qui est de dimension finie non nulle ; la proposition~\ref{ann-com-cr} permet d'en déduire la finitude des dimensions de toutes les valeurs de $T_n(M)$.
\end{proof}

Le lemme précédent permet d'appliquer le théorème~\ref{Mn} aux anneaux commutatifs sous la forme suivante.

\begin{thm}\label{st-gl-ann-com} Soient $A$ un anneau commutatif et $K$ un corps de décomposition de $\mathbf{P}(A)$. Soit $M$ une représentation EML-polynomiale de $\M_n(A)$, de dimension finie sur $K$. 
Alors $M$ est simple si et seulement si elle admet une décomposition à la Steinberg. 
Si c'est le cas, alors $M$ est absolument simple, et sa décomposition à la Steinberg est unique.
\end{thm}

Nous pouvons déduire du théorème~\ref{st-gl-ann-com} l'existence de décompositions à la Steinberg pour les représentations EML-polynomiales simples de $\GL_n(A)$ ou $\SL_n(A)$. En effet une telle représentation $M$ est \emph{par définition} la restriction d'une représentation EML-polynomiale $\overline{M}$ de $\M_n(A)$ (nécessairement simple). Le résultat suivant s'obtient en restreignant la décomposition à la Steinberg de $\overline{M}$ à $\GL_n(A)$ ou $\SL_n(A)$.

\begin{cor}\label{cor-GLcom}
Soient $A$ un anneau commutatif et $K$ un corps de décomposition de $\mathbf{P}(A)$. Toute représentation EML-polynomiale simple de $\GL_n(A)$ ou $\SL_n(A)$, de dimension finie sur $K$, admet une décomposition à la Steinberg. 
\end{cor}

\begin{rem}\label{rem-Kth}
Le théorème~\ref{st-gl-ann-com} et le corollaire~\ref{cor-GLcom} peuvent tomber en défaut sur des anneaux commutatifs généraux, même sans quotient fini\,\footnote{Un quotient fini d'ordre inversible dans $K$ d'un anneau donne lieu à des représentations simples non polynomiales de dimensions finies de ses groupes linéaires, analogues aux foncteurs simples antipolynomiaux.}, si l'on omet l'hypothèse <<~EML-polynomiale~>>. En effet, on peut construire des représentations de dimension $1$ de $G_n(A)$ dont la restriction à $\SL_n(A)$ est non triviale -- elles  ne possèdent donc pas de décomposition à la Steinberg, par la remarque~\ref{dec-St-dim1} -- dès que $n$ est assez grand et qu'on dispose d'un caractère non trivial $\phi:SK_1(A)\to K^\times$ du groupe de Whitehead spécial de $A$. (Un tel $\phi$ existe par exemple si $K$ est un corps de caractéristique différente de $2$ et  si $A$ est la $\mathbb{R}$-algèbre des fonctions continues sur le cercle  \cite[Chap. III, ex. 1.4.5]{Weibel}.)

Pour la construction, il suffit de considérer le cas $G_n=\GL_n$, le cas de $\SL_n$ s'en déduit par restriction et celui de $\M_n$ s'en déduit en faisant agir les matrices singulières par $0$, comme dans le lemme \ref{lm-rep-mono-grp}. Comme $A$ est commutatif, on dispose \cite[Chap. III, ex. 1.1.1]{Weibel} d'un morphisme de groupes $\pi:\GL(A)\to SK_1(A)$ dont la restriction à $\SL(A)$ est surjective. Pour $n$ assez grand, $\phi\circ\pi$ se restreint alors en un caractère de $\GL_n(A)$ dont la restriction à $\SL_n(A)$ est non triviale.
\end{rem}

Les décompositions à la Steinberg des représentations de $\GL_n(A)$ sont loin d'être uniques en général, comme le montre l'exemple suivant.

\begin{ex}\label{ex-pas-unique}
Soit $\phi_a:K[X]\to K$ le morphisme de $K$-algèbres tel que $\phi_a(X)=a$. Les  $\phi_a$ induisent tous le \emph{même} isomorphisme $\GL_1(K[X])\to\GL_1(K)$, de sorte que si $K$ est la représentation tautologique de $\GL_1(K)$, les représentations EML-polynomiales simples $K^{[\phi_a]}$ sont toutes isomorphes.
\end{ex}

Dans la suite de la section, nous restreignons notre attention à $\SL_n(A)$, pour lequel le genre de pathologie de l'exemple précédent ne se produit pas. Nous allons compléter l'énoncé du corollaire \ref{cor-GLcom} pour obtenir un analogue du théorème \ref{st-gl-ann-com} pour $\SL_n(A)$. L'outil pour passer de $\M_n(A)$ à $\SL_n(A)$ est un résultat de densité que nous allons maintenant présenter. 
Nous nous placerons dans la situation suivante.

\begin{sit}\label{sit}
Soit $A$ un anneau commutatif, $K$ un corps commutatif, et $f:A\to K^d$ un morphisme d'anneaux (pour la structure d'anneau produit sur $K^d$) dont les composantes sont les morphismes d'anneaux notés $f_i:A\to K$, $1\le i\le d$. On suppose que l'une des deux hypothèses suivantes est vérifiée.
\begin{enumerate}
\item[(H1)] Le corps $K$ est de caractéristique $0$.
\item[(H2)] Le corps $K$ est de caractéristique $p>0$, et pour tout couple $i\ne j$ et tout $t\in\mathbb{N}$ on a $\phi^t\circ f_i\ne f_j$, où $\phi:K\to K$ désigne le morphisme de Frobenius ($\phi(x)=x^p$).
\end{enumerate}
\end{sit}
La proposition suivante est 
due\footnote{Borel et Tits énoncent la proposition \ref{lm-densite} lorsque $A$ est un corps infini, mais leur démonstration fonctionne pour un anneau quelconque.} à Borel et Tits \cite[Prop.~2.1]{Borel-Tits}. Elle est aussi un corollaire du théorème \ref{st-gl-ann-com} pour $n=1$ (voir la remarque \ref{rq-demo-densite}).

\begin{pr}\label{lm-densite}On se place dans la situation \ref{sit} et on suppose de plus que les ensembles $f_i(A)$ sont tous de cardinal infini.
Soit $P\in K[X_1,\dots,X_d]$. Si $P(f_1(x),\dots,f_d(x))=0$ pour tout $x\in A$, alors $P=0$.
\end{pr}

\begin{cor}\label{cor-densite}
On se place dans la situation \ref{sit}.
La clôture de Zariski de $f(A)$ dans $K^d$ est égale au produit des clôtures de Zariski des $f_i(A)$ dans $K$.
\end{cor}

\begin{proof}
Si tous les $f_i$ sont d'images infinies, la proposition~\ref{lm-densite} nous dit que $f(A)$ est Zariski-dense dans $K^d$, ce qui implique le résultat. Sinon,  $K$ est de caractéristique $p>0$, donc (H2) est satisfaite ; par ailleurs, quitte à permuter les coordonnées, on peut supposer que seules les $e>0$ premières coordonnées $f_i$ sont d'images finies. Pour $1\leq i\leq e$, les idéaux $\ker f_i$ sont maximaux (les images des $f_i$ sont des sous-anneaux finis du corps $K$, ce sont donc des corps), et comme l'hypothèse (H2) est vérifiée, les $\ker f_i$ sont deux à deux étrangers. D'après le lemme chinois, le morphisme d'anneaux $\alpha : A\to \prod_{1\le i\le e} \mathrm{Im}\, f_i$ dont les composantes sont les $f_i$ est surjectif. Ceci finit la démonstration si $d=e$. On suppose donc par la suite $0<e<d$.

Soient $I=\ker\alpha$ et $I^+=\mathbb{Z}\oplus I$ l'anneau obtenu en adjoignant formellement une unité à l'idéal $I$, et $\iota:I^+\to A$ le morphisme d'anneaux induit par l'inclusion $I\subset A$.

\medskip
\noindent
{\em Fait : } si $g,h:A\to K$ sont deux morphismes d'anneaux distincts d'images infinies, alors $g\circ\iota$ et $h\circ\iota$ sont également distincts et d'images infinies. 
\medskip

En effet, comme $g$ est d'image infinie et que $A/I$ est fini, il existe $b\in I$ tel que $g(b)\ne 0$. Supposons que $g$ et $h$ coïncident sur $\mathrm{Im}\,\iota\supset I$. Soit $a\in A$ : la relation $ab\in I$ implique $g(ab)=h(ab)$, d'où $g(a)=h(a)$ en divisant par $g(b)=h(b)\in K^\times$.

Considérons maintenant un polynôme $P\in K[X_1,\dots,X_d]$ s'annulant sur $f(A)$. Fixons $x\in \mathrm{Im}\,\alpha$. Soit $a\in A$ tel que $\alpha(a)=x$, notons $(x_1,\dots,x_d)=f(a)\in K^d$,
$$Q_x:=P(x_1,\dots,x_e,x_{e+1}+X_{e+1},\dots,x_d+X_{d})\in K[X_{e+1},\dots,X_d]\quad\text{et}$$
$$R_x:=\prod_{i=0}^{p-1}Q_x(X_{e+1}+i,\dots,X_d+i)\;.$$
Notons $\psi:I^+\to K^{d-e}$ le morphisme d'anneaux de coordonnées les $f_i\circ\iota$, $e<i\le d$. Les $f_i\circ \iota$ sont d'images infinies et satisfont (H2) grâce au fait ci-avant. De plus, $Q_x$ s'annule sur $\psi(I)$, car pour $t\in I$, $Q_x(\psi(t))=P\big(f_1(a+t),\dots,f_d(a+t)\big)=0$. Comme $K$ est de caractéristique $p$, $\psi$ se factorise par la projection $I^+=\mathbb{Z}\oplus I\twoheadrightarrow\mathbb{Z}/p\mathbb{Z}\oplus I$, donc il suit de ce qui précède que $R_x$ est nul sur $\psi(I^+)$. La proposition~\ref{lm-densite} donne alors $R_x=0$, d'où $Q_x=0$, puisque l'anneau $K[X_{e+1},\dots,X_d]$ est intègre. Ceci étant valide pour tout $x\in \mathrm{Im}\,\alpha$, on obtient que $P$ s'annule sur $\mathrm{Im}\,\alpha\times K^{d-e}$, d'où la proposition.
\end{proof}

\begin{cor}\label{cor-densite-SL}
On se place dans la situation \ref{sit}. La clôture de Zariski de l'image du morphisme de groupes
$$
\begin{array}{cccc}
\SL_n(f): & \SL_n(A) &\to & \SL_n(K^d)=\SL_n(K)^{d}\\
& [a_{k,\ell}] & \mapsto & \left(\,[f_i(a_{k,\ell})]\,\right)_{1\le i\le d}
\end{array}$$
est égale au produit des groupes $\SL_n(\overline{f_i(A)})$ où $\overline{f_i(A)}$ désigne la clôture de Zariski dans $K$ de l'image de $f_i$.
\end{cor}
\begin{proof}
Pour tout anneau $B$ commutatif et tout couple d'entier $(r,s)$ tel que $1\le r\ne s\le n$, on note $e_{r,s}^B(\lambda)= 1_n+\lambda e_{r,s}\in\SL_n(B)$
où $e_{r,s}$ est la matrice avec $1$ en position $(r,s)$ et $0$ ailleurs. 
Soit $L=\prod_{1\le i\le d} \overline{f_i(A)}\subset K^d$. Les $e_{r,s}^L(\lambda)$, $\lambda\in L$, $1\le r\ne s\le n$, engendrent le sous-groupe $\SL_n(L)$ de $\SL_n(K^d)=\SL_n(K)^d$. Pour démontrer le corollaire, il suffit donc de montrer que tout polynôme $P$ défini sur l'espace vectoriel $\M_n(K)^d$ et s'annulant sur $\SL_n(f)(A)$, s'annule sur l'image de l'application $e^L:L^N\to\SL_n(L)\subset \M_n(K^d)$ définie par:
$$e^L(\lambda_1,\dots,\lambda_N)=\prod_{1\le k\le N}e_{(i_k,j_k)}^L(\lambda_k)\;,$$
et ceci pour tous les $N>0$ et tous les couples $(i_k,j_k)$.
Mais on a un carré commutatif (où $e^A$ est la \og même\fg~application que $e^L$, définie sur $A$) 
$$\xymatrix{
A^N\ar[d]^-{e^A}\ar[rr]^-{f^N}&& L^N\ar[d]^-{e^L}\\
\SL_n(A)\ar[rr]^-{\SL_n(f)}&&\SL_n(L)
}$$
de sorte que si $P$ s'annule sur l'image de $\SL_n(f)$, alors $P\circ e^L$ s'annule sur l'image de $f^N$.  Comme $P\circ e^L$ est un polynôme, il s'annule sur $L^N$ d'après le corollaire \ref{cor-densite}, donc $P$ s'annule sur l'image de $e^L$.
\end{proof}

Le cas spécial linéaire du résultat de Borel et Tits \cite[Cor.~10.4]{Borel-Tits} montre que si $A$ est un corps infini, toutes les représentations simples de dimension finie de $\SL_n(A)$ admettent une unique décomposition à la Steinberg. Le théorème suivant est plus général: l'anneau commutatif $A$ est quelconque, mais le prix à payer pour cette généralité est qu'il ne concerne que les représentations simples EML-polynomiales. 
\begin{thm}\label{th-SL}
Soient $A$ un anneau commutatif et $K$ un corps de décomposition de $\mathbf{P}(A)$. Soit $M$ une représentation EML-polynomiale de $\SL_n(A)$, de dimension finie sur $K$. 
Alors $M$ est simple si et seulement si elle admet une décomposition à la Steinberg. 
Si c'est le cas, alors $M$ est absolument simple, et sa décomposition à la Steinberg est unique.
\end{thm}
\begin{proof}
L'existence de la décomposition est le corollaire \ref{cor-GLcom}. 
Réciproquement, soit $M=M_1^{[\phi_1]}\otimes\dots\otimes M_m^{[\phi_m]}$, où les $\phi_i : A\to K$ sont des morphismes d'anneaux deux à deux distincts et les $M_i$ des représentations élémentaires de $\SL_n(K)$ ; montrons que $M$ est une représentation simple de $\SL_n(A)$. Supposons tout d'abord $\mathrm{car}(K)=0$. D'après la propriété de densité du corollaire \ref{cor-densite-SL}, il suffit de montrer que $M_1\otimes\dots\otimes M_m$ est simple comme représentation de $\SL_n(K)^m$. Pour cela, il suffit de montrer que chaque $M_k$ est absolument simple comme représentation de $\SL_n(K)$, ce qui suit du fait qu'en caractéristique nulle, les foncteurs élémentaires sont les foncteurs de Schur (proposition \ref{pr-car0}), et induisent donc des représentations simples de $\SL_n(K)$ (voir par exemple \cite[Chap. 6]{FH}). Si $K$ est de caractéristique $p>0$, on se réduit de même en utilisant le corollaire \ref{cor-densite-SL} à démontrer que si $\FF\subset K$ est un corps de cardinal supérieur à $p^r$, alors tout produit tensoriel 
$N_0^{[\phi^0]}\otimes\dots \otimes N_{r-1}^{[\phi^r]}$ où les $N_i$ sont des représentations élémentaires de $\SL_n(K)$, est encore simple après restriction à $\SL_n(\FF)$. Mais les représentations élémentaires du groupe algébrique $\GL_{n,K}$ sont des représentations de plus hauts poids $p$-restreints (cf. proposition \ref{pr-B3}), et donnent donc par restriction des representations simples de $\SL_{n,K}$ de plus hauts poids $p$-restreints. L'absolue simplicité du produit tensoriel découle alors du théorème de Steinberg \cite[Thm 7.4]{St-TPT} \cite[Thm 2.11]{Hum} décrivant la restriction des représentations simples du groupe algébrique $\SL_{n,K}$ aux groupes $\SL_n(\FF)$. 

Enfin, d'après le corollaire \ref{cor-densite-SL}, si on considère deux produits tensoriels $M=M_1^{[f_1]}\otimes\dots\otimes M_d^{[f_d]}$ et $N=N_1^{[f_1]}\otimes\dots\otimes  N_d^{[f_d]}$ où les $M_i$, $N_i$ sont des représentations polynomiales au sens classique du groupe algébrique $\GL_{n,K}$, alors la restriction le long de $f=\prod_{1\le i\le d} f_i:A\to L=\prod_{1\le i\le d} \overline{f_i(A)}$ induit un isomorphisme
$$\mathrm{Hom}_{K[\SL_n(L)]}(M_1\otimes\dots\otimes M_d, N_1\otimes\dots\otimes  N_d) \simeq \mathrm{Hom}_{K[\SL_n(A)]}(M,N)\;.$$
Si les $M_i$ et les $N_i$ sont de dimensions finies, alors le terme de gauche est lui-même isomorphe au produit tensoriel des $\mathrm{Hom}_{K[\SL_n(\overline{f_i(A)}]}(M_i,N_i)$. L'unicité de la décomposition de $M$ et son absolue simplicité en découlent.
\end{proof}

\begin{rem}
Le recours au théorème de Steinberg \cite[Thm 7.4]{St-TPT} dans la démonstration précédente n'est nécessaire que si des corps finis apparaissent comme quotients de $A$. Dans le cas contraire, on peut remplacer l'utilisation de ce théorème par l'observation élémentaire que sur un corps infini, les représentations (rationnelles) du groupe algébrique $\SL_{n,K}$ forment une sous-catégorie pleine stable par sous-quotient de $K[\SL_n(K)]\Md$. 
\end{rem}

\section{Fonctions polynomiales multiplicatives}\label{sta}

Nous déclinons maintenant les résultats de la section~\ref{seml} pour les représentations EML-polynomiales du monoïde multiplicatif sous-jacent à un anneau commutatif.

\begin{pr}\label{cor-arith1}
Soient $A$ un anneau commutatif et $\zeta : A\to K$ une fonction polynomiale (au sens d'Eilenberg-MacLane) de degré $d>0$ et multiplicative au sens où $\zeta(xy)=\zeta(x)\zeta(y)$ pour tous $x$, $y\in A$. Alors il existe une extension finie de corps commutatifs $K\subset L$ de degré au plus $d!$ et des morphismes d'anneaux $f_i : A\to L$, pour $1\leq i\leq d$, tels que
\[\forall x\in A\qquad\zeta(x)=\prod_{i=1}^d f_i(x).\]

De plus, si $g_1,\dots,g_r : A\to L$ sont des morphismes d'anneaux tels que
\[\forall x\in A\qquad\zeta(x)=\prod_{i=1}^r g_i(x),\]
alors $r=d$ et il existe une permutation $\sigma\in\Si_d$ telle que $g_i=f_{\sigma(i)}$ pour tout $i$ dans chacun des trois cas suivants :
\begin{enumerate}
\item[(i)] $\mathrm{car}(K)=0$ ;
\item[(ii)] $\mathrm{car}(K)=p>0$ et pour tout indice $i$, l'ensemble $\{j\,|\,g_j=g_i\}$ est de cardinal strictement inférieur à $p$ ;
\item[(iii)] $r\le d$.
\end{enumerate}
\end{pr}

\begin{proof}  
Le morphisme $\zeta$ définit une représentation absolument simple EML-polynomiale de dimension $1$ sur $K$ du monoïde multiplicatif $\M_1(A)$, notée $K^{[\phi]}$. Si $K$ est un corps de décomposition de $\mathbf{P}(A)$, alors d'après le théorème \ref{st-gl-ann-com}, $K^{[\phi]}$ se décompose comme un produit tensoriel $M_1^{[\phi_1]}\otimes\dots M_n^{[\phi_n]}$, où les $M_i$ sont des représentations élémentaires de degrés $d_i$ et de dimension $1$. Donc pour tout $a\in A$ on a $\zeta(a)=\phi_1(a)^{d_1}\dots \phi_n(a)^{d_n}$ avec $d=\sum d_i$, d'où l'existence. 

Dans le cas où $K$ n'est pas un corps de décomposition de $\mathbf{P}(A)$, la démonstration de l'existence de la décomposition est la même que dans le théorème~\ref{st-gl-ann-com} (pour $n=1$), en utilisant le théorème~\ref{st-pol-petit-corps} (et la proposition~\ref{pr-fin-cr} pour la borne sur le degré -- cf. remarque~\ref{rq-taille-corps}) au lieu du théorème~\ref{steinberg-pol-gal} (noter qu'il n'y a pas besoin de filtration car le foncteur considéré prend une valeur de dimension $1$).

Pour établir les résultats d'unicité partielle, quitte à plonger $K$ dans sa clôture algébrique, on peut supposer $K$ algébriquement clos. Supposons
$\zeta=\prod_{i=1}^m g_i^{a_i}$,
où les $g_i$ sont des morphismes d'anneaux $A\to K$ deux à deux distincts et $a_i\in\mathbb{N}^*$ avec $\mathrm{car}(K)=0$, ou $a_i<\mathrm{car}(K)$ pour tout $i$. Cela implique que la puissance symétrique $S^{a_i}$ (sur $K$) est un endofoncteur élémentaire des $K$-espaces vectoriels, donc que $S^{a_i}(K)$ est une représentation élémentaire de $\M_1(K)$. Notre décomposition se traduit donc par un isomorphisme
$K^{[\zeta]}\simeq S^{a_1}(K)^{[g_1]}\otimes\dots\otimes S^{a_m}(K)^{[g_m]}$
et l'unicité découle alors de l'unicité dans le théorème \ref{st-gl-ann-com}. Il reste à établir la propriété d'unicité dans le cas (iii). Celui-ci se ramène à (ii), car si $\zeta=\prod_{i=1}^m g_i^{a_i}$ avec $\sum_i a_i\le d$, alors $a_i<p:=\mathrm{car}(K)$ pour tout $i$. En effet, comme les $g_i^p$ sont des morphismes d'anneaux, si $a_i\ge p$ pour un $i$, le degré polynomial de $\zeta$ serait au plus égal à $(\sum_i a_i)+1-p\le d+1-p<d$, contradiction qui achève la démonstration.
\end{proof}

\begin{rem}\label{rq-demo-densite}
La proposition \ref{lm-densite} de Borel-Tits peut être obtenue comme conséquence de l'unicité de la proposition~\ref{cor-arith1} de la façon suivante. Soient $f_i:A\to K$ des morphismes d'anneaux d'images infinies vérifiant les hypothèses (H1) ou (H2) de la situation~\ref{sit}. Si $d=(d_1,\dots,d_n)\in\mathbb{N}^n$, on note $\psi_d(a)=f_1(a)^{d_1}\dots f_n(a)^{d_n}$. Alors les fonctions $\psi_d:A\to K$ sont des caractères du monoïde multiplicatif $A$, qui sont deux à deux distincts d'après la proposition~\ref{cor-arith1} (si $\mathrm{car}(K)=p>0$, il faut écrire $f_i(a)^d_i=\prod_{1\le i\le k}(\phi^{k}\circ f_i)(a)^{\alpha_k(d_i)}$, où $\phi$ désigne le morphisme de Frobenius et les $\alpha_k(d_i)\in[0,p-1]$ sont les chiffres de la décomposition $p$-adique de $d_i$, pour pouvoir appliquer la proposition). La proposition \ref{lm-densite} découle alors de l'indépendance linéaire des caractères d'un monoïde \cite[Chap. V, §\,6, cor. 1]{Bki2}. 
\end{rem}

Dans ce qui suit, on reformule une partie des résultats de la proposition~\ref{cor-arith1}. Nous introduisons à cet effet quelques notations.
Si $A$ et $B$ sont des anneaux, notons $\mathrm{Pol}_\mu(A,B)$ l'ensemble des applications $f : A\to B$ polynomiales (au sens d'Eilenberg-MacLane) qui préservent la multiplication ($f(xy)=f(x)f(y)$ pour tous $x$, $y$ dans $A$) et l'unité. La multiplication (point par point) fait de $\mathrm{Pol}_\mu(A,B)$ un sous-monoïde du monoïde multiplicatif sous-jacent à l'anneau $\mathrm{Pol}(A,B)$. Si l'on suppose $B$ commutatif, il est commutatif, de sorte que l'inclusion dans $\mathrm{Pol}_\mu(A,B)$ de l'ensemble $\mathbf{Ann}(A,B)$ des morphismes d'anneaux de $A$ dans $B$ s'étend de façon unique en un morphisme de monoïdes commutatifs $\Phi_{A,B} : \mathbb{N}[\mathbf{Ann}(A,B)]\to\mathrm{Pol}_\mu(A,B)$, où $\mathbb{N}[E]$ désigne le monoïde commutatif libre sur un ensemble $E$. 
L'énoncé suivant découle directement de la proposition~\ref{cor-arith1}.

\begin{cor}\label{pr-mult-pol}
Soit $A$ un anneau commutatif.
\begin{enumerate}
\item Si $K$ est algébriquement clos, le morphisme de monoïdes $\Phi_{A,K}$ est surjectif.
\item Si $K$ est de caractéristique nulle, alors $\Phi_{A,K}$ est injectif.
\item Si $K$ est de caractéristique $p>0$, notons $\overline{\mathbf{Ann}}(A,K)$ le quotient de l'ensemble $\mathbf{Ann}(A,K)$ par la relation d'équivalence engendrée par $f\sim f^p$, et choisissons une section à la projection de façon à voir $\overline{\mathbf{Ann}}(A,K)$  comme un sous-ensemble de $\mathbf{Ann}(A,K)$. Alors la restriction à $\mathbb{N}[\overline{\mathbf{Ann}}(A,K)]$ de $\Phi_{A,K}$ est injective.
\end{enumerate}
\end{cor}

Notons $G(K)$ le monoïde des endomorphismes de corps de $K$ (qui est un groupe si $K$ est algébriquement clos et de degré de transcendance fini sur son sous-corps premier, ou de degré fini sur son sous-corps premier). Comme précédemment, on dispose d'un morphisme d'anneaux $\chi_K : \mathbb{Z}[G(K)]\to {\rm End}(K^\times)$, et l'image du sous-semi-anneau $\mathbb{N}[G(K)]$ de $\mathbb{Z}[G(K)]$ est incluse dans l'ensemble des morphismes multiplicatifs $K^\times\to K^\times$ qui se prolongent en une fonction polynomiale $K\to K$.

\begin{cor}\label{cor-mult-pol} 
\begin{enumerate}
\item Si $K$ est algébriquement clos, alors $\chi_K(\mathbb{N}[G(K)])$ est exactement l'ensemble des morphismes multiplicatifs $K^\times\to K^\times$ qui se prolongent en une fonction polynomiale $K\to K$.
\item Si $K$ est de caractéristique nulle, alors $\chi_K$ est injectif.
\item Si $K$ est de caractéristique $p>0$, alors le noyau de $\chi_K$ est l'idéal (bilatère) de $\mathbb{Z}[G(K)]$ engendré par l'élément central $\phi-p$, où $\phi\in G(K)$ désigne le morphisme de Frobenius.
\end{enumerate}
\end{cor}

\part{Applications aux propriétés de finitude des foncteurs}\label{part-finitude}

Nous allons utiliser les résultats de la partie~\ref{part-resfond} pour démontrer des résultats qualitatifs sur les foncteurs, qui concernent trois propriétés de finitude fondamentales : la longueur finie, la noethérianité et la présentation finie. L'étude de ces propriétés est notoirement difficile dans les catégories de foncteurs. Auslander les a abordées (dans \cite{Aus82} par exemple), dans des catégories de foncteurs {\it additifs} sur une catégorie {\it abélienne} vérifiant de fortes hypothèses de finitude, obtenant alors des caractérisations complètes liées à des propriétés importantes en théorie des représentations, comme le type de représentation fini. Nos résultats traitent de foncteurs à valeurs de dimensions finies beaucoup plus généraux, non nécessairement additifs. Grâce à nos théorèmes de décomposition à la Steinberg, on peut ramener l'étude de propriétés de finitude de foncteurs à valeurs de dimensions finies à des questions analogues pour des foncteurs \phs\ (ou polynomiaux) d'une part, et pour des foncteurs antipolynomiaux d'autre part ; le cas polynomial peut souvent lui-même se ramener au cas additif.

Dans la section~\ref{snoeth}, nous discutons du problème du caractère localement noethérien ou localement fini de la catégorie $\F^{\df}(\A;K)$, dont nous commençons par voir qu'il se présente de façon fort différente du problème analogue pour $\F(\A;K)$. Nous étudions ensuite (section~\ref{sptf}) le problème de la stabilité par produit tensoriel de la classe des foncteurs de longueur finie et à valeurs de dimensions finies. Enfin, dans la section~\ref{spfff}, nous abordons le problème de la présentation finie des foncteurs simples (ou finis) à valeurs de dimensions finies.

Certains résultats de cette partie seront généralisés dans plusieurs travaux en cours.

\section{Propriétés localement noethérienne et localement finie dans $\F^{\df}(\A;K)$}\label{snoeth}

Nous commencerons par nous intéresser à la propriété localement finie pour les foncteurs antipolynomiaux. Le cas d'une source $\mathbf{P}(A)$, où $A$ est un anneau, donne lieu à bien davantage de résultats en raison du théorème~\ref{thpss}, dû à Putman-Sam-Snowden, qui garantit la noethérianité locale de $\F(A,K)$ lorsque $A$ est fini. Pour un anneau $A$ quelconque, nous verrons qu’on peut en déduire que tout foncteur antipolynomial de $\F(A,K)$ est localement fini. Cette propriété se démontre sans utiliser les résultats du reste de l’article.

Lorsque $K$ est de caractéristique nulle, nous en déduirons au théorème~\ref{th-lf-car0}, grâce à notre décomposition à la Steinberg globale, que tout foncteur de $\F^{\df}(A,K)$  est localement fini.

Lorsque $K$ est de caractéristique non nulle,
le lien nettement plus subtil entre foncteurs \phs\ et polynomiaux rend le problème plus délicat. Toutefois, nous donnerons des conjectures, étudiées dans un travail en cours.

\subsection{Objets noethériens, artiniens et finis.}
Un objet d'une catégorie abélienne $\V$ est {\em noethérien} (resp. {\em artinien}) si toute famille de sous-objets possède un élément maximal (resp. minimal) pour l'inclusion. La classe des objets noethériens (resp. artiniens) de $\V$ est stable par somme directe finie, sous-quotient et extensions. Un objet de $\V$ est noethérien (resp. artinien) si et seulement si tous ses sous-objets (resp. tous ses quotients) sont de type fini (resp. de type cofini).
Un objet de $\V$ est de longueur finie (on dira souvent simplement {\em fini} dans la suite) si et seulement s'il est noethérien et artinien. De manière équivalente, un objet est fini si et seulement s'il possède une filtration finie dont les sous-quotients sont simples. 

La catégorie $\V$ est dite {\em localement noethérienne} (resp. {\em localement finie}, {\em localement artinienne}) si elle est engendrée par ses objets noethériens (resp. finis, artiniens). Par exemple, si $A$ est un anneau, la catégorie des $A$-modules à gauche est localement noethérienne si et seulement si $A$ est noethérien à gauche. Elle est localement finie si et seulement si elle est localement artinienne, ce qui équivaut à dire que $A$ est artinien à gauche. Si $\V$ est une sous-catégorie épaisse d'une catégorie de Grothendieck $\E$, alors $\V$ est localement noethérienne (resp. localement finie, localement artinienne) si et seulement si tout objet de $\V$ est localement noethérien (resp. localement fini, localement artinien) au sens où il est la colimite de ses sous-objets noethériens (resp. finis, artiniens). Nous avons en vue le cas de la sous-catégorie épaisse $\F^{\df}(\A;K)$ (qui n'est pas cocomplète) de la catégorie de Grothendieck $\F(\A;K)$.

\subsection{Noethérianité locale de $\F(\A;K)$}

La proposition suivante montre que la noethérianité locale de $\F(\A;K)$ est un phénomène relativement exceptionnel.

\begin{pr}\label{rq-anninf} Si la catégorie $\F(\A;K)$ est localement noethérienne, alors tous les ensembles de morphismes de $\A$ sont finis.
\end{pr}

\begin{proof}
Comme $\A(x,y)\subset \A(x\oplus y,x\oplus y)$, il suffit de montrer que pour tout objet $x$ de $\A$, l'anneau $A:=\mathrm{End}_\A(x)$ est fini. 

Considérons le foncteur
$$\Phi : K[\GL_2(A)]\Md\to\F(\A;K)\qquad M\mapsto K[\mathrm{Inj}(x\oplus x,-)]\underset{K[\GL_2(A)]}{\otimes}M,$$
où $\mathrm{Inj}(x\oplus x,a)\subset\A(x\oplus x,a)$ est le sous-ensemble des monomorphismes scindés ; comme foncteur, $K[\mathrm{Inj}(x\oplus x,-)]$ est un {\it quotient} de $P_\A^{x\oplus x}$, et la structure $K[\GL_2(A)]$-linéaire provient de l'isomorphisme canonique $\GL_2(A)\simeq\mathrm{Aut}_\A(x\oplus x)$. Alors $\Phi$ est un foncteur {\it exact} (car $\mathrm{Inj}(x\oplus x,a)$ est un $\GL_2(A)$-ensemble libre) et fidèle (car l'évaluation en $x\oplus x$ en fournit une rétraction). Il s'ensuit que, si $\Phi(M)$ est un foncteur noethérien, alors $M$ est un $K[\GL_2(A)]$-module noethérien. Comme $\Phi$ envoie $K[\GL_2(A)]$ sur le foncteur de type fini $K[\mathrm{Inj}(x\oplus x,-)]$, cela montre que la noethérianité locale de $\F(\A;K)$ implique que l'anneau $K[\GL_2(A)]$ est noethérien à gauche.

Mais ceci n'est jamais vrai lorsque $A$ est un anneau infini (cf. par exemple \cite[§\,2.6]{PSam}\,\footnote{Cette référence suppose l'anneau commutatif, mais donne une démonstration qui s'applique sans changement à un anneau infini arbitraire.}), d'où la proposition. 
\end{proof}

\begin{nota}\label{notalnc} Nous désignerons par $\lnc(\A,K)$ la condition suivante :  {\it pour tout idéal $K$-cotrivial $\I$ de $\A$, la catégorie $\F(\A/\I;K)$ est localement noethérienne}.
\end{nota}

Les lettres \textbf{\textsc{l}}, \textbf{\textsc{n}}, \textbf{\textsc{c}} abrègent \textbf{\textsc{l}}ocalement, \textbf{\textsc{n}}oethérienne et \textbf{\textsc{c}}otrivial.

Pour la démonstration du profond théorème ci-dessous, voir Putman-Sam \cite[théorème~C]{PSam} et Sam-Snowden \cite[corollaire~8.3.3]{SamSn}.

\begin{thm}[Putman-Sam-Snowden]\label{thpss} Si $A$ est un anneau fini, la catégorie $\F(A,K)$ est localement noethérienne.
\end{thm}

\begin{cor}\label{corlnc} Pour tout anneau $A$, la condition $\lnc(\mathbf{P}(A),K)$ est vérifiée.
\end{cor}

\subsection{Finitude locale des foncteurs antipolynomiaux}\label{ssfp} Le résultat principal de ce paragraphe (proposition~\ref{corsam}) concerne une catégorie source $\mathbf{P}(A)$, mais il repose sur le lemme~\ref{lm-dual-fini} et la proposition~\ref{pr-biz} ci-après, qui s'appliquent à une catégorie additive arbitraire $\A$. Dans les démonstrations de ces propriétés, nous ferons usage du foncteur de dualité $D : \F(\A;K)^{\op}\to\F(\A^{\op};K)$ donné par la post-composition par le foncteur de dualité $\mathrm{Hom}_K(-,K):(K\Md)^{\op}\to K\Md$ des $K$-espaces vectoriels. Le foncteur $D$ est exact et fidèle ; si $F$ est un foncteur de $\F^{\df}(\A;K)$, alors $F$ est de type cofini dans $\F(\A;K)$ si et seulement si $D(F)$ est de type fini dans $\F(\A^{\op};K)$, car $D$ induit une équivalence entre $\F^{\df}(\A;K)^{\op}$ et $\F^{\df}(\A^{\op};K)$.

\begin{lm}\label{lm-dual-fini} On suppose que $\A$ est $K$-triviale.
\begin{enumerate}
\item Si $K$ contient les racines de l'unité, alors on dispose d'un isomorphisme $K[\A(x,-)]\simeq K^{\A(x,-)^\sharp}$ dans $\F(\A;K)$ pour tout objet $x$ de $\A$, où $(-)^\sharp={\rm Hom}_\mathbb{Z}(-,\mathbb{Q}/\mathbb{Z})$ désigne la dualité des groupes abéliens finis.
\item Si la catégorie $\mathbf{Add}(\A;\mathbb{Z})$ est localement artinienne, alors le foncteur $P^x_\A=K[\A(x,-)]$ de $\F(\A;K)$ est de type cofini pour tout objet $x$ de $\A$.
\end{enumerate}
\end{lm}

\begin{proof} Si $K$ contient les racines de l'unité et que $V$ est un groupe abélien fini tel que $V\otimes_{\mathbb{Z}}K=0$, alors $V^\sharp$ est naturellement isomorphe au groupe abélien des morphismes de $V$ vers le groupe multiplicatif $K^\times$. De plus, $K[V]$ est une $K$-algèbre semi-simple naturellement isomorphe à l'algèbre produit $K^{V^\sharp}$. Cela fournit la première assertion.

Commençons par montrer la deuxième assertion lorsque $K$ contient les racines de l'unité : si $\mathbf{Add}(\A;\mathbb{Z})$ est localement artinienne, alors le foncteur additif de type fini $\A(x,-)$ est artinien pour tout $x\in {\rm Ob}\,\A$. Il s'ensuit que son dual $\A(x,-)^\sharp$ est un foncteur noethérien, donc de type fini, de $\mathbf{Add}(\A^{\op};\mathbb{Z})$, c'est-à-dire qu'il est isomorphe à un quotient d'un foncteur représentable $\A(-,t)$. Par conséquent, $K^{\A(x,-)^\sharp}$ est isomorphe à un sous-foncteur de $K^{\A(-,t)}$ (dans $\F(\A;K)$), et son dual est isomorphe à un quotient du projectif $P^t_{\A^{\op}}$ (dans $\F(\A^{\op};K)$). Comme ce foncteur est à valeurs de dimensions finies, cela montre que $P^x_\A$ est de type cofini.

Traitons maintenant le cas général : soit $K\subset L$ une extension de corps commutatifs, où $L$ contient les racines de l'unité. Le foncteur $L[\A(x,-)]$ a un dual $L^{\A(x,-)}$ de type fini dans $\F(\A^{\op};L)$, d'après ce qu'on vient de voir. Il s'ensuit que $K^{\A(x,-)}$ est de type fini dans $\F(\A^{\op};K)$, car $K^{\A(x,-)}\otimes L\simeq L^{\A(x,-)}$ (puisque les $\A(x,t)$ sont finis), et le foncteur d'extension des scalaires $-\otimes L : \F(\A^{\op};K)\to\F(\A^{\op};L)$ est exact, fidèle, et commute aux colimites.
Par conséquent, $K[\A(x,-)]$ est de type cofini dans $\F(\A;K)$.
\end{proof}

\begin{pr}\label{pr-biz}On suppose que $\A$ est $K$-triviale.
On suppose également que les catégories $\F(\A;K)$ et $\F(\A^{\op};K)$ sont localement noethériennes et que la catégorie $\mathbf{Add}(\A;\mathbb{Z})$ est localement artinienne.

Alors la catégorie $\F(\A;K)$ est localement finie.
\end{pr}

\begin{proof} Soit $x$ un objet de $\A$. Le foncteur de type fini $P^x_\A$ de $\F(\A;K)$ est noethérien puisque cette catégorie est localement noethérienne. Le lemme~\ref{lm-dual-fini} montre par ailleurs qu'il est de type cofini. Comme il est à valeurs de dimensions finies, cela signifie que son dual est de type fini dans la catégorie localement noethérienne $\F(\A^{\op};K)$, donc noethérien. En reprenant le dual on en déduit que $P^x_\A$ est artinien. Par conséquent, les générateurs $P^x_\A$ de $\F(\A;K)$ sont finis et cette catégorie est localement finie.
\end{proof}

\begin{pr}\label{corsam} Soit $A$ un anneau fini tel que $A\otimes_\mathbb{Z} K=0$. La catégorie $\F(A,K)$ est localement finie.
\end{pr}

\begin{proof} Les catégories $\F(A,K)$ et $\F(A^{\op},K)$ sont localement noethériennes d'après le théorème~\ref{thpss}. De plus, $\mathbf{Add}(\mathbf{P}(A);\mathbb{Z})$ est équivalente (proposition~\ref{additifs-bimodules}) à $A^{\op}\Md$, qui est localement finie puisque $A$ est un anneau fini. La conclusion résulte donc de la proposition~\ref{pr-biz}. 
\end{proof}

\begin{cor}\label{corflf} Pour tout anneau $A$, la sous-catégorie des foncteurs antipolynomiaux de $\F(A,K)$ est localement finie.
\end{cor}

\subsection{Propriétés de finitude de $\F^{\df}(A,K)$} Nous nous concentrons maintenant sur les propriétés de finitude des foncteurs à valeurs de dimensions finies, en commençant par le cas où $K$ est de caractéristique nulle. 

\begin{thm}\label{th-lf-car0} Soit $A$ un anneau. On suppose le corps $K$ de caractéristique $0$. Alors la catégorie abélienne $\F^{\df}(A,K)$ est localement finie.
\end{thm}

\begin{proof} Il suffit de montrer qu'un foncteur de type fini $F$ de $\F^{\df}(A,K)$ est fini. D'après le corollaire~\ref{cor-glob-car0}, on peut supposer que $F$ est la précomposition par la diagonale $\mathbf{P}(A)\to\mathbf{P}(A)\times\mathbf{P}(A)$ d'un bifoncteur de $\F(\mathbf{P}(A)\times\mathbf{P}(A);K)$ polynomial par rapport à la première variable et antipolynomial par rapport à la seconde. Autrement dit, via l'isomorphisme de catégories $\F(\mathbf{P}(A)\times\mathbf{P}(A);K)\simeq\fct(\mathbf{P}(A),\F(A,K))$, ce foncteur est polynomial, et à valeurs dans la sous-catégorie abélienne des foncteurs antipolynomiaux de $\F(A,K)$. Comme $F$ est de type fini, ces valeurs sont également de type fini dans $\F(A,K)$, de sorte que le corollaire~\ref{corflf} montre qu'elles sont en fait finies. On conclut grâce au lemme~\ref{lm-polfini} ci-dessous.
\end{proof}

\begin{lm}\label{lm-polfini} Soient $A$ un anneau, $\V$ une catégorie abélienne et $F$ un foncteur de $\fct(\mathbf{P}(A),\V)$, polynomial de degré $d$. Si $F(A^i)$ est un objet fini de $\V$ pour tout entier $0\le i\le d$, alors $F$ est un foncteur fini de $\fct(\mathbf{P}(A),\V)$.
\end{lm}

\begin{proof} Le foncteur
\[\Phi : \pol_d(\mathbf{P}(A),\V)\to\V^{d+1}\qquad F\mapsto (cr_i(F)(A,\dots,A))_{0\le i\le d}\]
est exact et fidèle (la fidélité provient de la décomposition \eqref{eq-dec-cr}, page~\pageref{eq-dec-cr}, qui montre que l'annulation de $\Phi(F)$ entraîne celle de $F(A^n)$ pour {\it tout} $n\in\mathbb{N}$). Par conséquent, si $\Phi(F)$ est un objet noethérien (resp. artinien, fini) de $\V^{d+1}$, il en est de même pour $F$. La conclusion s'ensuit.
\end{proof}

Si $K$ est de caractéristique non nulle, l'absence de scindement par le degré des foncteurs \phs\ rend la situation beaucoup plus délicate, et les foncteurs à valeurs de dimensions finies ne sont généralement pas localement finis (voir les travaux de Powell \cite{GP-cow}). Nous conjecturons alors les résultats suivants :

\begin{conj}\label{conj-lf-carp} Supposons que $K$ est de caractéristique $p>0$. Soit $A$ un anneau.
\begin{enumerate}
\item La catégorie $\F^{\df}(A,K)$ est localement noethérienne.
\item Si $A$ ne possède aucun idéal à droite $I$ tel que $A/I$ soit un $p$-groupe fini non nul, alors la catégorie $\F^{\df}(A,K)$ est localement finie.
\end{enumerate}
\end{conj}

Le cas particulier suivant découle des corollaires~\ref{cor-tfap} et~\ref{corflf}.

\begin{pr}\label{pr-conjlf} Si $A\otimes_\mathbb{Z}K=0$, la conjecture~\ref{conj-lf-carp} est vérifiée.
\end{pr}

\section{Produit tensoriel de foncteurs finis de $\F^{\df}(\A;K)$}\label{sptf}

\begin{pr}\label{ptens-pol}  Le produit tensoriel de deux foncteurs polynomiaux finis et à valeurs de dimensions finies de $\F(\A;K)$ est fini.
\end{pr}

\begin{proof} Cela découle du théorème~\ref{st-pol-petit-corps}, de la proposition~\ref{prop-L-p} et du caractère exact et fidèle de tout foncteur d'extension des scalaires au but $-\otimes L : \F(\A;K)\to\F(\A;L)$, où $L$ est un surcorps de $K$.
\end{proof}

\begin{rem} Sans hypothèse de valeurs de dimensions finies, il n'est pas vrai qu'un produit tensoriel de foncteurs polynomiaux finis, même {\em additifs}, reste toujours fini. Considérons ainsi une extension de degré infini de corps commutatifs $K\subset L$ et notons $\pi$ le foncteur d'oubli de $\mathbf{P}(L)$ dans $K\Md$ : c'est un foncteur additif simple de $\F(L,K)$. Pourtant, $\pi^{\otimes 2}$ n'est pas fini. En effet, le bifoncteur biadditif $\pi^{\boxtimes 2}$ n'est pas fini (utiliser la proposition~\ref{additifs-bimodules} et le fait que la $K$-algèbre $L\otimes L$ n'est pas artinienne), et il est facteur direct de $cr_2(\pi^{\otimes 2})$, or $cr_2$ envoie un foncteur polynomial {\it fini} de degré $2$ sur un bifoncteur biadditif fini, d'après la proposition~\ref{pr-fin-cr}. 
\end{rem}

\begin{pr}\label{pr-pts} Si les catégories $\F(\A;K)$ et $\F(\A^{\op};K)$ sont localement noethériennes, alors le produit tensoriel de deux foncteurs de longueur finie et à valeurs de dimensions finies de $\F(\A;K)$ est de longueur finie.
\end{pr}

\begin{proof} Comme $\F(\A;K)$ est localement noethérienne, il y a équivalence entre noethérianité et type-finitude dans cette catégorie, de sorte que la proposition~\ref{pr-pttf} montre qu'un produit tensoriel de foncteurs noethériens y est également noethérien.

Un foncteur $F$ de $\F^{\df}(\A;K)$ est artinien si et seulement si son dual $D(F)$ (introduit au début du §\,\ref{ssfp}) est un foncteur noethérien de $\F^{\df}(\A^{\op};K)$. Par conséquent, en appliquant ce qui précède à la catégorie additive $\A^{\op}$, on voit que le produit tensoriel de deux foncteurs artiniens est artinien dans $\F^{\df}(\A;K)$.

La conclusion s'obtient du fait que les foncteurs finis sont les foncteurs à la fois artiniens et noethériens.
\end{proof}

\begin{thm}\label{th-ptens-fini} Si, pour tout idéal $K$-cotrivial $\I$ de $\A$, les catégories $\F(\A/\I;K)$ et $\F((\A/\I)^{\op};K)$ sont localement noethériennes, alors la classe des foncteurs de longueur finie de $\F^{\df}(\A;K)$ est stable par produit tensoriel.
\end{thm}

\begin{proof} Il suffit de démontrer que le produit tensoriel de deux foncteurs simples à valeurs de dimensions finies est de longueur finie. Supposons d'abord le corps $K$ assez gros (par exemple, contenant toutes les racines de l'unité). Le corollaire~\ref{cor-tens-St1} permet de se ramener à démontrer deux cas distincts: (i) le cas où les deux foncteurs simples sont polynomiaux, et (ii) le cas où les deux foncteurs simples factorisent par un même idéal $K$-cotrivial $\I$. Le cas (i) est démontré dans la proposition \ref{ptens-pol}. D'autre part, comme $\F(\A/\I;K)$ s'identifie (via la restriction le long de $\pi_\I:\A\to \A/\I$) à une sous-catégorie de $\F(\A;K)$ stable par sous-quotient, on voit qu'un foncteur $F$ est fini dans $\F(\A/\I;K)$ si et seulement si $\pi_\I^*(F)$ est fini dans $\F(\A;K)$. Le cas (ii) se déduit donc de la proposition \ref{pr-pts}.

Le cas général se ramène au cas où $K$ est assez gros par un argument de changement de base au but (cf. proposition~\ref{pr-corpassegro}).
\end{proof}

En utilisant le corollaire~\ref{corlnc} et l'équivalence $\mathbf{P}(A)^{\op}\simeq\mathbf{P}(A^{\op})$, on en déduit :

\begin{cor}\label{cortensPA} Pour tout anneau $A$, la classe des foncteurs de longueur finie de $\F^{\df}(A,K)$ est stable par produit tensoriel.
\end{cor}

\section{Présentation finie des foncteurs finis}\label{spfff}

L'étude de la présentation finie dans $\F(\A;K)$, qui est préservée par produit tensoriel, est grandement facilitée par nos décompositions à la Steinberg. Nous ramènerons ainsi, au théorème~\ref{th-pol-pf}, la présentation finie des foncteurs polynomiaux simples (ou finis) à valeurs de dimensions finies à celle de foncteurs additifs. Pour les simples antipolynomiaux, nous montrerons la présentation finie lorsque les foncteurs antipolynomiaux forment une catégorie localement noethérienne (cf. section~\ref{snoeth}) modulo une hypothèse de finitude sur les idéaux cotriviaux.

Nous obtiendrons notamment, au théorème~\ref{thapf}, une caractérisation complète des anneaux $A$ tels que tous les foncteurs simples (ou finis) à valeurs de dimensions finies de $\F(A,K)$ soient de présentation finie.

Nous ne nous intéresserons à la présentation finie que dans la catégorie $\F(\A;K)$ ou sa sous-catégorie de foncteurs additifs, mais pas dans leurs sous-catégories de foncteurs à valeurs de dimensions finies, qui ne sont pas des catégories de Grothendieck, même si les foncteurs que nous considérerons y appartiennent. 

\smallskip

On rappelle qu'un objet $X$ d'une catégorie de Grothendieck $\E$ est dit {\it de présentation finie }si pour tout foncteur $\Phi$ d'une petite catégorie filtrante $I$ vers $\E$, l'application canonique suivante est bijective :
\[\underset{I}{\col} \E(X,-)\circ\Phi\to \E(X,\underset{I}{\col}\Phi).\]

Les propriétés suivantes sont bien connues, on pourra se référer par exemple à Popescu \cite[§\,3.5, énoncés 5.9 à 5.13]{Pop} pour leur démonstration.

\begin{pr}\label{prpfclas}\begin{enumerate}
\item Tout objet de présentation finie de $\E$ est de type fini. Tout objet projectif de type fini de $\E$ est de présentation finie.
\item\label{itpfext} Si dans une suite exacte courte $0\to Y\to X\to Z\to 0$ de $\E$, $Y$ et $Z$ sont de présentation finie, alors $X$ est de présentation finie.
\item Supposons que $\E$ possède une famille génératrice $\G$ constituée d'objets projectifs de type fini. Étant donné un objet $X$ de $\E$, les assertions suivantes sont équivalentes :
\begin{enumerate}\label{itpfcns}
\item  $X$ est de présentation finie ;
\item il existe une suite exacte $Q\to P\to X\to 0$ de $\E$ dans laquelle $P$ et $Q$ sont des sommes directes finies d'éléments de $\G$ ;
\item $X$ est de type fini et pour toute suite exacte $0\to Z\to Y\to X\to 0$ avec $Y$ de type fini, $Z$ est de type fini ;
\item il existe une suite exacte $Z\to Y\to X\to 0$ avec $Y$ et $Z$ de présentation finie.
\end{enumerate}
\end{enumerate}
\end{pr}

Nous utiliserons à de nombreuses reprises cette proposition dans la catégorie $\F(\A;K)$, en prenant pour $\G$ la classe des foncteurs projectifs $P_\A^a$. Nous l'utiliserons également dans $\mathbf{Add}(\A;k)$ (avec $k=K$ ou $\mathbb{Z}$) en prenant pour $\G$ la classe des foncteurs $k\otimes_\mathbb{Z}\A(a,-)$. Comme cette classe est stable par somme directe finie, en utilisant la proposition et le lemme de Yoneda, on voit qu'un foncteur de présentation finie de $\mathbf{Add}(\A;\mathbb{Z})$ s'exprime comme conoyau du morphisme $\A(y,-)\xrightarrow{\alpha^*}\A(x,-)$ induit par une certaine flèche $\alpha\in\A(x,y)$.

\begin{nota} Nous considérerons les conditions suivantes : 
\begin{itemize}
\item[$\bullet$] $\pfa(\A,K)$ : tous les foncteurs simples à valeurs de dimensions finies de la catégorie $\mathbf{Add}(\A;K)$ sont de présentation finie dans $\mathbf{Add}(\A;K)$ ;
\item[$\bullet$] $\tfc(\A,K)$ : pour tout idéal $K$-cotrivial $\I$ et tout objet $x$ de $\A$, le foncteur $\I(x,-)$ est de type fini.
\end{itemize}
S'il n'y a pas d'ambigu\"ité possible, il nous arrivera de noter simplement $\pfa$ au lieu de $\pfa(\A,K)$, et de même pour $\tfc$ (et $\lnc$).
\end{nota}

Ces notations abrègent {\bf \textsc{p}}résentation {\bf \textsc{f}}inie, {\bf \textsc{t}}ype {\bf \textsc{f}}ini, et {\bf \textsc{a}}dditifs.

\begin{ex}\label{ex-pf-aus} Soit $A$ une algèbre de dimension finie sur un corps commutatif $k$. L'hypothèse $\pfa(A\md,K)$ est vérifiée, où $A\md$ est la catégorie des modules à gauche finis sur $A$, si $K$ est le sous-corps premier de $k$. Ceci est un résultat d'Auslander \cite[corollaire~3.4 ou théorème~3.5]{Aus82}.
\end{ex}

\begin{lm}\label{pfaddf} Un foncteur additif $F : \A\to K\Md$ est de présentation finie dans $\mathbf{Add}(\A;K)$ si et seulement s'il est de présentation finie dans $\F(\A;K)$.
\end{lm}

\begin{proof} Comme le foncteur d'inclusion $\mathbf{Add}(\A;K)\to\F(\A;K)$ est pleinement fidèle et commute aux colimites, si $F$ est de présentation finie dans $\F(\A;K)$, il l'est nécessairement dans $\mathbf{Add}(\A;K)$.

Pour la réciproque, il suffit de montrer que les générateurs projectifs de type fini $\A(x,-)\otimes_\mathbb{Z}K$ de $\mathbf{Add}(\A;K)$ sont de présentation finie dans $\F(\A;K)$, en raison de la proposition~\ref{prpfclas}.\ref{itpfcns}. Cela provient de la suite exacte de $\F(\A;K)$
\[P_\A^{x\oplus x}\to P_\A^x\to\A(x,-)\otimes_\mathbb{Z}K\to 0\]
qui s'obtient à partir de la suite exacte de $K$-espaces vectoriels
\[K[V\oplus V]\xrightarrow{[u,v]\mapsto [u]+[v]-[u+v]}K[V]\to V\otimes_\mathbb{Z}K\to 0\]
naturelle en le groupe abélien $V$.
\end{proof}

\begin{lm}\label{lm-ptpf} Le produit tensoriel de deux foncteurs de présentation finie de $\F(\A;K)$ est de présentation finie.
\end{lm}

\begin{proof} Cela résulte de l'exactitude en chaque variable du produit tensoriel, de l'isomorphisme \eqref{eq-ptproj} (page~\pageref{eq-ptproj}) et de la proposition~\ref{prpfclas}.\ref{itpfcns}.
\end{proof}

\begin{lm}\label{chgtbase-pf} Soit $L$ un surcorps de $K$. Un foncteur $F$ de $\F(\A;K)$ est de type fini (resp. de présentation finie) si et seulement si le foncteur $F\otimes L$ de $\F(\A;L)$ est de type fini (resp. de présentation finie).
\end{lm}

\begin{proof} Comme le foncteur d'extension des scalaires $-\otimes_K L : \F(\A;K)\to\F(\A;L)$ est exact et envoie les générateurs projectifs $K[\A(a,-)]$ de la source sur les générateurs projectifs $L[\A(a,-)]$ du but, il préserve les objets de type fini et les objets de présentation finie.

Comme ce foncteur est exact, fidèle et commute aux colimites, si $F$ est un foncteur de $\F(\A;K)$ tel que $F\otimes L$ est de type fini dans $\F(\A;L)$, alors $F$ est de type fini. Si maintenant $F$ est tel que $F\otimes L$ est de présentation finie dans $\F(\A;L)$, alors $F$ est de type fini, donc on peut trouver une suite exacte $0\to N\to P\to F\to 0$ de $\F(\A;K)$ avec $P$ projectif de type fini. La suite exacte $0\to N\otimes L\to P\otimes L\to F\otimes L\to 0$ de $\F(\A;L)$ montre que $N\otimes L$ est de type fini (par la proposition~\ref{prpfclas}.\ref{itpfcns}), donc $N$ est de type fini dans $\F(\A;K)$ d'après ce qui précède. Il s'ensuit que $F$ est de présentation finie, par la même proposition.
\end{proof}

\begin{lm}\label{postcomp-eltr} Soient $E$ un endofoncteur élémentaire des $K$-espaces vectoriels et $\pi$ un foncteur additif de $\F(\A;K)$, de présentation finie dans $\mathbf{Add}(\A;K)$. Alors $E\circ\pi$ est un foncteur de présentation finie de $\F(\A;K)$.
\end{lm}

\begin{proof} Soit $d\in\mathbb{N}$ le degré de $E$ : ce foncteur est donc un quotient de la $d$-ième puissance tensorielle $T^d$. De plus, ce quotient provient de la catégorie des foncteurs polynomiaux {\it stricts} sur $K$ (cf. appendice~\ref{app-elt}). Comme les foncteurs polynomiaux stricts forment une catégorie abélienne localement finie, on en déduit l'existence dans $\F(K,K)$ d'une suite exacte
\[\bigoplus_{i=1}^n\Gamma^d(-\otimes V_i)\to T^d\to E\to 0\]
où les $V_i$ sont des $K$-espaces vectoriels de dimensions finies, et $\Gamma^d$ désigne la $d$-ième puissance divisée (dans les $K$-espaces vectoriels).

On en déduit une suite exacte
\[\bigoplus_{i=1}^n\Gamma^d(\pi\otimes V_i)\to\pi^{\otimes d}\to E\circ\pi\to 0\]
dans $\F(\A;K)$. Les lemmes~\ref{pfaddf} et~\ref{lm-ptpf} montrent que $\pi^{\otimes d}$ est de présentation finie, il suffit donc (cf.  proposition~\ref{prpfclas}.\ref{itpfcns}) de montrer que les foncteurs $\Gamma^d(\pi\otimes V_i)$ sont de type fini.

Comme $\pi$ est de type fini dans $\mathbf{Add}(\A;K)$, c'est un quotient d'un foncteur du type $\A(x,-)\otimes_\mathbb{Z}K$, de sorte que $\Gamma^d(\pi\otimes V_i)$ est un quotient de $\Gamma^d(\A(t,-)\otimes_\mathbb{Z}K)$ (où $t$ est une somme directe de copies de $x$ en nombre égal à la dimension de $V_i$ sur $K$). Un tel foncteur est de type fini : si le corps $K$ est infini, c'est un quotient de $P^t_\A$, et si $K$ est fini, cela provient de ce que le foncteur $\Gamma^d(-\otimes_\mathbb{Z}K)$ de $\F(K,K)$ est fini.
Le lemme en découle. 
\end{proof}

\begin{thm}\label{th-pol-pf} La propriété $\pfa$ est vérifiée si et seulement si tous les foncteurs polynomiaux de longueur finie et à valeurs de dimensions finies de $\F(\A;K)$ sont de présentation finie.
\end{thm}

\begin{proof} Supposons $\pfa$ vérifiée. Il suffit de vérifier que tout foncteur simple et polynomial de $\F^{\df}(\A;K)$ est de présentation finie, par la proposition~\ref{prpfclas}.\ref{itpfext}. Pour un tel foncteur admettant une décomposition à la Steinberg, la conclusion découle des lemmes~\ref{postcomp-eltr} et~\ref{lm-ptpf}. Le cas général s'y ramène par le théorème~\ref{st-pol-petit-corps} et le lemme~\ref{chgtbase-pf}.

La réciproque découle du lemme~\ref{pfaddf}.
\end{proof}

En particulier (cf. exemple~\ref{ex-pf-aus}) :

\begin{cor} Soit $A$ une algèbre de dimension finie sur un corps commutatif $k$. Supposons que $K$ est le sous-corps premier de $k$. Alors tous les foncteurs polynomiaux simples à valeurs de dimensions finies de $\F(A\md;K)$ sont de présentation finie.
\end{cor}

Nous abordons maintenant le problème de la présentation finie pour les foncteurs simples antipolynomiaux.

\begin{lm}\label{pf-linearis}
Un foncteur additif $\pi : \A\to\mathbf{Ab}$ est de type fini (resp. de présentation finie) dans la catégorie $\mathbf{Add}(\A;\mathbb{Z})$ si et seulement si $K[\pi]$ est de type fini (resp. de présentation finie) dans $\F(\A;K)$.
\end{lm}

\begin{proof} Si $\pi$ est de type fini, c'est un quotient d'un foncteur représentable $\A(x,-)$, donc $K[\pi]$ est un quotient de $P^x_\A$ et est par conséquent de type fini. Réciproquement, si $K[\pi]$ est de type fini, comme la post-composition par $K[-]$ commute aux colimites filtrantes et induit une fonction strictement croissante entre ensembles ordonnés de sous-foncteurs, $\pi$ est de type fini.

Si $\pi$ est de présentation finie, il existe une flèche $\alpha\in\A(x,y)$ et une suite exacte $\A(y,-)\xrightarrow{\alpha^*}\A(x,-)\to\pi\to 0$ dans $\mathbf{Add}(\A;\mathbb{Z})$, par la proposition~\ref{prpfclas}.\ref{itpfcns}. On en déduit, en utilisant \eqref{eqsers} (appendice~\ref{app-linearis}), une  suite exacte
\[K[\A(x\oplus y,-)]\to K[\A(x,-)]\to K[\pi]\to 0\]
de $\F(\A;K)$ qui constitue une présentation finie de $K[\pi]$.

Réciproquement, supposons $K[\pi]$ de présentation finie. Comme le foncteur $K[-] : \mathbf{Ab}\to K\Md$ commute aux colimites filtrantes, il en est de même pour le foncteur ${\rm Hom}(K[\pi],-)\circ K[-] : \mathbf{Add}(\A;\mathbb{Z})\to\F(\A;K)$. Mais $K[\pi]$, donc $\pi$, est de type fini, d'après ce qu'on a vu plus haut. Par conséquent, le lemme~\ref{lm-linearis} montre que ${\rm Hom}(K[\pi],-)\circ K[-]\simeq K[-]\circ {\rm Hom}(\pi,-)$. Comme le foncteur $K[-]$ commute aux colimites filtrantes et reflète les isomorphismes, la commutation de cette composée aux colimites filtrantes entraîne la même propriété pour $\mathrm{Hom}(\pi,-)$, c'est-à-dire que $\pi$ est de présentation finie.
\end{proof}

\begin{pr}\label{pr-pf-postcomp}
Soit $\Phi : \A\to\B$ un foncteur additif entre petites catégories additives. Supposons que, pour tout objet $b$ de $\B$, le foncteur $\Phi^*\B(b,-)$ est de présentation finie dans la catégorie $\mathbf{Add}(\A;\mathbb{Z})$. Alors le foncteur $\Phi^* : \F(\B;K)\to\F(\A;K)$ préserve les objets de présentation finie.
\end{pr}

\begin{proof} Le lemme~\ref{pf-linearis} montre que $\Phi^*$ envoie un projectif $P_\B^b$ sur un objet de présentation finie de $\F(\A;K)$. Comme $\Phi^*$ est exact, la conclusion s'ensuit, par la proposition~\ref{prpfclas}.\ref{itpfcns}.
\end{proof}

\begin{cor}\label{corpfcotr} Si la condition $\tfc$ est vérifiée, alors, pour tout idéal $K$-cotrivial $\I$ de $\A$, le foncteur de précomposition $\pi_I^* : \F(\A/\I;K)\to\F(\A;K)$ préserve les objets de présentation finie.
\end{cor}

\begin{cor}\label{cotrivpf} Si les conditions $\lnc$ et $\tfc$ sont vérifiées, alors tout foncteur $F$ de type fini et antipolynomial de $\F(\A;K)$ est de présentation finie.
\end{cor}

\begin{proof} Soit $\I$ un idéal $K$-cotrivial de $\A$ tel que $F\simeq\pi_\I^*(X)$ pour un certain foncteur $X$ de $\F(\A/\I;K)$. Comme $F$ est de type fini, $X$ est de type fini. L'hypothèse $\lnc$ montre alors que $X$ est de présentation finie (dans la catégorie localement noethérienne $\F(\A/\I;K)$). La conclusion provient donc du corollaire~\ref{corpfcotr}.
\end{proof}

Nous pouvons maintenant donner le résultat principal de cette section.

\begin{thm}\label{th-simples-pf}
 Si les propriétés $\lnc$, $\pfa$ et $\tfc$ sont vérifiées, alors tous les foncteurs finis et à valeurs de dimensions finies de $\F(\A;K)$ sont de présentation finie.
\end{thm}

\begin{proof} Si $K$ est assez gros (par exemple, contient les racines de l'unité), le résultat se déduit du théorème~\ref{th-pol-pf}, du corollaire~\ref{cotrivpf}, du lemme~\ref{lm-ptpf} et du corollaire~\ref{cor-tens-St1}. Le cas général s'y ramène grâce au lemme~\ref{chgtbase-pf}.
\end{proof}

Dans le cas d'une catégorie source $\mathbf{P}(A)$, on peut donner une condition {\em nécessaire et suffisante} à la présentation finie des foncteurs finis à valeurs de dimensions finies. Pour cela, nous aurons besoin d'une propriété valable dans un cadre plus général.

\begin{pr}\label{pr-recip-pfadd} Soit $\I\triangleleft\A$ un idéal $K$-cotrivial. Supposons que la catégorie $\F(\A/\I;K)$ est localement finie. Si tous les foncteurs $\pi_\I^*(S)$ sont de présentation finie dans $\F(\A;K)$, où $S$ est un foncteur simple de $\F(\A/\I;K)$, alors le foncteur $\I(x,-) : \A\to\mathbf{Ab}$ est de type fini pour tout objet $x$ de $\A$.
\end{pr}

\begin{proof} Comme la classe des foncteurs de présentation finie est stable par extensions (proposition~\ref{prpfclas}.\ref{itpfext}), l'hypothèse implique que $\pi_\I^*(F)$ est un foncteur de présentation finie de $\F(\A;K)$ pour tout foncteur fini $F$ de $\F(\A/\I;K)$. Comme cette dernière catégorie est supposée localement finie, pour tout $x\in {\rm Ob}\,\A$, le foncteur de type fini $P^x_{\A/\I}$ est fini, de sorte que $\pi_\I^*(P^x_{\A/\I})=K[\A(x,-)/\I(x,-)]$ est de présentation finie dans $\F(\A;K)$. Le lemme~\ref{pf-linearis} montre alors que $\A(x,-)/\I(x,-)$ est de présentation finie dans $\mathbf{Add}(\A;\mathbb{Z})$, c'est-à-dire que $\I(x,-)$ est de type fini.
\end{proof}

\begin{cor}\label{cor-recip-pfadd} Soit $A$ un anneau. Si tous les foncteurs simples antipolynomiaux de $\F(A,K)$ sont de présentation finie, alors la propriété $\tfc(\mathbf{P}(A);K)$ est vérifiée.
\end{cor}

\begin{proof} Combiner le corollaire~\ref{corflf} à la proposition~\ref{pr-recip-pfadd}.
\end{proof}

Pour une catégorie source $\mathbf{P}(A)$, nous pouvons améliorer le théorème~\ref{th-simples-pf} en donnant une condition nécessaire et suffisante à la présentation finie des simples à valeurs de dimensions finies.

La notion d'idéal cotrivial d'un anneau qui apparaît dans l'énoncé ci-dessous est celle de l'exemple~\ref{rem-ideal} page~\pageref{rem-ideal}.

\begin{thm}\label{thapf} Soit $A$ un anneau. Les foncteurs finis et à valeurs de dimensions finies de $\F(A,K)$ sont de présentation finie si et seulement si les deux conditions suivantes sont satisfaites :
\begin{enumerate}
\item tout $A^{\op}\otimes_\mathbb{Z}K$-module simple, de dimension finie sur $K$, est de présentation finie ;
\item tout idéal $K$-cotrivial de $A$ est de type fini comme idéal à droite.
\end{enumerate}
\end{thm}

\begin{proof} Cela découle des théorèmes~\ref{th-simples-pf} et \ref{th-pol-pf} ainsi que des corollaires~\ref{corlnc} et~\ref{cor-recip-pfadd}.
\end{proof}

\begin{cor}\label{cor-pfdf} Si les anneaux $A$ et $A\otimes_\mathbb{Z}K$ sont noethériens à droite, alors tout foncteur fini de $\F^{\df}(A,K)$ est de présentation finie dans $\F(A,K)$.
 \end{cor}

\begin{ex} L'hypothèse du corollaire précédent est vérifiée dans chacun des deux cas suivants :
\begin{enumerate}
\item l'anneau $A$ est noethérien à droite et $K$ est une extension de type fini de son sous-corps premier ; 
\item $A$ est un anneau commutatif de type fini.
\end{enumerate}
\end{ex}

\appendix

\section{Foncteurs simples ne prenant pas des valeurs de dimensions finies}\label{apa}

\subsection{Foncteurs simples prenant une valeur non nulle de dimension finie}\label{apa-1}

On rappelle (corollaire~\ref{simples-eval}) que, si $F$ est un foncteur simple de $\F(\A;K)$ et $t$ est un objet de $\A$ tel que $F(t)$ soit non nul, alors $F$ est isomorphe au prolongement intermédiaire associé au foncteur d'évaluation en $t$ du $K[{\rm End}(t)]$-module $F(t)$. Cela permet théoriquement de disposer d'un contrôle sur les foncteurs simples de $\F(\A;K)$ qui prennent \emph{une} valeur non nulle de dimension finie, à partir des représentations simples de dimension finie sur $K$ des monoïdes multiplicatifs d'endomorphismes des objets de $\A$.

La propriété suivante, qui constitue un corollaire des résultats du §\,\ref{phdfx}, montre que, dans de nombreux cas, les foncteurs simples avec \emph{une} valeur non nulle de dimension finie ont \emph{toutes} leurs valeurs de dimensions finies si et seulement s'ils sont polynomiaux.

\begin{pr}\label{cor-rep-pol-cotr} Soient $A$ un anneau commutatif, $n\in\mathbb{N}$ et $M$ une représentation simple de dimension finie sur $K$ du monoïde $\M_n(A)$. On suppose que $A$ ne possède pas d'autre idéal $K$-cotrivial que $A$ (par exemple, que $A$ est un corps infini). Notons $S_M$ le foncteur simple de $\F(A,K)$ obtenu en appliquant à $M$ le prolongement intermédiaire associé au foncteur réflexif $\F(A,K)\to K[\M_n(A)]\Md$ d'évaluation en $A^n$ (cf. propositions~\ref{pr-prolongement-interm} et~\ref{restr-fct-refl}).

Les assertions suivantes sont équivalentes.
\begin{enumerate}
\item[{\rm (a)}] Le foncteur $S_M$ est polynomial.
\item[{\rm (b)}] Le foncteur $S_M$ est à valeurs de dimensions finies.
\item[{\rm (c)}] La représentation $M$ de $\M_n(A)$ est polynomiale à la Eilenberg-MacLane.
\end{enumerate}
\end{pr}

\begin{proof} La proposition~\ref{pr-prolongement-interm} montre que $S_M(A^n)\simeq M$ comme $K[\M_n(A)]$-modules. Il en résulte en particulier que (a) entraîne (c). La proposition~\ref{ann-com-cr} permet également d'en déduire que (a) implique (b).

Le lemme~\ref{lm-sucosu} montre que (c) entraîne (a) (en effet, $A^n$ constitue un support et un co-support de $S_M$). L'implication (b)$\Rightarrow$(a) provient quant à elle du corollaire~\ref{cor-pasdecotr}.
\end{proof}

\noindent
\textbf{Convention : dans toute la suite de cet appendice, $k$ désigne un corps commutatif infini.}

Nous nous concentrerons sur la catégorie $\F(k,K)$, où adviennent déjà de nombreux phénomènes intéressants.

 Pour $n>0$, il est facile de construire des représentations simples de $\M_n(k)$ de dimension finie sur $K$ qui ne sont pas polynomiales. Pour $n=1$, tout morphisme de groupes $k^\times\to K^\times$ dont le prolongement par zéro fournit une fonction non polynomiale $k\to K$ en donne un cas.

\begin{ex}\label{ex-q-dim-infinie} Prenons $k=K=\mathbb{Q}$. Les seuls foncteurs simples non constants de $\F^{\df}(\mathbb{Q},\mathbb{Q})$ non nuls sur $\mathbb{Q}$ sont les puissances symétriques $S^d$ pour $d\in\mathbb{N}^*$. Ils sont associés aux endomorphismes $x\mapsto x^d$ du groupe $\mathbb{Q}^\times$. Mais ce groupe possède de nombreux autres endomorphismes, comme $x\mapsto x^d$ pour $d\leq 0$ -- et encore bien d'autres puisque ce groupe abélien, isomorphe à $\mathbb{Z}/2\oplus\mathbb{Z}^{\oplus\mathbb{N}}$, a un anneau d'endomorphismes qui a la puissance du continu !
\end{ex}

En fait, on peut souvent décrire de façon assez explicite les foncteurs simples obtenus par ce procédé, à l'aide de la construction générale suivante, qui a été introduite et étudiée (dans le cas où $k=K$ est un corps fini) par Powell \cite{GP-cow}. Soient $n$ un entier naturel et $M$ une représentation $K$-linéaire de $\GL_n(k)$. On définit un foncteur $Q_M$ de $\F(k,K)$ par
\begin{equation*}
Q_M:=K[\mathrm{Inj}_k(k^n,-)]\underset{K[\GL_n(k)]}{\otimes}M
\end{equation*}
où $\mathrm{Inj}_k(k^n,V)$ désigne l'ensemble des injections $k$-linéaires $k^n\to V$ (on fait de $K[\mathrm{Inj}_k(k^n,-)]$ un foncteur de $\F(k,K)$ quotient de $P_{\mathbf{P}(k)}^{k^n}$ en tuant, dans $P_{\mathbf{P}(k)}^{k^n}(V)=K[{\rm Hom}_k(k^n,V)]$, les morphismes non injectifs $k^n\to V$).

Par ailleurs, si $M$ est une représentation de dimension finie $m$ sur $K$ du groupe $\GL_n(k)$, nous noterons encore $M$, par abus, la représentation du monoïde $\M_n(k)$ obtenue en prolongeant l'action de $\GL_n(k)$ par $0$ aux matrices singulières de $\M_n(k)$, et désignerons par $\phi : \M_n(k)\to\M_m(K)$ la fonction multiplicative correspondante obtenue en choisissant une $K$-base de $M$. Comme dans la proposition~\ref{cor-rep-pol-cotr}, $S_M$ désigne le foncteur de $\F(k,K)$ associé à $M$ par prolongement intermédiaire ; c'est naturellement un quotient de $Q_M$, et $S_M$ est simple si $M$ est une représentation simple de $\GL_n(k)$.

\begin{pr}\label{pr-simples-biz1} Soient $n, m\in\mathbb{N}^*$, $M$ une représentation simple de dimension finie $m$ sur $K$ de $\GL_n(k)$, $\phi : \M_n(k)\to\M_m(K)$ le morphisme associé comme ci-dessus et $\phi_1,\dots,\phi_m : \M_n(k)\to K^m$ les colonnes de $\phi$. Le foncteur $Q_M$ est simple (et donc isomorphe à $S_M$) si et seulement si la condition suivante est vérifiée : la famille des translatées $(\tau_a(\phi_i))_{1\leq i\leq m,a\in\M_n(k)}$ (où $\tau_a(\psi)(x):=\psi(a+x)$) est libre dans le $K$-espace vectoriel $K^{\M_n(k)}$.
\end{pr}

(Sous les hypothèses de la proposition, l'espace vectoriel $S_M(k^i)$ est donc de dimension finie non nulle si et seulement si $i=n$.)

\begin{proof} Explicitement, $S_M$ est l'image du morphisme $Q_M\to M^{\mathrm{Hom}_k(V,k^n)}$ donné sur l'espace vectoriel $V$ par
\[K[{\rm Inj}_k(k^n,V)]\underset{K[\GL_n(k)]}{\otimes} M\to M^{\mathrm{Hom}_k(V,k^n)}\qquad ([f]\otimes\xi)\mapsto (g\mapsto \phi(g.f).\xi).\]

Comme $\GL_n(k)$ opère librement sur $\mathrm{Inj}_k(k^n,V)$ et que les orbites sont classifiées par les sous-espaces de dimension $n$ de $V$ via l'image d'une telle injection, les éléments non nuls de $Q_M(V)$ sont exactement les images par la projection canonique des éléments $\sum_{i=1}^r [f_i]\otimes\xi_i$ de $K[\mathrm{Inj}_k(k^n,V)]\otimes M$, où $r\in\mathbb{N}^*$, les $\xi_i$ sont des éléments non nuls de $M$, les $E_i:=\mathrm{Im}\,f_i$ sont des sous-espaces vectoriels deux à deux distincs de $V$ et de dimension $n$. Comme $k$ est infini, et qu'une réunion finie de sous-espaces stricts d'un $k$-espace vectoriel en est donc toujours un sous-ensemble strict, il existe une application linéaire $V\to k^n$ telle que les composées $E_i\hookrightarrow V\to k^n$ soient des isomorphismes pour tout $i\in\{1,\dots,r\}$. Par conséquent, quitte à modifier les $f_i$ (et les $\xi_i$) par l'action d'éléments de $\GL_n(k)$, on peut supposer que $V$ est de la forme $k^n\oplus E$ et que la première composante $k^n\to k^n$ des $f_i$ est l'identité ; nous noterons $g_i : k^n\to E$ leur deuxième composante. Les $g_i$ sont donc deux à deux distincts. L'annulation de $\sum_{i=1}^r [f_i]\otimes\xi_i$ dans $M^{{\rm Hom}_k(V,k^n)}$ signifie que, pour toute application linéaire $g : V=k^n\oplus E\to k^n$, dont nous noterons $u : k^n\to k^n$ et $v : E\to k^n$ les composantes, on a
\[\sum_{i=1}^r\phi(u+vg_i).\xi_i=0\]
dans $M$. Utilisant encore qu'une réunion finie de sous-espaces stricts d'un $k$-espace vectoriel en est un sous-ensemble strict, on voit que le fait que les $g_i$ soient deux à deux distincts entraîne qu'il existe $v\in\mathrm{Hom}_k(E,k^n)$ tel que les $vg_i$ soient deux à deux distincts.

Autrement dit, la projection $Q_M\twoheadrightarrow S_M$ est un isomorphisme si et seulement s'il existe $r\in\mathbb{N}^*$, des éléments deux à deux distincts $a_1,\dots,a_r$ de $\M_n(k)$ et des éléments non nuls $\xi_1,\dots,\xi_r$ de $M\simeq K^m$ tels que
\[\sum_{i=1}^r\tau_{a_i}(\phi).\xi_i=0\]
ce qui termine la démonstration.
\end{proof}

\begin{ex}\label{ex-grasm-simples} La représentation triviale de $\GL_n(k)$ dans $K$ correspond à la fonction $\phi : \M_n(k)\to K$ qui vaut $1$ sur $\GL_n(k)$ et $0$ sur les matrices singulières. Celle-ci vérifie la condition d'indépendance linéaire des translatées requise dans l'énoncé précédent. Supposons en effet qu'on dispose d'une relation $\sum_{i=1}^r\lambda_i\tau_{a_i}(\phi)=0$ avec les $a_i\in\M_n(k)$ deux à deux distincts. Comme $k$ est infini, on peut trouver $x\in\M_n(k)$ telle que toutes les matrices $a_i+x$ soient inversibles, d'où $\sum_{i=1}^r\lambda_i=0$. On peut également trouver, pour tout indice $j$, $x\in\M_n(k)$ telle que $a_i+x$ soit inversible si et seulement $i\neq j$, ce qui fournit la relation $\sum_{i\neq j}\lambda_i=0$ et donc $\lambda_j=0$ pour tout $j$. Le foncteur simple qu'on en déduit est donné sur les objets par $V\mapsto K[\G r_n(V)]$, où $\G r_n(V)$ désigne l'ensemble des sous-espaces de dimension $n$ de $V$ ; ce foncteur envoie une flèche $f : V\to W$ sur l'application linéaire envoyant $[E]$, pour $E\in\G r_n(V)$, sur $[f(E)]$ si $f(E)$ est de dimension $n$, et sur $0$ sinon.

Si $k$ est un corps \emph{fini}, le foncteur $K[\G r_n]$ est encore simple \emph{si $K$ est de caractéristique différente de celle de $k$}, mais la démonstration précédente ne s'applique plus. La simplicité se déduit de résultats de Kuhn \cite{Ku-adv}. En revanche, si $k$ est fini et de même caractéristique que $K$, alors $K[\G r_n]$ n'est \emph{jamais simple} si $n>0$ -- on peut même montrer qu'il est de dimension de Krull au moins $n$ (et conjecturalement, de dimension de Krull exactement $n$ -- cf. \cite{GP-cow,Dja-gr}).
\end{ex}

\begin{ex} Prenons $K=k$ et considérons l'endomorphisme $x\mapsto x^d$ de $k^\times$, où $d$ est un entier strictement négatif. Son prolongement par $0$ fournit un morphisme multiplicatif $f : k\to k$ qui vérifie la condition d'indépendance linéaire de la proposition~\ref{pr-simples-biz1} par un argument de fractions rationnelles immédiat (on rappelle que $k$ est un corps commutatif \emph{infini}).

Pour $k=K=\mathbb{R}$, des arguments d'analyse élémentaires montrent que le prolongement par $0$ à $\mathbb{R}$ de tout endomorphisme \emph{continu} $f$ du groupe topologique $\mathbb{R}^\times$ vérifie également cette condition, hormis si ce prolongement est polynomial, c'est-à-dire si $f$ est de la forme $x\mapsto x^d$ avec $d\in\mathbb{N}^*$.
\end{ex}

\begin{rem} Nous ignorons s'il existe un morphisme de monoïdes $\M_n(k)\to\M_m(K)$ fournissant une représentation simple \emph{non EML-polynomiale} de $\M_n(k)$ et ne vérifiant pas la condition de la proposition~\ref{pr-simples-biz1}, même en nous restreignant à $n=m=1$.
\end{rem}

\begin{rem}\label{rem-comparaison-SL-GL} Le cas $n=m=1$ est sans doute le plus important dans les constructions précédentes, en vertu des résultats de Steinberg \cite{RSt} et de Bass-Milnor-Serre \cite{BMS} (déjà mentionnés dans la section~\ref{section-generique}) dont notre corollaire~\ref{cor-tens-St1} s'inspire. Ainsi, pour $k=\mathbb{Q}$, toute représentation complexe de dimension finie de $\SL_n(\mathbb{Q})$ est polynomiale, où $n\in\mathbb{N}$ est arbitraire, ce qui entraîne que toute représentation simple de dimension finie de $\GL_n(\mathbb{Q})$ s'exprime comme le produit tensoriel entre une représentation polynomiale et une représentation se factorisant par le déterminant, donc provenant d'un morphisme de groupes $\mathbb{Q}^\times\to\mathbb{C}^\times$ : ce sont ces morphismes qui font apparaître une <<~partie non polynomiale~>>, donnant lieu à des foncteurs dont les valeurs ne sont pas de dimensions finies.
\end{rem}

\subsection{Foncteurs simples qui ne prennent aucune valeur de dimension finie non nulle}\label{apa-2}

Il paraît très difficile de donner des résultats généraux sur les foncteurs simples dont aucune valeur non nulle n'est de dimension finie. Il est toutefois facile de voir qu'il en existe souvent beaucoup, dès lors que les ensembles de morphismes dans la catégorie source ne sont pas trop petits. L'énoncé suivant en donne une illustration.

\begin{pr}\label{simples-sans-val-df} Soient $A$ un anneau commutatif ou noethérien de caractéristique nulle et $M$ un $A$-module projectif de type fini dont $A^2$ est facteur direct. Il existe un foncteur simple $S$ de  $\F(A,K)$ tel que, pour tout objet $V$ de $\mathbf{P}(A)$ :
\begin{enumerate}
\item $\dim_K S(V)=\infty$ si $M$ est facteur direct de $V$ ;
\item $S(V)=0$ sinon.
\end{enumerate} 
\end{pr}

\begin{proof} Comme $A^2$ est facteur direct de $M$, le groupe $G:={\rm Aut}_{\mathbf{P}(A)}(M)$ contient $\GL_2(A)$, qui contient $\GL_2(\mathbb{Z})$ (car $A$ est de caractéristique $0$), et donc un groupe libre non abélien $H$.  D'après Formanek \cite[théorème~5]{Form}, il existe un $K[H]$-module simple et fidèle, donc de dimension infinie sur $K$, ce qui entraîne l'existence d'un $K[G]$-module $W$ simple et de dimension infinie. Comme $A$ est commutatif ou noethérien, on peut prolonger l'action de $G$ par $0$ aux endomorphismes non inversibles de $M$ (obtenant une représentation simple de dimension infinie de ce monoïde) et obtenir par prolongement intermédiaire un foncteur simple $S$ de $\F(A,K)$ tel que $S(M)\simeq W$ comme $K[{\rm End}_{\mathbf{P}(A)}(M)]$-modules. Si $M$ est facteur direct de $V$, alors $S(M)$ est facteur direct de $S(V)$, qui est donc de dimension infinie. Si $M$ n'est pas facteur direct de $V$, le fait que les endomorphismes non inversibles de $M$ opèrent par $0$ sur $W$ et l'expression du prolongement intermédiaire entraînent que $S(M)$ est nul.
\end{proof}

\begin{rem}\label{rqdfa} La conclusion de la proposition vaut également dans d'autres cas, par exemple si $A$ est la clôture algébrique d'un corps fini et que $K$ est de caractéristique nulle, par Snider \cite{Sni83}.
\end{rem}

\subsection{Produits infinis de foncteurs simples}\label{apa-4}
 
Les produits infinis de copies d'un foncteur simple à valeurs de dimensions finies sont isomorphes à une somme directe de copies de ce même foncteur, par {\it endofinitude} (cf. par exemple Krause \cite[Th.~6.14]{Kra}, qui traite de catégories de modules, mais dont les arguments s'appliquent pareillement à notre contexte). En revanche, un produit de copies d'un foncteur simple $S$ qui n'est pas à valeurs de dimensions finies peut avoir de nombreux sous-quotients simples non isomorphes à $S$. En voici une illustration élémentaire dans une catégorie de foncteurs dont les simples à valeurs de dimensions finies sont décrits par nos résultats.
 
 \begin{ex}\label{ex-simple-gros-prod}
  Soit $n\geq 2$ un entier. Toute droite $D\in\mathbb{P}^{n-1}(k)$ du $k$-espace vectoriel $k^n$ (on rappelle que $k$ est un corps commutatif infini) fournit un morphisme $K[{\rm Hom}_k(k^n,-)]\to K[{\rm Hom}_k(D,-)]$ dans $\F(k,K)$, qui induit un morphisme $K[({\rm Hom}_k(k^n,-)\setminus\{0\})/k^\times]\to K[\G r_1]$ (selon la notation de l'exemple~\ref{ex-grasm-simples}); de plus, le morphisme
\[K[({\rm Hom}_k(k^n,-)\setminus\{0\})/k^\times]\to K[\G r_1]^{\mathbb{P}^{n-1}(k)}\]
  qu'on en déduit est \emph{injectif}. En effet, un élément non nul de $K[({\rm Hom}_k(k^n,V)\setminus\{0\})/k^\times]$ (où $V$ est un $k$-espace vectoriel de dimension finie sur $k$) s'écrit $x=\sum_{i=1}^r\lambda_i[f_i]$, où $r\in\mathbb{N}^*$, $\lambda_i\in K^\times$ et les $f_i : k^n\to V$ sont des applications linéaires deux à deux non proportionnelles. Par conséquent, pour $i\neq j$, l'ensemble $\{D\in\mathbb{P}^{n-1}(k)\,|\,f_i(D)\neq f_j(D)\}$ est un \emph{ouvert dense} de $\mathbb{P}^{n-1}(k)$ pour la topologie de Zariski, de même que 
 $\{D\in\mathbb{P}^{n-1}(k)\,|\,f_i(D)\neq 0\}$. Il s'ensuit qu'on peut trouver une droite $D$ de $k^n$ telle que les $f_i(D)$ soient des droites deux à deux distinctes de $V$, ce qui montre que l'image de $x$ par notre application est non nulle.
  
   Le foncteur $K[({\rm Hom}_k(k^n,-)\setminus\{0\})/k^\times]$ se surjecte sur $K[\G r_n]$, qui est donc un sous-quotient d'un produit de copies de $K[\G r_1]$. Or on a vu dans l'exemple~\ref{ex-grasm-simples} que tous les foncteurs $K[\G r_n]$ sont simples, et ils sont manifestement deux à deux non isomorphes.
 \end{ex}
 
 D'autres exemples contre-intuitifs par rapport à la situation où l'on ne se considère que des foncteurs à valeurs de dimensions finies proviennent des produits infinis de foncteurs simples \emph{deux à deux non isomorphes}. Remarquons déjà que, si $(S_i)_{i\in E}$ une famille infinie d'objets simples deux à deux non isomorphes de la catégorie de Grothendieck $\E$, l'objet $\underset{i\in E}{\prod}S_i$ de $\E$ est semi-simple si et seulement s'il coïncide avec son sous-objet $\underset{i\in E}{\bigoplus}S_i$ (qui constitue toujours son socle). Pour $\E=\F(\C;K)$, où $\C$ est une petite catégorie, si l'on suppose que les $S_i$ sont à valeurs de dimensions finies, cette condition équivaut à dire que  $\underset{i\in E}{\bigoplus}S_i$ (ou $\underset{i\in E}{\prod}S_i$) est à valeurs de dimensions finies.
 
 Lorsque cette condition n'est pas satisfaite, l'objet $\underset{i\in E}{\prod}S_i$ peut avoir de nombreux sous-quotients simples qui ne sont isomorphes à aucun des $S_i$. En voici une illustration.
 
 \begin{ex}\label{ex-inf-qq}
  Dans $\F(\mathbb{Q},\mathbb{Q})$, le foncteur projectif $P^\mathbb{Q}=\mathbb{Q}[-]$ a pour plus grand quotient polynomial de degré au plus $n$ la somme directe sur les entiers $0\leq i\leq n$ des foncteurs $S^i$ ($i$-ème puissance symétrique sur $\mathbb{Q}$). Chaque foncteur $S^i$ est simple. De plus, le morphisme $P^\mathbb{Q}\to\underset{i\in\mathbb{N}}{\prod}S^i$ qu'on en déduit est injectif, car l'idéal d'augmentation de l'algèbre sur un corps de caractéristique nulle d'un groupe abélien sans torsion est résiduellement nilpotent \cite[Chap.~VI, th.~2.26]{Passi}. Mais $P^\mathbb{Q}$ a plein de quotients simples non polynomiaux : on dispose d'une injection de l'ensemble (infini non dénombrable) des endomorphismes de $\mathbb{Q}^\times$ dont l'extension par $0$ à $\mathbb{Q}$ n'est pas une fonction polynomiale dans l'ensemble de ces quotients -- cf. exemple~\ref{ex-q-dim-infinie}.
 \end{ex}
 
 \begin{rem}\label{rq-cat-biloc} 
  La catégorie $\F(\mathbb{Q},\mathbb{Q})$ est engendrée par les foncteurs projectifs $P^{\mathbb{Q}^n}\simeq (P^\mathbb{Q})^{\otimes n}$ pour $n\in\mathbb{N}$, qui s'injectent grâce à l'exemple précédent dans un produit dénombrable de foncteurs polynomiaux finis et à valeurs de dimensions finies. On en déduit en particulier que la plus petite sous-catégorie \emph{bilocalisante} (c'est-à-dire épaisse et stable par limites) de $\F(\mathbb{Q},\mathbb{Q})$ engendrée par $\F^{\df}(\mathbb{Q},\mathbb{Q})$, ou par les foncteurs polynomiaux, est $\F(\mathbb{Q},\mathbb{Q})$. A contrario, la sous-catégorie localisante de $\F(\mathbb{Q},\mathbb{Q})$ engendrée par $\F^{\df}(\mathbb{Q},\mathbb{Q})$, ou par les foncteurs polynomiaux, est exactement la catégorie des foncteurs \emph{analytiques} (c'est-à-dire des colimites de foncteurs polynomiaux) de $\F(\mathbb{Q},\mathbb{Q})$, qui est classiquement équivalente à $\underset{n\in\mathbb{N}}{\prod}\mathbb{Q}[\Si_n]\Md$.
  
  Dans d'autres catégories de foncteurs, comme $\F(k,K)$, avec $\mathrm{car}(k)\neq\mathrm{car}(K)$, le corollaire~\ref{cor-dfcst} montre que la sous-catégorie bilocalisante de $\F(k,K)$ engendrée par $\F^{\df}(k,K)$ est réduite, comme la sous-catégorie bilocalisante engendrée par les foncteurs polynomiaux, aux foncteurs constants.
 \end{rem}

\section{Foncteurs élémentaires}\label{app-elt}

Cet appendice regroupe les principales propriétés des foncteurs élémentaires. Certaines sont établies dans la section \ref{sect-stein-equi}, nous les complétons avec des énoncés basiques provenant de la théorie des représentations des groupes algébriques et des foncteurs associés (les foncteurs \emph{strictement polynomiaux} \cite{FS}).

Soit $K$ un corps commutatif quelconque. On rappelle (cf. définition~\ref{df-felem}) qu'un endofoncteur $E$ des $K$-espaces vectoriels est dit \emph{élémentaire} s'il existe un entier $d\ge 0$ et un $K[\Si_d]$-module simple $M$ tel que $E$ soit isomorphe à l'image de la transformation naturelle $N$ donnée par la norme :
$$ (T^d\otimes M)_{\Si_d}\xrightarrow[]{N} (T^d\otimes M)^{\Si_d}\;,$$
où un élément $\sigma\in\Si_d$ agit via la transformation naturelle donnée par les endomorphismes $K$-linéaires de $T^d(V)\otimes M=V^{\otimes d}\otimes M$ qui envoient les éléments $v_1\otimes\dots\otimes v_d\otimes m$ sur les éléments $v_{\sigma^{-1}(1)}\otimes\dots\otimes v_{\sigma^{-1}(d)}\otimes \sigma m$. Dans la suite de cet appendice, il nous sera commode d'identifier les foncteurs élémentaires à leur restriction à la catégorie $\mathbf{P}(K)$ des $K$-espaces vectoriels {\em de dimensions finies} -- c'est-à-dire à les considérer comme des objets de $\F(K,K)$. Cet abus est anodin car les foncteurs élémentaires commutent aux colimites filtrantes. Dans ce cadre, les foncteurs élémentaires peuvent se voir comme l'image des représentations simples des groupes symétriques par des prolongements intermédiaires associés aux effets croisés (cf. lemme~\ref{lm-ident-printer}). 

Les faits suivants sont formels (cf. section~\ref{sect-stein-equi}).
\begin{enumerate}
\item[1.] Les foncteurs élémentaires sont absolument simples.
\item[2.] Pour tout $d\ge 0$, les classes d'isomorphisme de foncteurs élémentaires de degré $d$ sont en bijection avec les classes d'isomorphisme de $K[\Si_d]$-modules simples.
\item[3.] Les foncteurs élémentaires sont les quotients simples des $T^d$, $d\ge 0$.
\end{enumerate}

Nous nous attachons dans la suite à des considérations plus explicites sur les foncteurs élémentaires, en en donnant d'abord une description plus précise.

\smallskip

Si $K$ est de caractéristique nulle, la norme est un isomorphisme, si bien que les foncteurs élémentaires de degré $d$ sont isomorphes aux $(T^d\otimes M)_{\Si_d}$ pour les $K[\Si_d]$-modules simples $M$. Ces derniers  peuvent se décrire en termes de symétriseurs de Young \cite[Chap.~4.2]{FH} et les foncteurs élémentaires sont donc les foncteurs de Schur classiques \cite[Chap.~6.1]{FH}. On a ainsi le résultat suivant.
\begin{pr}\label{pr-car0} On suppose que $K$ est de caractéristique nulle.
\begin{enumerate}
\item Les foncteurs de Schur $\mathbb{S}_\lambda$, pour $\lambda$ une partition, forment un ensemble de représentants des classes d'isomorphisme de foncteurs élémentaires.
\item Si $\lambda=(\lambda_1,\dots,\lambda_n)$ est une partition avec $\lambda_n\ne 0$, alors $\mathbb{S}_\lambda(V)\ne 0$ si et seulement si $\dim V\ge n$.
\item Le produit tensoriel de deux foncteurs de Schur se décompose comme une somme directe finie de foncteurs de Schur.
\end{enumerate}
\end{pr}
 
Supposons maintenant que $K$ est de caractéristique $p>0$. Nous relions les foncteurs élémentaires aux foncteurs de Schur d'Akin, Buchsbaum et Weyman \cite{ABW}. 
Pour toute partition $\lambda=(\lambda_1,\dots,\lambda_n)$ de $d$, on note
\[\Lambda^\lambda:=\Lambda^{\lambda_1}\otimes\dots\otimes\Lambda^{\lambda_n}\quad\text{et}\quad S^\lambda:=S^{\lambda_1}\otimes\dots\otimes S^{\lambda_n}.\]
Soit $\widetilde{\lambda}$ la partition conjuguée et soit $\mathbb{S}_\lambda$ l'image de la transformation naturelle 
$$\Lambda^{\widetilde{\lambda}}\hookrightarrow T^d\xrightarrow{\simeq}T^d \twoheadrightarrow S^\lambda $$
où l'isomorphisme est la permutation des facteurs de $T^d$ qui envoie $v_1\otimes\dots\otimes v_d$ sur $v_{\sigma^{-1}(1)}\otimes\dots\otimes v_{\sigma^{-1}(d)}$, où $\sigma(i)$ est le $i$-ème terme dans la lecture en colonne  du tableau de forme $\lambda$ rempli en ligne par les nombres entiers $1$, $2$,\dots, $d$. (Le foncteur de Schur $\mathbb{S}_\lambda$ est noté $L_{\widetilde{\lambda}}$ dans  \cite{ABW}). On rappelle qu'une partition $\lambda=(\lambda_1,\dots,\lambda_n)$ est 
\emph{$p$-restreinte} si $\lambda_n<p$ et $\lambda_i-\lambda_{i+1}<p$ pour $1\le i<n$. 

\begin{pr}\label{prop-L-p} On suppose que $K$ est de caractéristique $p> 0$.
\begin{enumerate}
\item Pour toute partition $p$-restreinte $\lambda$, le foncteur $\mathbb{S}_\lambda$ a un socle simple $\mathbb{L}_\lambda$. Les foncteurs $\mathbb{L}_\lambda$ indexés par les partitions $p$-restreintes $\lambda$ forment une famille de représentants des classes d'isomorphisme de foncteurs élémentaires.
\item Si $\lambda=(\lambda_1,\dots,\lambda_n)$ est une partition $p$-restreinte avec $\lambda_n\ne 0$, alors $\mathbb{L}_\lambda(V)\ne 0$ si et seulement si $\dim V\ge n$.
\item Le produit tensoriel de deux foncteurs élémentaires est de longueur finie dans $\F(K,K)$.
\end{enumerate}
\end{pr}

\begin{proof}
L'énoncé de la proposition est bien connu dans la catégorie $\mathcal{P}_K$ des foncteurs strictement polynomiaux de Friedlander et Suslin \cite{FS}. Plus précisément, la catégorie $\mathcal{P}_K$ peut être décrite en termes d'algèbres de Schur \cite[Thm~3.2]{FS}, c'est donc une catégorie de plus haut poids. Les foncteurs de Schur $\mathbb{S}_\lambda$ sont les objets costandard \cite[Remark~6.1]{TouzeRingel}, ils ont un socle simple $\mathbb{L}_\lambda$, et les $\mathbb{L}_\lambda$ forment un ensemble de représentants des simples de $\mathcal{P}_K$. Un théorème classique de Clausen et James \cite{Clausen,James} (voir aussi \cite[App.~B]{Touze} pour une démonstration fonctorielle) implique que les foncteurs $\mathbb{L}_\lambda$ apparaissant comme sous-foncteurs des $T^d$, ou dualement comme quotients des $T^d$, sont ceux indexés par les partitions $p$-restreintes. Ceci prouve le premier point. Pour le deuxième point, on a $\mathbb{L}_\lambda(V)\subset \mathbb{S}_\lambda(V)=0$ si $\dim V<n$ car $\mathbb{S}_\lambda(V)$ est un quotient de $\Lambda^{\widetilde{\lambda}}(V)$ qui est alors nul. Réciproquement, $\mathbb{L}_\lambda$ est un sous-foncteur de $S^\lambda$, donc par dualité un quotient de $\Gamma^\lambda$. Si $\dim V\ge n$, on a donc $\mathbb{L}_\lambda(V)\ne 0$ d'après \cite[Cor.~2.12]{FS}. Enfin le troisième point est une conséquence de \cite[Thm~3.2]{FS}.

Nous montrons que l'énoncé dans $\mathcal{P}_K$ implique la proposition \ref{prop-L-p}.
Il y a un foncteur d'oubli $\U:\mathcal{P}_K\to \F(K,K)$ fidèle et exact. \emph{Si $K$ est infini}, ce foncteur d'oubli identifie $\mathcal{P}_K$ avec la sous-catégorie pleine de $\F(K,K)$ dont les objets sont les foncteurs qu'on peut écrire comme des sous-quotients de sommes directes finies de foncteurs $T^d$, en particulier l'image de $\U$ est stable par sous-quotients et $\U$ préserve les simples. L'énoncé de la proposition \ref{prop-L-p} découle donc formellement de l'énoncé dans $\mathcal{P}_K$. 
Supposons \emph{$K$ fini}. Un argument classique de polynomialité (cf. \cite[Thm 4.14]{Ku1} ou \cite{Pira88}) montre que $\U$ préserve les foncteurs de longueur finie, d'où le troisième point. La définition des foncteurs élémentaires montre qu'ils admettent une structure de foncteur strictement polynomial et qu'ils sont des sous-foncteurs strictement polynomiaux de $T^d$. Pour démontrer les deux premiers points, il reste donc à vérifier que les $\mathbb{S}_\lambda$ indexés par des partitions $p$-restreintes ont un socle simple dans $\F(K,K)$. Comme $T^d$ est injectif dans $\mathcal{P}_K$, le monomorphisme $\mathbb{L}_\lambda\hookrightarrow T^d$ s'étend en un monomorphisme $\mathbb{S}_\lambda\hookrightarrow T^d$. Le socle de $\mathbb{S}_\lambda$ dans $\F(K,K)$ n'est donc constitué que de sous-foncteurs simples de $T^d$. Pour démontrer que le socle est simple, il suffit donc de montrer que pour toute partition $p$-restreinte $\mu$ de l'entier $d$, les espaces vectoriels $\mathcal{P}_K(\mathbb{L}_\mu,\mathbb{S}_\lambda)$ et  $\F(K,K)(\mathbb{L}_\mu,\mathbb{S}_\lambda)$ ont même dimension. Pour cela on considère le diagramme commutatif dont les lignes sont exactes et dont les injections verticales sont induites par le foncteur d'oubli :
\[\xymatrix{
0\ar[r]& \mathcal{P}_K(\mathbb{L}_\mu,\mathbb{S}_\lambda)\ar[r]\ar@{^{(}->}[d]^-{(1)} & \mathcal{P}_K(\mathbb{L}_\mu,T^d)\ar[r]\ar@{^{(}->}[d]^-{(2)} & \mathcal{P}_K(\mathbb{L}_\mu,T^d/\mathbb{S}_\lambda)\ar@{^{(}->}[d]\\
0\ar[r]& \F(K,K)(\mathbb{L}_\mu,\mathbb{S}_\lambda)\ar[r] &\F(K,K)(\mathbb{L}_\mu,T^d)\ar[r]& \F(K,K)(\mathbb{L}_\mu,T^d/\mathbb{S}_\lambda)
}\]
Le morphisme $(2)$ est un isomorphisme car sa source et son but ont même dimension, égale à la dimension de $cr_d(\mathbb{L}_\mu)(K,\dots,K)$, par la proposition \ref{adj-diagcr}. Le morphisme $(1)$ est donc un isomorphisme, ce qui achève la démonstration.
\end{proof}

Revenons au cas général où $K$ est un corps commutatif quelconque. Si $K$ est de caractéristique nulle, on note $\mathbb{L}_\lambda:= \mathbb{S}_\lambda$, pour toute partition $\lambda$. Si $K$ est de caractéristique $p>0$, les $\mathbb{L}_\lambda$ sont définis, seulement pour $\lambda$ $p$-restreinte, dans la proposition \ref{prop-L-p}. Dans tous les cas, la formule définissant les foncteurs $\mathbb{L}_\lambda$ (ou bien celle définissant les foncteurs élémentaires) montre que les $\mathbb{L}_\lambda(K^n)$ sont des représentations polynomiales du \emph{groupe algébrique} $\GL_{n,K}$ (nous utilisons une notation légèrement différente pour distinguer le groupe algébrique de son groupe des $K$-points, le groupe abstrait $\GL_n(K)$). Les représentations polynomiales (ou plus généralement rationnelles) du groupe algébrique $\GL_{n,K}$ sont classifiées par la théorie des plus hauts poids (voir par exemple \cite[Sections 2.1 et 2.2]{Hum}). Plus précisément, toute représentation rationnelle $V$ de $\GL_{n,K}$ admet une base $(v_1,\dots,v_N)$, telle que l'action du groupe $D_{n,K}$ des matrices diagonales sur un élément de cette base est donnée par une formule du type $\mathrm{diag}(x_1,\dots,x_n)\cdot v_i = x_1^{\lambda_1}\dots x_n^{\lambda_n}v_i$, pour un certain uplet $\lambda=(\lambda_1,\dots,\lambda_n)\in \mathbb{Z}^n$. Les uplets $\lambda$ apparaissant ainsi ne dépendent pas du choix de la base, ce sont les \emph{poids} de la représentation $V$. Si $V$ est simple, il existe un poids $\lambda$ de $V$ maximal pour l'ordre lexicographique, qu'on appelle le \emph{plus haut poids} de $V$, et qui caractérise $V$ à isomorphisme près.  
\begin{pr}\label{pr-B3}
Soit $\lambda$ une partition telle que $\mathbb{L}_\lambda(K^n)\ne 0$. Alors c'est une représentation simple du groupe algébrique $\GL_{n,K}$, de plus haut poids $\lambda$.
\end{pr}
\begin{proof}
La base de l'espace vectoriel $\mathbb{S}_\lambda(K^n)$ \cite[Thm~II.2.16]{ABW} montre que c'est une représentation de plus haut poids $\lambda$. Comme $\mathbb{L}_\lambda\subset S^\lambda$ ou par dualité $L_\lambda$ est un quotient de $\Gamma^\lambda$, \cite[Cor 2.12]{FS} montre que $\lambda$ est un poids de $\mathbb{L}_\lambda(K^n)$. Ainsi $L_\lambda(K^n)$ est une représentation de plus haut poids $\lambda$. Il reste à montrer qu'elle est simple.
Comme $\mathbb{L}_\lambda$ est auto-dual (cf. par exemple \cite[section 2.2]{Touze}), la représentation $\mathbb{L}_\lambda(K^n)$ est auto-duale (pour la dualité utilisant la transposition de $\GL_{n,K}$). Pour montrer que $\mathbb{L}_\lambda(K^n)$ est simple, il suffit donc de montrer que $\mathrm{End}_{\GL_{n,K}}(\mathbb{L}_\lambda(K^n))$ est de dimension (inférieure ou égale à) $1$ sur $K$. Mais cet espace vectoriel est un sous-espace vectoriel $\mathrm{Hom}_{\GL_{n,K}}(\Gamma^\lambda(K^n),\mathbb{L}_\lambda(K^n))$. D'après \cite[Lm~6.6]{TouzeENS} ce dernier est isomorphe à $\mathcal{P}_K(\Gamma^\lambda,\mathbb{L}_\lambda)$, qui est de dimension $1$ sur $K$ d'après \cite[Cor.~2.12]{FS}. La simplicité de $\mathbb{L}_\lambda(K^n)$ en découle.
\end{proof}

Dans l'énoncé suivant, $\det$ désigne la représentation de dimension $1$ sur $K$ de $\M_n(K)$ donnée par le déterminant.

\begin{pr}\label{pr-technique1}
Soit $\lambda=(\lambda_1,\dots,\lambda_n)$ une partition telle que $\lambda_n\ne 0$. Si $\nu=(\lambda_1-1,\dots,\lambda_n-1)$, on a des isomorphismes de $K[\mathcal{M}_n(K)]$-modules
\[\mathbb{L}_\lambda(K^n)\simeq \mathbb{L}_\nu(K^n)\otimes \det \;,\qquad\mathbb{L}_\lambda(K^n)\simeq \mathbb{L}_\lambda(K^n)\otimes \delta\;,\]
où $\delta$ désigne le $K[\M_n(K)]$-module $K$ sur lequel les matrices inversibles (resp. singulières) opèrent par l'identité (resp. par $0$).
\end{pr}
\begin{proof}
Les représentations $\mathbb{L}_\lambda(K^n)$ et $\mathbb{L}_\nu(K^n)\otimes \det$ sont deux représentations simples de même plus haut poids de $\GL_{n,K}$, elles sont donc isomorphes. Par Zariski-densité, elle sont isomorphes comme représentations du monoïde algébrique $\mathcal{M}_n$, donc comme représentations du monoïde des $K$-points $\mathcal{M}_n(K)$, ce qui donne le premier isomorphisme. Pour le deuxième isomorphisme, il suffit de remarquer que $\det\otimes\delta=\det$.
\end{proof}

\section{Un lemme de linéarisation}\label{app-linearis}

Commençons par rappeler le fait suivant. Soient $k$ un anneau et
\begin{equation}\label{eqseap}Z\xrightarrow{q} Y\xrightarrow{p} X\to 0\qquad\text{(resp. }0\to X\xrightarrow{u}Y\xrightarrow{v}Z\text{)}
\end{equation}
une suite exacte de groupes abéliens. Alors la suite de $k$-modules
\begin{equation}\label{eqsers}k[Y\oplus Z]\xrightarrow{\alpha}k[Y]\xrightarrow{k[p]}k[X]\to 0\qquad\text{(resp. }0\to k[X]\xrightarrow{k[u]}k[Y]\xrightarrow{\beta}k[Y\oplus Z]\text{)}
\end{equation}
où $\alpha([(y,z)])=[y+q(z)]-[y]$ (resp. $\beta([y])=[(y,v(y))]-[(y,0)]$) est exacte. Comme elle est naturelle en la suite exacte initiale, cela vaut également lorsque \eqref{eqseap} est une suite exacte de {\em foncteurs} vers les groupes abéliens.

Le lemme suivant, qui est bien connu des experts, n'interviendra, dans le présent article, que pour démontrer le théorème~\ref{th-simples-pf}, mais il présente un intérêt intrinsèque.

\begin{lm}\label{lm-linearis} Soient $k$ un anneau, $A, B : \A\to\mathbf{Ab}$ des foncteurs additifs. Le morphisme canonique de $k$-modules (où le premier ${\rm Hom}$ est pris dans la catégorie $\mathbf{Add}(\A;\mathbb{Z})$ et le deuxième dans la catégorie $\F(\A;k)$)
\[\gamma_{A,B} : k[{\rm Hom}(A,B)]\to {\rm Hom}(k[A],k[B])\]
est injectif. Il est bijectif si $A$ est de type fini.
\end{lm}

\begin{proof} 
Si $A$ est un foncteur représentable $\A(t,-)$, le lemme de Yoneda entraîne que $\gamma_{A,B}$ est un isomorphisme. Supposons maintenant que $A$ est une somme directe $\underset{i\in I}{\bigoplus}\A(t_i,-)$ de tels foncteurs. On a donc ${\rm Hom}(A,B)\simeq\underset{i\in I}{\prod} B(t_i)$. Soit $x\neq 0$ un élément de $k[{\rm Hom}(A,B)]\simeq k[\underset{i\in I}{\prod} B(t_i)] $ : $x$ peut s'écrire comme $\sum_{n=1}^m\lambda_n[\xi_n]$ où $m\in\mathbb{N}^*$, $\lambda_n\in k\setminus\{0\}$ et les $\xi_n$ sont des éléments deux à deux distincts de $\underset{i\in I}{\prod} B(t_i)$. Il existe donc une partie \emph{finie} $J$ de $I$ telle que les projections des $\xi_n$ dans $\underset{i\in J}{\prod} B(t_i)$ soient deux à deux distinctes. Ainsi, l'image de $x$ dans $k[\mathrm{Hom}(A',B)]$, où $A':=\underset{i\in J}{\bigoplus}\A(t_i,-)\simeq\A(\underset{i\in J}{\bigoplus}t_i,-)$, par l'application induite par l'inclusion $A'\hookrightarrow A$ est non nulle. Comme $\gamma_{A',B}$ est injective d'après ce qui précède, la naturalité en $A$ de $\gamma_{A,B}$ implique que l'image de $x$ par $\gamma_{A,B}$ est non nulle : $\gamma_{A,B}$ est donc injective.

On en déduit l'injectivité de $\gamma_{A,B}$ pour \emph{tout} foncteur additif $A$ en écrivant $A$ comme quotient d'une somme directe de foncteurs additifs représentables et en utilisant la préservation des monomorphismes par le foncteur $k[-]$.

Si $A$ est de type fini, on peut trouver une suite exacte $Q\to P\to A\to 0$ de $\mathbf{Add}(\A;\mathbb{Z})$ avec $P$ représentable. Utilisant \eqref{eqsers}, on en déduit un diagramme commutatif aux lignes exactes
\[\xymatrix{0\ar[r] & k[{\rm Hom}(A,B)]\ar[r]\ar[d]_-{\gamma_{A,B}} & k[{\rm Hom}(P,B)]\ar[r]\ar[d]_-{\gamma_{P,B}} & k[{\rm Hom}(P\oplus Q,B)]\ar[d]_-{\gamma_{P\oplus Q,B}} \\
0\ar[r] & {\rm Hom}(k[A],k[B])\ar[r] & {\rm Hom}(k[P],k[B])\ar[r] & {\rm Hom}(k[P\oplus Q],k[B]).
}\]
D'après le début de la démonstration, $\gamma_{P,B}$ est bijectif et $\gamma_{P\oplus Q,B}$ est injectif, ce qui implique que $\gamma_{A,B}$ est bijectif, comme souhaité.
\end{proof}

\bibliographystyle{plain}
\bibliography{bib-simples}

\begin{thebibliography}{10}

\bibitem{ABW}
Kaan Akin, David~A. Buchsbaum, and Jerzy Weyman.
\newblock Schur functors and {S}chur complexes.
\newblock {\em Adv. in Math.}, 44(3):207--278, 1982.

\bibitem{Aus-rec}
Maurice Auslander.
\newblock Functors and morphisms determined by objects.
\newblock pages 1--244. Lecture Notes in Pure Appl. Math., Vol. 37, 1978.

\bibitem{Aus82}
Maurice Auslander.
\newblock A functorial approach to representation theory.
\newblock In {\em Representations of algebras ({P}uebla, 1980)}, volume 944 of
  {\em Lecture Notes in Math.}, pages 105--179. Springer, Berlin-New York,
  1982.

\bibitem{BMS}
H.~Bass, J.~Milnor, and J.-P. Serre.
\newblock Solution of the congruence subgroup problem for {${\rm
  SL}_{n}\,(n\geq 3)$} and {${\rm Sp}_{2n}\,(n\geq 2)$}.
\newblock {\em Inst. Hautes \'{E}tudes Sci. Publ. Math.}, (33):59--137, 1967.

\bibitem{Betley}
Stanislaw Betley.
\newblock Stable {$K$}-theory of finite fields.
\newblock {\em $K$-Theory}, 17(2):103--111, 1999.

\bibitem{BBD}
A.~A. Be\u{\i}linson, J.~Bernstein, and P.~Deligne.
\newblock Faisceaux pervers.
\newblock In {\em Analysis and topology on singular spaces, {I} ({L}uminy,
  1981)}, volume 100 of {\em Ast\'{e}risque}, pages 5--171. Soc. Math. France,
  Paris, 1982.

\bibitem{Borc}
Francis Borceux.
\newblock {\em Handbook of categorical algebra. 1}, volume~50 of {\em
  Encyclopedia of Mathematics and its Applications}.
\newblock Cambridge University Press, Cambridge, 1994.
\newblock Basic category theory.

\bibitem{Borel-Tits}
Armand Borel and Jacques Tits.
\newblock Homomorphismes ``abstraits'' de groupes alg\'{e}briques simples.
\newblock {\em Ann. of Math. (2)}, 97:499--571, 1973.

\bibitem{Bki}
N.~Bourbaki.
\newblock {\em \'El\'ements de math\'ematique. {A}lg\`ebre. {C}hapitre 8.
  {M}odules et anneaux semi-simples}.
\newblock Springer, Berlin, 2012.
\newblock Second revised edition of the 1958 edition.

\bibitem{Bki2}
Nicolas Bourbaki.
\newblock {\em \'{E}l\'{e}ments de math\'{e}matique}.
\newblock Masson, Paris, 1981.
\newblock Alg\`ebre. Chapitres 4 \`a 7. [Algebra. Chapters 4--7].

\bibitem{CEF}
Thomas Church, Jordan~S. Ellenberg, and Benson Farb.
\newblock F{I}-modules and stability for representations of symmetric groups.
\newblock {\em Duke Math. J.}, 164(9):1833--1910, 2015.

\bibitem{Clausen}
Michael Clausen.
\newblock Letter place algebras and a characteristic-free approach to the
  representation theory of the general linear and symmetric groups. {II}.
\newblock {\em Adv. in Math.}, 38(2):152--177, 1980.

\bibitem{CuR}
Charles~W. Curtis and Irving Reiner.
\newblock {\em Methods of representation theory. {V}ol. {I}}.
\newblock John Wiley \& Sons, Inc., New York, 1981.
\newblock With applications to finite groups and orders, Pure and Applied
  Mathematics, A Wiley-Interscience Publication.

\bibitem{Dja-gr}
Aur\'{e}lien Djament.
\newblock Foncteurs en grassmanniennes, filtration de {K}rull et cohomologie
  des foncteurs.
\newblock {\em M\'{e}m. Soc. Math. Fr. (N.S.)}, (111):xxii+213 pp. (2008),
  2007.

\bibitem{DjaR}
Aur\'{e}lien Djament.
\newblock Sur l'homologie des groupes unitaires \`a coefficients polynomiaux.
\newblock {\em J. K-Theory}, 10(1):87--139, 2012.

\bibitem{Dja-FM}
Aur\'elien Djament.
\newblock Des propri\'et\'es de finitude des foncteurs polynomiaux.
\newblock {\em Fund. Math.}, 233(3):197--256, 2016.

\bibitem{DV}
Aur{\'e}lien Djament and Christine Vespa.
\newblock Sur l'homologie des groupes orthogonaux et symplectiques \`a
  coefficients tordus.
\newblock {\em Ann. Sci. \'Ec. Norm. Sup\'er. (4)}, 43(3):395--459, 2010.

\bibitem{E-add}
Samuel Eilenberg.
\newblock Abstract description of some basic functors.
\newblock {\em J. Indian Math. Soc. (N.S.)}, 24:231--234 (1961), 1960.

\bibitem{EML}
Samuel Eilenberg and Saunders MacLane.
\newblock On the groups {$H(\Pi,n)$}. {II}. {M}ethods of computation.
\newblock {\em Ann. of Math. (2)}, 60:49--139, 1954.

\bibitem{FarbICM}
Benson Farb.
\newblock Representation stability.
\newblock In {\em Proceedings of the {I}nternational {C}ongress of
  {M}athematicians---{S}eoul 2014. {V}ol. {II}}, pages 1173--1196. Kyung Moon
  Sa, Seoul, 2014.

\bibitem{Form}
Edward Formanek.
\newblock Group rings of free products are primitive.
\newblock {\em J. Algebra}, 26:508--511, 1973.

\bibitem{FFSS}
Vincent Franjou, Eric~M. Friedlander, Alexander Scorichenko, and Andrei Suslin.
\newblock General linear and functor cohomology over finite fields.
\newblock {\em Ann. of Math. (2)}, 150(2):663--728, 1999.

\bibitem{FS}
Eric~M. Friedlander and Andrei Suslin.
\newblock Cohomology of finite group schemes over a field.
\newblock {\em Invent. Math.}, 127(2):209--270, 1997.

\bibitem{FH}
William Fulton and Joe Harris.
\newblock {\em Representation theory}, volume 129 of {\em Graduate Texts in
  Mathematics}.
\newblock Springer-Verlag, New York, 1991.
\newblock A first course, Readings in Mathematics.

\bibitem{Gabriel}
Pierre Gabriel.
\newblock Des cat\'{e}gories ab\'{e}liennes.
\newblock {\em Bull. Soc. Math. France}, 90:323--448, 1962.

\bibitem{Gab75}
Pierre Gabriel.
\newblock Repr\'{e}sentations ind\'{e}composables.
\newblock pages 143--169. Lecture Notes in Math., Vol. 431, 1975.

\bibitem{Green}
James~A. Green.
\newblock {\em Polynomial representations of {${\rm GL}_{n}$}}, volume 830 of
  {\em Lecture Notes in Mathematics}.
\newblock Springer-Verlag, Berlin-New York, 1980.

\bibitem{Harman}
Nate Harman.
\newblock Effective and infinite-rank superrigidity in the context of
  representation stability.
\newblock arXiv:1902.05603.

\bibitem{HLS}
Hans-Werner Henn, Jean Lannes, and Lionel Schwartz.
\newblock The categories of unstable modules and unstable algebras over the
  {S}teenrod algebra modulo nilpotent objects.
\newblock {\em Amer. J. Math.}, 115(5):1053--1106, 1993.

\bibitem{Hum}
James~E. Humphreys.
\newblock {\em Modular representations of finite groups of {L}ie type}, volume
  326 of {\em London Mathematical Society Lecture Note Series}.
\newblock Cambridge University Press, Cambridge, 2006.

\bibitem{Ja-sym}
G.~D. James.
\newblock {\em The representation theory of the symmetric groups}, volume 682
  of {\em Lecture Notes in Mathematics}.
\newblock Springer, Berlin, 1978.

\bibitem{James}
Gordon~D. James.
\newblock The decomposition of tensors over fields of prime characteristic.
\newblock {\em Math. Z.}, 172(2):161--178, 1980.

\bibitem{Kra}
Henning Krause.
\newblock The spectrum of a module category.
\newblock {\em Mem. Amer. Math. Soc.}, 149(707):x+125, 2001.

\bibitem{Ku1}
Nicholas~J. Kuhn.
\newblock Generic representations of the finite general linear groups and the
  {S}teenrod algebra. {I}.
\newblock {\em Amer. J. Math.}, 116(2):327--360, 1994.

\bibitem{Ku2}
Nicholas~J. Kuhn.
\newblock Generic representations of the finite general linear groups and the
  {S}teenrod algebra. {II}.
\newblock {\em $K$-Theory}, 8(4):395--428, 1994.

\bibitem{Ku3}
Nicholas~J. Kuhn.
\newblock Generic representations of the finite general linear groups and the
  {S}teenrod algebra. {III}.
\newblock {\em $K$-Theory}, 9(3):273--303, 1995.

\bibitem{Ku-strat}
Nicholas~J. Kuhn.
\newblock A stratification of generic representation theory and generalized
  {S}chur algebras.
\newblock {\em $K$-Theory}, 26(1):15--49, 2002.

\bibitem{Ku-adv}
Nicholas~J. Kuhn.
\newblock Generic representation theory of finite fields in nondescribing
  characteristic.
\newblock {\em Adv. Math.}, 272:598--610, 2015.

\bibitem{LRGS}
Wendy Lowen, Julia Ramos~Gonz\'{a}lez, and Boris Shoikhet.
\newblock On the tensor product of linear sites and {G}rothendieck categories.
\newblock {\em Int. Math. Res. Not. IMRN}, (21):6698--6736, 2018.

\bibitem{Marg}
G.~A. Margulis.
\newblock {\em Discrete subgroups of semisimple {L}ie groups}, volume~17 of
  {\em Ergebnisse der Mathematik und ihrer Grenzgebiete (3) [Results in
  Mathematics and Related Areas (3)]}.
\newblock Springer-Verlag, Berlin, 1991.

\bibitem{Mi72}
Barry Mitchell.
\newblock Rings with several objects.
\newblock {\em Advances in Math.}, 8:1--161, 1972.

\bibitem{Nag}
Rohit Nagpal.
\newblock V{I}-modules in nondescribing characteristic, part {I}.
\newblock {\em Algebra Number Theory}, 13(9):2151--2189, 2019.

\bibitem{Passi}
Inder Bir~S. Passi.
\newblock {\em Group rings and their augmentation ideals}, volume 715 of {\em
  Lecture Notes in Mathematics}.
\newblock Springer, Berlin, 1979.

\bibitem{Pira88}
T.~I. Pirashvili.
\newblock Polynomial functors.
\newblock {\em Trudy Tbiliss. Mat. Inst. Razmadze Akad. Nauk Gruzin. SSR},
  91:55--66, 1988.

\bibitem{PiraWald}
Teimuraz Pirashvili and Friedhelm Waldhausen.
\newblock Mac {L}ane homology and topological {H}ochschild homology.
\newblock {\em J. Pure Appl. Algebra}, 82(1):81--98, 1992.

\bibitem{Pop}
N.~Popescu.
\newblock {\em Abelian categories with applications to rings and modules}.
\newblock Academic Press, London-New York, 1973.
\newblock London Mathematical Society Monographs, No. 3.

\bibitem{GP-cow}
Geoffrey M.~L. Powell.
\newblock The structure of indecomposable injectives in generic representation
  theory.
\newblock {\em Trans. Amer. Math. Soc.}, 350(10):4167--4193, 1998.

\bibitem{PSam}
Andrew Putman and Steven~V. Sam.
\newblock Representation stability and finite linear groups.
\newblock {\em Duke Math. J.}, 166(13):2521--2598, 2017.

\bibitem{SamSn}
Steven~V. Sam and Andrew Snowden.
\newblock Gr\"obner methods for representations of combinatorial categories.
\newblock {\em J. Amer. Math. Soc.}, 30(1):159--203, 2017.

\bibitem{Serre}
Jean-Pierre Serre.
\newblock Le probl\`eme des groupes de congruence pour ${SL}_2$.
\newblock {\em Ann. of Math. (2)}, 92:489--527, 1970.

\bibitem{Silv-det_bloc}
John~R. Silvester.
\newblock Determinants of block matrices.
\newblock {\em The Mathematical Gazette}, 84(501):460--467, 2000.

\bibitem{Sni83}
Robert~L. Snider.
\newblock Group rings with finite endomorphism dimension.
\newblock {\em Arch. Math. (Basel)}, 41(3):219--225, 1983.

\bibitem{Spri}
T.~A. Springer.
\newblock {\em Linear algebraic groups}, volume~9 of {\em Progress in
  Mathematics}.
\newblock Birkh\"{a}user Boston, Inc., Boston, MA, second edition, 1998.

\bibitem{BSt}
Benjamin Steinberg.
\newblock {\em Representation theory of finite monoids}.
\newblock Universitext. Springer, Cham, 2016.

\bibitem{St-TPT}
Robert Steinberg.
\newblock Representations of algebraic groups.
\newblock {\em Nagoya Math. J.}, 22:33--56, 1963.

\bibitem{RSt}
Robert Steinberg.
\newblock Some consequences of the elementary relations in {${\rm SL}_n$}.
\newblock In {\em Finite groups---coming of age ({M}ontreal, {Q}ue., 1982)},
  volume~45 of {\em Contemp. Math.}, pages 335--350. Amer. Math. Soc.,
  Providence, RI, 1985.

\bibitem{Street}
Ross Street.
\newblock Ideals, radicals, and structure of additive categories.
\newblock {\em Appl. Categ. Structures}, 3(2):139--149, 1995.

\bibitem{TouzeENS}
Antoine Touz\'{e}.
\newblock Troesch complexes and extensions of strict polynomial functors.
\newblock {\em Ann. Sci. \'{E}c. Norm. Sup\'{e}r. (4)}, 45(1):53--99, 2012.

\bibitem{TouzeRingel}
Antoine Touz\'{e}.
\newblock Ringel duality and derivatives of non-additive functors.
\newblock {\em J. Pure Appl. Algebra}, 217(9):1642--1673, 2013.

\bibitem{Touze}
Antoine Touz\'{e}.
\newblock Connectedness of cup products for polynomial representations of
  {${\rm GL}_n$} and applications.
\newblock {\em Ann. K-Theory}, 3(2):287--329, 2018.

\bibitem{VS-survol}
N.~A. Vavilov and A.~V. Stepanov.
\newblock Linear groups over general rings. {I}. {G}eneralities.
\newblock {\em Zap. Nauchn. Sem. S.-Peterburg. Otdel. Mat. Inst. Steklov.
  (POMI)}, 394(Voprosy Teorii Predstavleni\u{\i} Algebr i Grupp. 22):33--139,
  295, 2011.

\bibitem{V2008}
Christine Vespa.
\newblock Generic representations of orthogonal groups: the functor category
  {${\mathcal{F}}_{\mathrm{quad}}$}.
\newblock {\em J. Pure Appl. Algebra}, 212(6):1472--1499, 2008.

\bibitem{W-add}
Charles~E. Watts.
\newblock Intrinsic characterizations of some additive functors.
\newblock {\em Proc. Amer. Math. Soc.}, 11:5--8, 1960.

\bibitem{W83}
B.~A.~F. Wehrfritz.
\newblock Group rings with finite central endomorphism dimension.
\newblock {\em Glasgow Math. J.}, 24(2):169--176, 1983.

\bibitem{Weibel}
Charles~A. Weibel.
\newblock {\em The {$K$}-book}, volume 145 of {\em Graduate Studies in
  Mathematics}.
\newblock American Mathematical Society, Providence, RI, 2013.
\newblock An introduction to algebraic $K$-theory.

\end{thebibliography}
\end{document}